\documentclass[12pt]{amsart}

\usepackage{latexsym, amssymb, amscd, amsfonts}
\usepackage[all,cmtip]{xy}
\usepackage{tikz-cd}

\usepackage{bm}
\usepackage{multirow}
\usepackage{nicematrix}
\usepackage[colorlinks=true,urlcolor=black,citecolor=black,linkcolor=black,%
pdftitle={On Generalized Root Systems},%
pdfauthor={Ivan Dimitrov and Rita Fioresi},%
pdfsubject={Representation theory},%
pdfkeywords={Parabolic subalgebras, Kostant root systems, Positive roots}]{hyperref}
\usepackage{ifthen}
\usepackage{verbatim}  
\usepackage{hyperref}
\usepackage{longtable}

\DeclareMathAlphabet{\mathpzc}{OT1}{pzc}{m}{it}

\newboolean{bourbaki}
\setboolean{bourbaki}{false}

\newcommand{\np}{\medskip\noindent}

\newcounter{eqcounter}[section]
\renewcommand{\theeqcounter}{\arabic{section}.\arabic{eqcounter}}
\renewenvironment{equation}
{\medskip\noindent\refstepcounter{eqcounter}\makebox[0pt][l]{({\bf\theeqcounter})}
\begin{minipage}[b]{\textwidth}$$}{$$\end{minipage}\medskip\noindent}

\renewcommand{\geq}{\geqslant}
\renewcommand{\leq}{\leqslant}

\newcommand{\Z}{\mathrm{Z}}         

\newcommand{\supp}{\operatorname{supp}}

\newcommand{\cA}{\mathcal{A}}
\newcommand{\cB}{\mathcal{B}}
\newcommand{\cC}{\mathcal{C}}
\newcommand{\cD}{\mathcal{D}}
\newcommand{\cE}{\mathcal{E}}
\newcommand{\cF}{{\mathcal F}} 
\newcommand{\cG}{{\mathcal G}}

\newcommand{\cX}{{\mathcal X}}

\def\vep{\varepsilon}

\def\gc{\mathfrak{c}}

\def\gg{\mathfrak{g}}

\def\gh{\mathfrak{h}}

\def\gk{\mathfrak{k}}
\def\gl{\mathfrak{l}}
\def\gm{\mathfrak{m}}

\def\go{\mathfrak{o}}
\def\gp{\mathfrak{p}}

\def\gs{\mathfrak{s}}
\def\gt{\mathfrak{t}}

\ifthenelse{\boolean{bourbaki}}
{
\newcommand{\CC}{\mathbf{C}} 
\newcommand{\RR}{\mathbf{R}} 
\newcommand{\ZZ}{\mathbf{Z}} 
}
{
\newcommand{\CC}{\mathbb{C}} 
\newcommand{\RR}{\mathbb{R}} 
\newcommand{\ZZ}{\mathbb{Z}} 
}

\newcommand{\point}{\refstepcounter{subsection}\noindent{\bf \thesubsection.} }
\newcommand{\ppoint}{\refstepcounter{subsubsection}\noindent{\bf \thesubsubsection.} }

\newtheorem{thm}{Theorem}[section]
\newtheorem{theorem}[thm]{Theorem}

\newtheorem{corollary}[thm]{Corollary}

\newtheorem{conjecture}[thm]{Conjecture}
\newtheorem{proposition}[thm]{Proposition}

\theoremstyle{definition}
\newtheorem{definition}[thm]{Definition}
\newtheorem{example}[thm]{Example}

\newtheorem{remark}[thm]{Remark}

\newcommand{\htt}{\mathrm{ht}}
\newcommand{\Span}{\mathrm{span\,}}

\usepackage{color}

\newcommand{\al}{\alpha}
\newcommand{\be}{\beta}

\newcommand{\bebar}{{\overline{\beta}}}

\newcommand{\gabar}{\overline{\gamma}}
\newcommand{\nubar}{\overline{\nu}}
\newcommand{\R}{\mathbb R}
\newcommand{\beq}{\begin{equation}}
\newcommand{\eeq}{\end{equation}}
\newcommand{\lra}{\longrightarrow}
\newcommand{\RIperp}{R_I^\perp}
\newcommand{\id}{{\mathrm{id}}}

\begin{document}
\pagestyle{plain} 
\title{{ \large{Generalized Root Systems}}}
\author{Ivan Dimitrov}
\address{Ivan Dimitrov: Department of Mathematics and Statistics, 
Queen's University, Kingston,
Ontario,  K7L 3N6, Canada} 
\email{dimitrov@queensu.ca} 
\thanks{I.\ Dimitrov was partially supported by an NSERC Discovery grant.}
\author{Rita Fioresi}
\address{Rita Fioresi: FaBiT, 
Universit\'a di Bologna, via S. Donato 15, 40127 Bologna, Italy}
\email{rita.fioresi@unibo.it} 
\thanks{R. Fioresi was partially supported by GHAIA H2020-MSCA-RISE-2017 777822,
PNRR (PE12) MNESYS, INDAM-GNSAGA, CaLISTA CA 21109, and CaLIGOLA MSCA-2021-SE-
01-101086123.}


\begin{abstract} We generalize the notion of a root system by relaxing the
conditions that ensure that it is invariant  under reflections and study the 
resulting structures, which we call generalized root systems (GRSs for short).
Since both Kostant root systems and root systems of Lie superalgebras are 
examples of GRSs, studying GRSs provides a uniform axiomatic approach to
studying both of them. GRSs inherit many of the properties of root systems.
In particular, every GRS defines a crystallographic hyperplane arrangement.
We believe that GRSs provide an intrinsic counterpart to finite Weyl groupoids 
and crystallographic hyperplane arrangements, extending the relationship between
finite Weyl groupoids and crystallographic hyperplane arrangements established by 
Cuntz. An important difference between GRSs and root systems is that GRSs may
lack a (large enough) Weyl group.
In order to compensate for this, we introduce the notion of a virtual reflection,
building on a construction of Penkov and Serganova in the context of root systems of 
Lie superalgebras. 

\np 
The most significant new feature 
of GRSs is that, along with subsystems, one can define quotient GRSs. Both 
Kostant root systems and root systems of Lie superalgebras are equivalent to 
quotients of root systems and all root systems are isomorphic to 
quotients of simply-laced root systems. We classify all rank 2 GRSs and show that
they are equivalent to quotients of root systems. Finally, we discuss in detail
quotients of root systems. In particular we provide all isomorphisms and equivalences among them.
Our results on quotient of root systems provide a different
point of view on flag manifolds, reproving results of Alekseevsky and Graev.

\np
\end{abstract}

\maketitle

\tableofcontents

\section*{Introduction}

\np
The idea of a root system can be traced back to Killing 
who used this notion towards classifying simple complex Lie algebras,
a project completed in full generality by Cartan.
Since then root systems and various analogs and generalizations 
have become an indispensable tool in Lie theory and beyond, e.g.
roots of Kac-Moody and Borcherds algebras, \cite{Ka2}, roots of Lie superalgebras,
\cite{Ka1}, $T$-roots (also known as restricted roots or Kostant roots 
\footnote{The different terms used for these objects reflect the fact that they appear in different contexts. We 
choose to refer to them as ``Kostant roots'' mainly because of our background in Lie theory and the work \cite{Ko}.}), \cite{Ko}, etc.
The versatility of  root systems comes from 
their elegantly simple definition which yields a remarkably rich 
algebraic and combinatorial structure. Two fundamental concepts related to a root system $\Delta$
are its Weyl group and the hyperplane arrangement whose complement consists of the Weyl chambers of $\Delta$.
The fact that the Weyl group acts simply transitively on the Weyl chambers is crucial 
in studying the structure and representations of the (complex) Lie algebra associated with $\Delta$.
The generalizations of root systems mentioned above still 
admit analogs of the Weyl group and the
corresponding hyperplane arrangements. 
The interplay among root systems, Weyl groups, and hyperplane arrangements 
allows natural generalizations of one of these structures to be translated to
the others.

\np 
In the last decade generalizations of 
Weyl groups and Weyl chambers have become the focus of renewed interest
in connection with studying Nichols algebras, see \cite{CL}, 
and 3-fold flopping contractions in the context of the Minimal model program, see \cite{IW}.
Below we provide a very short description of the structures discussed in \cite{CL} and \cite{IW}
insofar as they relate to the present work. 
For details on these connections we refer the reader to the works cited above and the references
therein. In \cite{Cu} M. Cuntz introduces the notion of crystallographic arrangement as a special 
type of simplicial arrangement in $\RR^n$ and establishes a relationship between crystallographic
arrangements and finite Weyl groupoids. Since finite Weyl groupoids were classified 
by M. Cuntz and I. Heckenberger, see \cite{CH},
\cite{Cu} and \cite{CH} together provide a classification of crystallographic arrangements 
(at least up to equivalence). In \cite{IW} O. Iyama and M. Wemyss study hyperplane arrangements 
that correspond to Kostant root systems (which they refer to as ``intersection arrangements'').

\np 
Given a simplicial arrangement $\cA$, a set $R$ of vectors orthogonal to the hyperplanes in $\cA$
(two mutually opposite vectors for each hyperplane), is usually referred to as a root system associated to $\cA$. 
It is natural to require that the choice of normal vectors to hyperplanes in $\cA$ is done in some
consistent way, so that $R$ can be useful in studying $\cA$. The definition of crystallographic 
arrangement is such a condition on $R$. While the relationship between crystallographic arrangements
and finite Weyl groupoids is beautifully described in \cite{Cu}, the relationship between 
the root system $R$ and the crystallographic arrangement $\cA$ is somewhat obscure. Indeed,
by definition, a crystallographic arrangement is a simplicial arrangement which admits a root 
system with a certain property but the corresponding root systems can only be recovered from
the classification of simplicial arrangements. Moreover, the simplicial arrangements arising
from Kostant root systems are crystallographic but Kostant root systems themselves are
not root systems of these arrangements in the sense of \cite{Cu}. This is due to the fact 
that, for Konstant root systems, multiples of roots can be roots, e.g. both $\alpha$ and $2 \alpha$
can be Kostant roots. Starting with a Kostant root system $R$ and considering its corresponding 
crystallographic arrangement $\cA$, the primitive roots in $R$, i.e., the roots whose fractions are
not roots, form a root system of $\cA$ in the sense of \cite{Cu}.
However, removing non-primitive roots from $R$ results in a 
loss of information about the Kostant root system; as pointed out by N. Nabijou and M. Wemyss
in \cite{NW}, root which are multiples of other roots carry important information in applications
to curve counting.

\np
Apart from their applications to intersection arrangements, 
Kostant root systems have played a central role in Lie theory, e.g. 
in the study of flag varieties, cf. \cite{A} and \cite{Gr}. 
Another extension of the notion of root system is provided 
by the root systems of Lie superalgebras which play for Lie superalgebras
the same role that root systems play for Lie algebras, see \cite{Ka1} for details.
(For an axiomatic treatment of root systems of Lie superalgebras, see
\cite{Se1}.)
These finite analogs of root systems exhibit many properties 
of root systems. However, they both lack enough reflections, which
results in some crucial new phenomena, for example, two different bases need 
not be connected by a chain of reflections. As a result, such a root 
system may have different Dynkin diagrams with respect to different bases
and the question of
determining which Dynkin diagrams are associated with isomorphic root systems may be challenging.

\np 
The goal of this work is to provide an axiomatic framework which
allows for a uniform treatment of roots systems of simple Lie algebras,
root systems of classical Lie superalgebras, Kostant roots systems,
and root systems which give rise to crystallographic arrangements.
More precisely, we start by defining a new structure, which we call
a generalized root system (GRS for short) and show that 
many of the properties of root systems like existence of bases, highest 
roots, etc. carry through to GRSs. In particular, GRSs give rise to crystallographic 
arrangements. 
We also introduce the notion of 
a virtual reflection and show that any two bases of a GRS are linked by a chain
of virtual reflections. In the language of simplicial arrangements,
virtual reflections are the same as wall-crossings. 
The most remarkable property of GRSs is that, unlike,
root systems, GRSs allow for taking quotients. Indeed, a Kostant root system
is nothing but a quotient of a root system (and thus a GRS). 
Moreover, while root systems of Lie superalgebra are not naturally
GRSs, it turns out that they are equivalent to quotients of roots systems
(and hence equivalent to GRSs). Equivalence between GRSs is a relation 
which is weaker than isomorphism but is sufficient for
studying and describing most combinatorial properties.
Here is a brief overview of the ideas and results of the paper.

\np 
A GRS is a finite subset $R$ of a Euclidean vector
space $V$ such that, for $\alpha, \beta \in R$, the angle
between $\alpha$ and $\beta$ forces $\alpha+\beta$ or $\alpha - \beta$
to belong to $R$ as well. Namely,
\[
\begin{array}{l}
  \langle \alpha, \beta \rangle <0 \quad {\text{implies}} \quad \alpha + \beta \in R \ ;   \\
\langle \alpha, \beta \rangle > 0 \quad {\text{implies}} \quad \alpha - \beta \in R \ ; \\
\langle \alpha, \beta \rangle = 0 \quad {\text{implies}} \quad \alpha + \beta \in R 
{\text{\,\,  if and only if \,\,}}\alpha - \beta \in R \ .
\end{array}\]
We also require that $R$ spans $V$ to avoid some technical issues, see Definition \ref{def2.5}.
The dimension of $V$ is then the rank of $R$.
Two GRSs $R'$ and $R''$ are isomorphic if there is a conformal linear map
between the corresponding Euclidean spaces $V'$ and $V''$ which is a bijection 
between $R'$ and $R''$. For many combinatorial invariants of GRSs the notion of 
isomorphism is too rigid and the weaker notion of equivalence suffices: 
$R'$ and $R''$ are equivalent if there is a vector space isomorphism
between $V'$ and $V''$ which is a bijection 
between $R'$ and $R''$. 

\np 
Most attributes of root systems carry over to GRSs: the notions of
bases, positive systems, strings of roots, etc. can be defined exactly as they
are defined for root systems, 
see Section \ref{reductive}. Moreover, most properties of root systems
carry over to GRSs: every GRSs admits a base, bases are in a bijection with positive 
systems, every positive systems, considered as an ordered set, is a lattice, etc.

\np 
One important feature of root systems that does not carry over to GRSs is
invariance with respect to reflections. This distinction leads us to
introduce the notion of virtual reflection, see Definition \ref{def3.15}. 
Namely, for a (primitive) root $\alpha \in R$, the virtual reflection
$\sigma_\alpha$ along $\alpha$ is the permutation of $R$ that
reverses all $\alpha$-strings in $R$. Clearly, $\sigma_\alpha(\alpha) = - \alpha$
and $\sigma_\alpha^2 = \id_R$. Note, however, that a virtual reflection 
does not necessarily extend to a linear transformation of $V$. 
Nevertheless, virtual reflections suffice to connect all bases of a GRS, i.e.,
any two bases of $R$ are linked by a chain of virtual reflections. 

\np 
The main new feature of GRSs is that they admit quotients.
Given a base $S$ of a GRS $R$ and a subset $I \subset S$, we define 
the quotient GRS $R/I$ as the image of $R$ under the projection $\pi_I : V \twoheadrightarrow (\Span I)^\perp$.
Remarkably, $R/I$ is a GRS in $V/I := (\Span I)^\perp$, see Section \ref{sec4.05}
and Theorem \ref{the5.25} for details on the setup and the proof of this  fact. 
Moreover, taking quotients of root systems is functorial and the fibres of $\pi_I$ 
possesses very nice properties, see Sections \ref{sec5.35} and \ref{sec5.45}. 
In particular Corollary \ref{cor5.37} (along with Theorem \ref{prop6.77}) 
reproves central results in \cite{Ko} and \cite{DF} in a more conceptual way.

\np 
Since Kostant root systems are, by definition, quotients of root systems and,
by Theorem \ref{prop6.77}, root systems of Lie superalgebras are equivalent 
to quotients of root systems, we study quotients of root systems in detail. 
The main theorem, Theorem \ref{the6.375}, in this section realizes explicitly the virtual 
reflections on quotients of root systems in terms of the corresponding Dynkin diagrams.
Theorem \ref{the6.375} is then used to discuss the isomorphism and equivalence classes
of quotients of root systems. In particular, we are able to recover the tables listed
in \cite{Gr}. Finally, Theorem \ref{prop6.61} shows that every root system is
isomorphic to a quotient of a simply-laced root system. Thus, when discussing properties
of quotients of root systems, it is sufficient to consider simply-laced root systems. 
This result parallels the well-known fact that every root system is the folding of a
simply-laced root system. An important question related to Kostant root systems is the description of 
the corresponding groupoids. This question was the subject of works by Howlett, \cite{Ho} and Brink and
Howlett, \cite{BH}. Eventhough we do not discuss the groupoids related to GRSs, some results in the present paper to Kostant root systems are closely related to results from \cite{Ho} and \cite{BH}.

\np 
As a first step towards classifying GRSs, we classify (up to equivalence and up to
isomorphism) irreducible GRSs of rank 2, see Section \ref{sec7.2}.
We organize the equivalence classes of irreducible GRSs of rank 2, Theorem \ref{the7.25}, 
according to the largest multiplier of a root, i.e., the largest $k \in \ZZ$ such
that there is a nonzero root $\alpha$ for which $k \alpha$ also is a root.
It turns out that, up to equivalence, 
there are 3 GRSs with largest multiplier 1 (reduced GRSs), 8 GRSs 
with largest multiplier 2, 4 with largest multiplier 3, and 1 GRS with largest multiplier 4.
For each equivalence class of GRSs we determine the isomorphism classes that it contains. 
It is remarkable that every GRS of rank 2 is equivalent to a quotient of a root system 
as Theorem \ref{the7.455} shows. 
In the language of hyperplane arrangements, O. Iyama and M. Wemyss 
have computed the rank 2 quotients of root systems, see \cite[Theorem 0.5]{IW}.
A detailed discussion of quotients of root systems (with a special attention on rank 2 quotients)
is carried out in \cite{Gr}.

\np 
The paper is organized as follows: In Section \ref{reductive} we introduce the main definitions and 
establish the elementary properties of GRSs which closely parallel the properties of root systems.
We also discuss in detail, including their Cartan matrices, the reduced irreducible GRSs of rank 2.
 In Section \ref{sec3} we define and study the properties of virtual reflections. 
In Section \ref{sec40} we define and study the properties of quotients of GRSs.
In Section \ref{Sec55} we study the properties of quotients of root systems. In particular,
we describe explicitly the virtual reflections in a quotient of a root system $X_l$ in terms 
of the Dynkin diagram $X_l$, see Theorem \ref{the6.375}. This result appears as the ``wall crossing rule''
in \cite[Theorem 0.1]{IW}.
We then introduce the graph $\cX_{l,k}$ whose connected components
represent the isomorphism classes of rank $k$ quotients of the root system $X_l$ provided by Theorem \ref{the6.375}.
We also establish that the root systems of Lie superalgebras are equivalent to quotients of root systems, Theorem
\ref{prop6.77}, and that every root system is isomorphic to a quotient of a simply-laced root system, Theorem 
\ref{prop6.61}. In Section \ref{sec7.2} we classify (up to equivalence and up to isomorphism) the irreducible GRSs
of rank 2 and prove that every such GRS is equivalent to a quotient of a root system. 
The proofs of these results are somewhat lengthy and repetitive, so in this section we provide an outline of the proof. 
The complete details are then presented in an appendix. These results along with numerous illustrations appeared also 
in \cite{Sp}.
The final Section \ref{sec8}
is a reference to the quotients of root systems. It contains a description of the quotients of classical root systems
and a table of the quotients of exceptional root systems. In higher ranks 
this table coincides with a table in \cite{CH}. The two tables differ in lower ranks due to the 
existence of roots with multiples which are also roots. The similarity between the these tables
is an evidence for Conjecture \ref{conj} which states that every irreducible GRS of rank at least 2 
is equivalent to a quotient of a 
root system. A positive solution of Conjecture \ref{conj} will be a major step towards
classifying all GRSs. Indeed, it would provide a classification up to equivalence and would likely
yield a classification up to isomorphism too. 
\footnote{Following the posting of a preliminary version of this text, M. Cuntz and B. M\"uhlherr have proved Conjecture \ref{conj}, \cite{CM}.}
The Appendix  contains the proofs omitted in Section \ref{sec7.2}.

\np 
{\bf Acknowledgements.} We thank Cole Gigliotti, Etan Ossip, Charles Paquette, and David Wehlau 
for numerous helpful discussions and ideas for improving this work. 
We are grateful to Michael Cuntz, Simon Lentner, Bernhard M\"uhlherr, and Michael Wemyss for their insightful comments
and for bringing to our attention connections of GRSs to hyperplane arrangements and their applications 
to Nichols algebras and the Minimal model program.
We acknowledge Francesco Faglioni's
support and help with computer codes that generated useful examples and drawings.
R. F. thanks Queen's University for the hospitality during three short visits.
Finally, we thank the referee for their insightful comments and suggestions.

\np
\section{Definitions and elementary properties} \label{reductive}

\np
\point {\bf Definitions.} \label{sec2.21}
We first define generalized root systems and introduce the necessary notation 
and terminology.

\np
\begin{definition} \label{def2.5}
A \textit{ generalized root system} (GRS for short) is a pair $(R, V)$, where 
$V$ is a finite dimensional Euclidean space and
$R \subset V$ is a non-empty finite set satisfying the following properties: 

\begin{enumerate}
\item[(i)] $V=\Span R$;
\item[(ii)] If $\alpha, \beta \in R$, then 

\begin{equation} \label{eq1.17}
\left\{
\begin{array}{ccl} \langle\al,\be\rangle < 0 & {\text{implies}} &\al+\be \in R\\
 \langle\al,\be\rangle > 0 & {\text{implies}} &\al-\be \in R\\
 \langle\al,\be\rangle = 0 & {\text{implies}} &\al+\be \in R 
 {\text{ if and only if }} \al - \be \in R \ .\end{array} \right.
 \end{equation}
\end{enumerate}
We call the elements of $R$ \textit{roots}. The {\it rank} of $(R,V)$ is by 
definition the dimension of $V$. The GRS $(0,0)$ is the {\it trivial GRS}.
\end{definition}

\np 
When the space $V$ is clear from the context, we will refer to the GRS $(R,V)$ by the set $R$ only.

\np 
If $R$ is a GRS, then $0 \in R$ and $R = - R$. Indeed, if $0 \neq \al \in R$, 
then $ \langle\al,\al\rangle > 0$ implies that $0 = \al - \al \in R$. 
Moreover, $0, \al \in R$ and $\langle0,\al\rangle = 0$ implies that $- \al \in R$
since $\al \in R$.

\np 
It is easy to see that the GRSs of rank 1 are the sets 
$\{0, \pm \al, \pm 2 \al, \ldots, \pm k \al\}$. 

\np
A nonzero root $\alpha \in R$ is {\it primitive} if $\alpha  = k \alpha'$ with $\alpha'\in R$ and
$k \in \ZZ_{>0}$ implies $k = 1$. If $\alpha \in R$ is primitive and 
$\alpha, \ldots, k \alpha \in R$ but
$(k+1) \alpha \not \in R$, we call $k$ the {\it multiplier of $\alpha$}. The GRS 
$R$ is {\it reduced} if every nonzero vector in $R$ is primitive or, equivalently, 
if, for every $0 \neq \alpha \in R$, $k \alpha \in R$ with $k \in \RR$ implies $k = \pm 1$.

\np
If $(R_1,V_1)$ and $(R_2,V_2)$ are two GRSs, their {\it sum} 
is the GRS $(R_1 \cup R_2, V_1 \oplus V_2)$, where 
$V_1 \oplus V_2$ denotes the orthogonal direct sum of the Euclidean spaces 
$V_1$ and $V_2$. We say that $(R,V)$ is \textit{irreducible} if it is not the 
sum of two non-trivial
GRSs. Every GRS $(R,V)$ decomposes as the sum of non-trivial
irreducible GRSs $(R^1,V^1), (R^2,V^2), \ldots, (R^s,V^s)$ in a unique way, cf.
Corollary \ref{cor2.565}.

\np
By definition, the abelian group $Q$ generated by $R$ is the {\it root lattice} of $R$.

\np
Let $(R,V)$ be a GRS with root lattice $Q$.
If $Q' \subset Q$ a sublattice, $R'=R \cap Q$, and $V' := \Span Q'$, then
$(R', V')$ is a GRS and we say that it is a \textit{subsystem} of $(R,V)$.

\np
\begin{example}
\begin{itemize}
\item[(i)] Let $W \subset V$ be a subspace and let $Q' := Q \cap W$. 
Then $R' = R \cap Q' = R \cap W$ is a GRS in $W':=\Span Q' \subset W$. Note that
$W'$ does not necessarily equal $W$.
\item[] 
\item[(ii)]
Let $R=B_2=\pm\{\al,\be, \be+\al, \be+2\al\}$, cf. Section \ref{subsec2.455} 
below, and let $Q'$ be the
lattice generated by $2\al$, $\be$. Then $R'=\pm\{\be, \be+2\al\}$
is a subsystem of $R$ and $R'=A_1 \times A_1$. In this example the ranks of $R$ and $R'$ coincide.
\end{itemize}
\end{example}

\np 
Two irreducible GRSs $(R_1, V_1)$ and $(R_2, V_2)$ are 
{\it isomorphic} if there is a conformal linear isomorphism $\varphi: V_1 \to V_2$
which is a bijection between $R_1$ and $R_2$;
$(R_1, V_1)$ and $(R_2, V_2)$ are {\it  equivalent} if 
there is a vector space isomorphism 
$\varphi: V_1 \to V_2$ which is a bijection between $R_1$ and $R_2$.  
Two (not necessarily
irreducible) GRSs $(R_1, V_1)$ and $(R_2, V_2)$ are {\it isomorphic}
(respectively, equivalent)
if, for $i = 1,2$,  $(R_i,V_i)$ is the sum of the non-trivial irreducible GRSs
$(R_i^1,V_i^1), (R_i^2,V_i^2), \ldots, (R_i^s,V_i^s)$ and $(R_1^j,V_1^j)$ is
isomorphic (respectively, equivalent) to $(R_2^j,V_2^j)$ for every $j = 1, 2, \ldots, s$.
Since the  
spaces $V_1$ and $V_2$ are spanned by the GRSs $R_1$ and $R_2$ respectively, 
we will often omit mentioning $V_1$ and $V_2$ and will write
$\varphi: R_1 \to R_2$.
 The notion of isomorphism is natural and we will provide a more
 straightforward necessary condition for isomorphism between GRSs
 in terms of their respective Cartan matrices, cf. Section \ref{BPS} below.

\np
We denote isomorphism and equivalence between GRSs by $\cong$ and $\approx$ 
respectively.

\np 
It is clear that isomorphism implies equivalence, but the converse 
is not true as the following example shows.

\np
\begin{example} \label{ex2.40}
Let $\{e_1, e_2\}$ be the standard orthonormal basis of  $\RR^2$. One can check 
easily that $R_1 = \{ 0, \pm e_1, \pm e_2, \pm e_1 \pm e_2\}$ and
$R_2 = \{ 0, \pm e_1, \pm (\frac{4}{5} e_2 + \frac{3}{5} e_1), 
\pm e_1 \pm (\frac{4}{5} e_2 + \frac{3}{5} e_1)\}$ are GRSs in $\RR^2$. 
They are equivalent but not isomorphic.
\end{example}

\np
\point {\bf Bases and positive systems.} \label{BPS}
\begin{definition} \label{rs-baseA}
We say that $S \subset R$ is a \textit{base} of the GRS $(R,V)$ if
\begin{enumerate}
\item[(i)] $S$ is a basis of $V$;
\item[(ii)] Any element $\be \in R$ can written as $\be =\sum_{\alpha \in S} 
m_\alpha \al$ with all $m_\alpha \in \ZZ_{\geq 0}$ or  
all $m_\alpha \in \ZZ_{\leq 0}$. 
\end{enumerate}
The elements of $S$ are called \textit{simple roots}. 
If $\be =\sum_{\alpha \in S} m_\alpha \al$, the \textit{height of $\beta$ with respect to $S$} 
 is defined as
$\htt_S(\be) := \sum_{\alpha \in S} m_\alpha$. When $S$ is clear from the context, we write just $\htt(\beta)$.
\end{definition}

\np
Note that the base $S$ defines a partial order $\prec_S$ on $R$  by
$\al \prec_S \be$ if $\be - \al =\sum_{\alpha \in S} m_\alpha \al$ with 
$m_\alpha \in \ZZ_{\geq 0}$ for every $\alpha \in S$. When $S$ is clear from the context,
we write $\prec$ instead of $\prec_S$.

\np
\begin{definition}
We say that $R^+$ is a \textit{positive system} of $R$ if
\begin{enumerate}
\item For $\al$, $\be \in R^+$ and $\al+\be \in R$, we have $\al+\be \in R^+$;
\item $R=R^+ \cup -R^+$;
\item $R^+ \cap -R^+= \{0\}$.
\end{enumerate}
We set $R^- := - R^+$.
A root in $R^+$ is \textit{indecomposable} if it is not
the sum of two nonzero roots in $R^+$.
\end{definition}

\np If $S$ is a base of $R$, we set 

\begin{equation} \label{eq2.10}
R^+(S) = \{ \beta \in R \, | \, \be =\sum_{\alpha \in S} m_\alpha \al 
{\text { with }} m_\alpha \in \ZZ_{\geq 0}\}.
\end{equation}
It is immediate that $R^+(S)$ is a positive system.

\np
Given $\zeta \in V \setminus (\cup_{\al \in R\backslash\{0\}} H_\al)$
where $H_\al$ is the hyperplane orthogonal to $\al \in R$, we set

\begin{equation} \label{R+}
R^+(\zeta)=\{\be \in R \, | \, \langle\be,\zeta\rangle \geq 0\}.
\end{equation}
Clearly, $R^+(\zeta)$ is a positive system.

\np
GRSs exhibit properties analogous to those of root systems. 
The next proposition is the analog of \cite[Theorem 10.1']{Hu} and the proof of 
\cite[Theorem 10.1']{Hu} works in the setting of GRSs.

\np
\begin{proposition}\label{base-prop}
If $R$ is a GRS, then
\begin{enumerate}
\item[(i)] For any $\zeta \in V \setminus (\cup_{\al \in R\backslash\{0\}} H_\al)$, 
the indecomposable elements of $R^+(\zeta)$ form a base $S(\zeta)$ of $R$ 
and $R^+(\zeta) = R^+(S(\zeta))$.
\item[(ii)] For any base $S$ of $R$ there exists 
$\zeta(S) \in V \setminus (\cup_{\al \in R\backslash\{0\}} H_\al)$ such that
$S = S(\zeta(S))$. In particular, $R^+(S) = R^+(\zeta(S))$.
\end{enumerate}
\end{proposition}

\np
\begin{corollary} \label{cor3.35}
Let $\alpha \in R$ be a primitive root. There exists a base $S$ of $R$ such that
$\alpha \in S$.
\end{corollary}

\begin{proof} Fix $\zeta \in V$ such that $\langle \zeta, \alpha \rangle = 0$
but $\langle \zeta, \beta \rangle \neq 0$ for $\beta \in R \backslash \ZZ \alpha$.
One checks immediately that, for small enough $x>0$, $\alpha$ is an indecomposable
element of $R^+(\zeta + x \alpha)$. Hence, by Proposition \ref{base-prop} (i),
$\alpha \in S(\zeta + x \alpha)$.
\end{proof}

\np
Note that Proposition \ref{base-prop} guarantees that every GRS
admits a base. Moreover, it is true that every
positive system $R^+$ for $R$ is of the form $R^+(S)$ for some base $S$ 
(equivalently, of the form $R^+(\zeta)$ for some 
$\zeta \in V \setminus (\cup_{\al \in R\backslash\{0\}} H_\al)$, see 
Proposition \ref{prop2.75} below.

\np
\begin{definition} \label{Cartan}
If $S$ is a base of $R$, the {\it Cartan matrix} of the pair $(R,S)$ is the matrix  
$C = (c_{\alpha \beta})$, where
$c_{\alpha \beta} = \frac{2 \langle \al,\beta \rangle}{\langle \alpha,\alpha \rangle}$.
\end{definition}

\np
It is immediate from the definitions that the diagonal entries of the Cartan matrix $C$
of a pair $(R,S)$ equal $2$ and the off-diagonal entries are non-positive real numbers.
Moreover, as in the case of Kac-Moody algebras, $c_{\alpha \beta} = 0$ implies
$c_{\beta \alpha} = 0$. However, the off-diagonal entries of $C$ need not be integer.

\np 
The notions of isomorphism and  equivalence extend naturally to bases: 
We say that two triples $(R_1, V_1, S_1)$ and $(R_2, V_2, S_2)$ are {\it isomorphic}
(respectively, {\it  equivalent}) if 
there is an isomorphism (respectively, an  equivalence) of GRSs
$\varphi: V_1 \to V_2$, cf. Section \ref{sec2.21},
which is a bijection between $S_1$ and $S_2$. When discussing isomorphisms 
or  equivalences betweeen bases,
we will usually suppress $V_1$, $V_2$, $R_1$ and $R_2$ in the notation and 
will speak of isomorphisms 
and  equivalences between $S_1$ and $S_2$. Note that the bases
$S$ and $-S$ are isomorphic and, in particular, their respective Cartan
matrices coincide.

\np
If $S_1$ and $S_2$ are isomorphic bases, then their respective Cartan matrices coincide.
However, unlike root systems, a GRS may have 
bases which are not  equivalent (and hence not isomorphic either) to each other. In particular,
a GRS may have different Cartan matrices with respect 
to different bases. 
Conversely, non-isomorphic GRSs may have the same Cartan matrices: the root system 
$B_2$ and the non-reduced root system $BC_2$ 
have the same Cartan matrix. 

\np 
Deciding whether two GRSs are  equivalent is easier than deciding 
whether they are isomorphic. 
The following proposition is straightforward but will be very useful in
deciding whether two GRSs are isomorphic.

\np 
\begin{proposition} \label{prop2.267}
Let $\varphi:R_1 \to R_2$ be an  equivalence and let $S$ be a base of $R_1$. 
Then $S$ and $\varphi(S)$ are equivalent and $\varphi$ 
is an isomorphism if and only if the Cartan matrices corresponding to $(R_1, S)$ and
$(R_2, \varphi(S))$ coincide. 
\end{proposition}

\np 
Assuming that we can decide whether two GRSs $R_1$ and $R_2$ are isomorphic 
(or equivalent), the question whether two bases $S_1$ of $R_1$ and $S_2$ of $R_2$ 
are isomorphic (or  equivalent) reduces to deciding whether two bases of the same GRS, say $R_2$,
are isomorphic (or  equivalent). Indeed, for $S_1$ and $S_2$ to be isomorphic (or 
equivalent), $R_1$ and $R_2$ must be such. Assume that $\varphi:R_1 \to R_2$ is an isomorphism (or  
equivalence). Since $S_1$ and $\varphi(S_1)$ are isomorphic (or  equivalent), we conclude that
$S_1$ and $S_2$ are isomorphic (or  equivalent) if and only if the bases $\varphi(S_1)$ and $S_2$ 
of $R_2$ are isomorphic (or  equivalent).

\np
\point {\bf Strings and the highest root.}
Throughout this subsection $S$ is a fixed base of the GRS $R$ and 
$R^+ := R^+(S)$. Below we record several
properties of GRSs which are analogous to properties of root systems.

\np
\begin{proposition} \label{prop2.55}
Assume that $\beta', \beta'' \in R$ satisfy $\beta' \prec_S \beta''$. 
There exists a sequence of roots
\[\beta' = \beta_0, \beta_1, \ldots, \beta_k = \beta''\]
such that $\beta_{i} - \beta_{i-1} \in S$ for every $1 \leq i \leq k$.
\end{proposition} 

\begin{proof} Assume not. Pick a pair
$\beta' = \sum_{\alpha \in S}  b'_\alpha \alpha$ and 
$\beta'' = \sum_{\alpha \in S}  b''_\alpha \alpha$ for which the
statement above fails and such that 
$\htt \beta'' - \htt \beta' = \sum_{\alpha \in S} (b''_\alpha - b'_\alpha)$
is minimal. Consider the set $J := \{\alpha \in S \,|\, b'_\alpha < b''_\alpha\}$.
Then the choice of $\beta'$ and $\beta''$ implies that, 
for every $\alpha \in J$, $\beta'+ \alpha \not \in R$ and
$\beta'' - \alpha \not \in R$. Hence, for every $\gamma \in J$, we have
$\langle \beta', \gamma \rangle \geq 0$ and $\langle \beta'', \gamma \rangle \leq 0$.
Combining these we obtain
\[0 \geq \langle \beta'',\gamma \rangle = \langle \beta', \gamma \rangle +
\sum_{\alpha \in J} (b''_\alpha - b'_\alpha) \langle\alpha, \gamma \rangle,\]
implying that 
\[\sum_{\alpha \in J} (b''_\alpha - b'_\alpha) \langle\alpha, \gamma \rangle \leq 0.\]
Multiplying this inequality by $b''_\gamma - b'_\gamma$ and 
summing over all $\gamma \in J$, we obtain ${\bf{b}}^t B {\bf{b}} \leq 0$,
where ${\bf{b}} := (b''_\alpha - b'_\alpha)_{\alpha \in J}$ and 
$B:= (\langle\alpha, \gamma \rangle)_{\alpha, \gamma \in J}$. Since $B$ is a 
positive-definite matrix, we conclude that ${\bf{b}} = {\bf{0}}$,
i.e., that $\beta'' = \beta'$, contradicting the choice of $\beta'$ and $\beta''$.
\end{proof}

\np 
\begin{remark} \label{rem2.85}
As an immediate consequence of Proposition \ref{prop2.55}, we conclude that
the Hasse diagram of the poset $R$ is connected and its edges can be labeled by the
elements of $S$.
\end{remark}

\np
\begin{corollary}\label{grs-prop}
\begin{enumerate}
\item[(i)] If $\be \in R^+ \setminus \{0\}$, then $\be -\al \in R^+$ for 
some $\al \in S$.  
\item[(ii)] Let $\be \in R^+$. Then $\be$ can be written as a sum of simple
roots $\be = \al_{1}+\dots +\al_{k}$
in such a way that each partial sum $ \al_{1}+\dots +\al_{i}$ for $1 \leq i \leq k$
is also a root.
\end{enumerate}
\end{corollary}

\begin{proof} Apply Proposition \ref{prop2.55} to the pair of roots $0, \beta$.
\end{proof}

\np
\begin{corollary} \label{prop2.12}
Let $\al \neq 0,\be \in R$ and let  $k < s$ be integers such that 
$\beta + k \alpha$ and $\beta + s \alpha$ belong to $R$.
Then $\beta + t \alpha \in R$ for every integer $t$ in the interval $[k, s]$.
\end{corollary}

\begin{proof}
Apply Proposition \ref{prop2.55} to the pair of roots 
$\beta+k\alpha, \beta + s \alpha$.
\end{proof}

\np
Corollary \ref{prop2.12} implies that the set of integers $k$ such that 
$\beta + k \alpha \in R$ is a subinterval  of $\ZZ$, i.e., it is of the 
form $[-p,q] \cap \ZZ$. 
The set of roots $\beta - p \alpha, \ldots, \beta + q \alpha$ is called 
{\it the $\alpha$-string through $\beta$}. If $\beta - \alpha \not \in R$,
then the $\alpha$-string through $\beta$ is of the form 
$\beta, \ldots, \beta + \ell_{\alpha \beta} \alpha$ for some $\ell_{\alpha\beta} \in \Z_{\geq 0}$.
If $R$ is a root system and $\alpha$ and $\beta$ belong to a base of $R$, then 
$\ell_{\alpha \beta} = - c_{\alpha \beta}$, the corresponding entry of the Cartan matrix.
However, for GRSs, $\ell_{\alpha \beta}$ and $ - c_{\alpha \beta}$ do not have to coincide as
the examples of Section \ref{subsec2.455} show. Indeed, for the GRS from \ref{subsec2.455}(i)
we have $\ell_{\alpha \beta} =1$, while $c_{\alpha \beta}$ can take any real value 
in the interval $(-2,0)$;
similarly, in \ref{subsec2.455}(ii), $\ell_{\alpha \beta} = 2$, while $c_{\alpha \beta}$
can take any real value in the interval $(-4,0)$.

\np
\begin{proposition} \label{prop2.11} If $R$ is irreducible, then
there exists a unique maximal with respect to $S$ root $\theta$. Moreover, if
$\theta = \sum_{\alpha \in S} k_\alpha \alpha$, then  $k_\alpha >0 $ for 
every $\alpha \in S$. 
\end{proposition}

\begin{proof} 
The proof is analogous to the proof of Lemma A  in \cite[Section 10.4]{Hu}.
\end{proof}

\np
The unique maximal root $\theta$ above is called {\it the highest root} with 
respect to $S$.

\np 
\point{\bf The support of a root.} 
For the rest of this section $S$ is a fixed basis of $R$ with corresponding positive system $R^+$.
 Given $\beta \in R$, 
we define the support of $\beta = \sum_{\alpha \in S} k_\alpha \alpha$ as
\[ \supp \beta := \{\alpha \in S \, | \, k_\alpha \neq 0\} \ .\] 
A subset $I \subset S$ is connected if $I$ cannot be decomposed as the disjoint union 
$I = I' \sqcup I''$ where $I''$ is orthogonal to $I'$.

\np 
\begin{proposition} \label{prop2.57}
Assume that $S = S' \sqcup S''$ with $S' \perp S''$. Then 
$R = R' \cup R''$, where $R' \subset \Span S'$ and $R'' \subset \Span S''$.
\end{proposition}

\begin{proof}
Let $S' = \{\alpha_1', \ldots, \alpha_s'\}$ and $S'' = \{\alpha_1'', \ldots, \alpha_t''\}$.
If $\beta = \beta' + \beta'' \in R^+$, where $\beta' \in \Span S'$, $\beta'' \in \Span S''$,
we first show that $\beta', \beta'' \in R$ by induction on $\htt \beta$. The case when $\htt \beta = 0$ 
is trivial.
If $\htt \beta > 0$, then there is $\alpha \in S$ such that $\beta - \alpha \in R$. Assume that
$\alpha = \alpha_j' \in S'$. Then the induction assumption applies to $\beta - \alpha_j'$.
Hence $\beta' - \alpha_j', \beta'' \in R$. Moreover, 
$\langle \beta'', \beta \rangle  = \langle \beta'', \beta' + \beta'' \rangle > 0$,
implying that $\beta' \in R$.
The assumption that both $\beta'$ and $\beta''$ are nonzero, along with 
$\langle \beta', \beta'' \rangle = 0$, $\beta', \beta'' \in R$ and $\beta' - \beta'' \not \in R$ 
implies that $\beta \not \in R$. This completes the proof.
\end{proof}

\np 
\begin{corollary} \label{cor2.57}
If $\beta \in R$, then $\supp \beta \subset S$ is connected.
\end{corollary}

\begin{proof}
Assume to the contrary that $\supp \beta = S_1 \sqcup S_2$ with $S_1 \perp S_2$ and set
$\tilde{R} := R \cap \Span(S_1 \sqcup S_2)$. Then $\tilde{R}$ is a subsystem of $R$ which, 
by Proposition \ref{prop2.57}, decomposes as $\tilde{R} = \tilde{R}_1 \cup \tilde{R}_2$.
On the other hand $\beta \in \tilde{R}$ while $\beta \not \in \tilde{R}_1 \cup \tilde{R}_2$
unless $S_1$ or $S_2$ is empty.
\end{proof}

\np 
\begin{corollary} \label{cor2.565}
Every GRS decomposes uniquely as the sum of irreducible 
GRSs.
\end{corollary}

\np 
\begin{remark} The following partial converse of Corollary \ref{cor2.57} holds: 
if $I \subset S$ is connected, then $\sum_{\alpha \in I} \, \alpha \in R$. 
We omit the elementary proof of this fact since we will not need it.
\end{remark}

\np 
\point{\bf The poset $R^+$.} 
\np
The fixed base $S$ of $R$ defines a partial order on the root lattice $Q$ and, as a 
consequence, on any subset of $Q$
by $\lambda \prec \mu$ if $\mu - \lambda = \sum_{\alpha \in S} k_\alpha \alpha$ 
with $k_\alpha \in \ZZ_{\geq 0}$.
Recall that a partially ordered set (poset for short) is a lattice if 
any two elements $x, y$ have a supremum, also called {\it join} and denoted by $x \vee y$, and an
infimum, also called {\it meet} and denoted by $x \wedge y$.

\np
\begin{proposition} \label{prop2.67}
The poset $R^+$ is a lattice. Moreover, if $\beta' = \sum_{\alpha \in S} k_\alpha' \alpha$ and
$\beta'' = \sum_{\alpha \in S} k_\alpha'' \alpha$, then
$\beta' \wedge \beta'' = \sum_{\alpha \in S} \min(k_\alpha', k_\alpha'') \alpha$.
\end{proposition}

\begin{proof} 
For $\beta' = \sum_{\alpha \in S} k_\alpha' \alpha$ and 
$\beta'' = \sum_{\alpha \in S} k_\alpha'' \alpha$,
set $\beta := \sum_{\alpha \in S} \min(k_\alpha', k_\alpha'') \alpha$.
To prove that $\beta = \beta' \wedge \beta''$ it suffices to show that $\beta \in R$.
Assume to the contrary that $\beta \not \in R$. We can also assume that
the pair $\beta', \beta''$ is chosen so that $\sum_{\alpha \in S} |k_\alpha' - k_\alpha''|$
is minimal among all $\beta', \beta''$ for which $\beta$ defined as above does not belong to $R$.

\np 
Set $\gamma' := \beta' - \beta = \sum_{k_\alpha' > k_\alpha''} \, (k_\alpha' - k_\alpha'') \alpha$
and $\gamma'' := \beta'' - \beta = 
\sum_{k_\alpha'' > k_\alpha'} \, (k_\alpha'' - k_\alpha') \alpha$. 
If $\alpha \in S$ is such that $k_\alpha' > k_\alpha''$ \, , 
then $\langle \beta', \alpha\rangle \leq 0$
since otherwise $\beta' - \alpha \in R$ and the pair $\beta'-\alpha, \beta''$ 
corresponds to the same
element $\beta$, contradicting the choice of $\beta'$ and $\beta''$. As a consequence 
we conclude that
$\langle \beta', \gamma'\rangle \leq 0$. Similarly, $\langle \beta'', \gamma''\rangle \leq 0$.
Furthermore, $\langle \beta', \beta'' \rangle \leq 0$ because $\beta' - \beta'' \not \in R$ and
$\langle \gamma', \gamma'' \rangle \leq 0$ because the supports of $\gamma'$ and $\gamma''$ are
disjoint. These inequalities lead to
\[0 < \langle \beta + \gamma' + \gamma'', \beta + \gamma' + \gamma'' \rangle =
\langle \beta' + \gamma'', \beta'' + \gamma' \rangle =
\langle \beta', \beta'' \rangle + \langle \beta', \gamma'\rangle +
\langle \beta'', \gamma'\rangle + \langle \gamma', \gamma''\rangle \leq 0.\]
This contradiction proves that $\beta \in R$, i.e., that
$\beta' \wedge \beta'' = \sum_{\alpha \in S} \min(k_\alpha', k_\alpha'') \alpha$.

\np 
To prove that $\beta' \vee \beta''$ exists, consider the set
$X:= \{ \gamma \in R \, | \, \beta' \prec \gamma, \beta'' \prec \gamma\}$.
Since $X$ contains the highest root, it is not empty. The first part of the statement 
implies that the meet of $X$ is given by the formula
\[ \wedge X = \sum_{\alpha \in S} \min\{k_\alpha^\gamma \, | \, \gamma \in X\}\, \alpha \ ,\]
where, for $\gamma \in X$ we write $\gamma = \sum_{\alpha \in S} k_\alpha^\gamma \, \alpha$.
The formula for $\wedge X$ shows that it is a upper bound of $\{\beta', \beta''\}$ and hence it is
the join of $\beta'$ and $\beta''$.
\end{proof}

\np 
\begin{remark} \label{rem3.445}
\begin{itemize}
\item[(i)] Note that, in general, 
$\beta' \vee \beta'' \neq \sum_{\alpha \in S} \max(k_\alpha', k_\alpha'') \alpha$.
The reason is that the support of every root is connected, while the set
$\{ \alpha \in S \, | \, \max(k_\alpha', k_\alpha'') \neq 0\}$ is not necessarily 
connected. 
For example, in the root system $A_3$ with basis $\{\alpha_1, \alpha_2,\alpha_3\}$,
we have $\alpha_1 \vee \alpha_3 = \alpha_1 + \alpha_2 + \alpha_3$.
\item[(ii)] It is clear that $R^-$ is also a lattice. However, in $R^-$ we have 
$\beta' \vee \beta'' = \sum_{\alpha \in S} \max(k_\alpha', k_\alpha'') \alpha$,
while $\beta' \wedge \beta''$ does not necessarily equal
$\sum_{\alpha \in S} \min(k_\alpha', k_\alpha'') \alpha$.
\item[(iii)] Combining Proposition \ref{prop2.67} and (ii) above, we conclude that
$R$ itself is a lattice. However, in $R$ neither of the explicit 
formulas for $\beta' \wedge \beta''$ and $\beta' \vee \beta''$ holds.
\item[(iv)] Proposition \ref{prop5.255} below identifies some sublattices of $R$ which
are interesting on their own.
\end{itemize}
\end{remark}

\np 
The first named author learned the following property of GRSs from Etan Ossip while working
on \cite{USRA}. Proposition \ref{prop2.655} below and its proof inspired the idea for 
Proposition \ref{prop2.67} and its proof. We thank Etan Ossip for generously allowing us to
use his idea.

\np 
\begin{proposition} \label{prop2.655}
Let $\alpha, \beta, \gamma, \gamma + \alpha + \beta \in R$ but $\alpha + \beta \not \in R$.
If $\gamma + \alpha \neq 0$, then $\gamma + \beta \in R$.
\end{proposition}

\np 
Since we will not use Proposition \ref{prop2.655} in the rest of the paper, we refer the reader to \cite{USRA}
for the proof.

\np
\point{\bf Reduced GRSs of rank 2.} \label{subsec2.455}
To illustrate the basic definitions and properties so far, we provide the 
classification of all reduced GRSs of rank $2$ and their bases. 
Every reduced reducible GRS of rank 2 is isomorphic to the root system
$A_1 \times A_1$, so we
assume that $R$ is a reduced irreducible GRS of rank $2$ and fix a basis 
$S = \{\alpha, \beta\}$ of $R$. Denote the Cartan matrix of $R$ with respect to $S$ by 
\[
C = \left( \begin{array}{cc} 2 & a \\ b & 2 \end{array} \right) \ ,
\]
where $a := c_{\alpha,\beta} = 
\frac{2 \langle \alpha, \beta \rangle}{\langle \alpha, \alpha \rangle}$
and $b := c_{\beta, \alpha} = \frac{2 \langle \beta, \alpha \rangle}{\langle \beta, \beta \rangle}$.

\np
The $\alpha$-string through $\beta$ contains at most 4 elements. Indeed,
if $\beta, \ldots, \beta + 4 \alpha \in R$ then either 
$\langle \beta, \beta + 2 \alpha \rangle >0$ or 
$\langle \beta + 2 \alpha, \beta + 4 \alpha \rangle >0$, implying that $2 \alpha \in R$,
which is impossible. Similarly, the $\beta$-string through $\alpha$ contains at most
4 elements. We consider three cases for these strings.

\np
{\bf (i)} Both the $\alpha$-string through $\beta$ and the $\beta$-string through
$\alpha$ contain two elements. 
Then the assumptions on $R$ imply that none of the vectors 
$\beta + 2\alpha, 2 \beta, 2 \beta + \alpha, 2 \beta + 2 \alpha$ belongs to $R$.
Applying Corollary \ref{grs-prop}, we conclude that 
$R^+ = \{0, \alpha, \beta, \alpha + \beta\}$ and 
$R =\{0, \pm \alpha, \pm \beta, \pm (\alpha + \beta)\}$.

\np
There are six bases of $R$: $S_1 = S$, $S_2 = \{-\alpha, \alpha + \beta\}$,
$S_3 = \{\alpha + \beta, -\beta \}$, and their opposites. Since opposite bases
are isomorphic and their Cartan matrices coincide, we will focus on $S_1, S_2$,
and $S_3$.
The Cartan matrices of $R$ with respect to $S_1$, $S_2$, and $S_3$ are as follows:

\[
C_1 = \left( \begin{array}{cc} 2 & a \\ b & 2 \end{array} \right) \ , \quad \quad
C_2 = \left( \begin{array}{cc} 2 & -2-a \\ -\frac{(2+a)b}{a+b+ab} & 2 \end{array}
\right) \ , \quad \quad
C_3 = \left( \begin{array}{cc} 2 & -\frac{(2+b)a}{a+b+ab} \\ -2-b & 2 \end{array}
\right) \ . \quad \quad
\]

\np 
Since $\langle \alpha, \beta \rangle < 0$, $\langle \alpha, \alpha + \beta \rangle > 0$,
and $\langle \beta, \alpha + \beta \rangle > 0$, the parameters $a, b$ above satisfy
the inequalities $-2 < a < 0$ and $-2< b < 0$. For $a = b = -1$, $R$ is the root system
$A_2$.

\np 
Applying Proposition \ref{prop2.267}, we conclude that the GRSs corresponding to the
pairs of parameters $(a',b')$ and $(a'', b'')$ are isomorphic if an only 
if these pairs belong to the same
orbit of the action of the group generated by
\[(a,b) \mapsto \left(-2-a, - \frac{(2+a)b}{a+ b + ab}\right), \quad \text{and} \quad
(a,b) \mapsto \left( - \frac{(2+b)a}{a+ b + ab}, -2 - b\right) \ .\]
A generic orbit consists of $6$ elements.

\np
The linear transformations $\varphi_{21}$ and $\varphi_{31}$ defined by
\[ \begin{array}{lll}
\varphi_{21}(\alpha) = - \alpha & \phantom{xxxxxxxxxxx}
&\varphi_{31}(\alpha) = \alpha + \beta\\
\varphi_{21}(\beta) = \alpha + \beta & \phantom{xxxxxxxxxxx} &
\varphi_{31}(\beta) = - \beta \end{array}\]
provide  equivalences $S_1 \approx S_2$ and $S_1 \approx S_3$. 
Hence all bases of $R$ are  equivalent among themselves.

\np 
The matrices $C_1$ and $C_2$ coincide if and only if $a = -1$. Hence $S_1$ and
$S_2$ may be isomorphic only if $a = -1$. One checks easily that, for $a = -1$,
$\varphi_{21}$ is, in fact, an isomorphism. Similarly, $S_1 \cong S_3$ if and 
only if $b = -1$ and $S_2 \cong S_3$ if and only if $a = b$. In particular, 
$S_1 \cong S_2 \cong S_3$ if and only if $a = b = -1$, i.e., $R$ is the root system
$A_2$.

\np
{\bf(ii)} The $\alpha$-string through $\beta$ contains two elements. 
Then the assumptions on $R$ imply that the vectors 
$\alpha, \beta, \beta + \alpha, \beta + 2\alpha$ belong to $R$,
while 
$\beta + 3\alpha, 2 \beta, 2 \beta + 2 \alpha$ do not belong to $R$.
Since neither $\beta + (\beta + 2 \alpha)$ nor $\beta - (\beta + 2 \alpha)$ belongs
to $R$, we conclude that $\langle \beta, \beta + 2 \alpha \rangle = 0$. 
This in turn implies that $\langle 2 \beta + \alpha, \beta + 2 \alpha \rangle>0$,
proving that $2 \beta + \alpha \not \in R$ since otherwise $\beta - \alpha = 
(2 \beta + \alpha) - (\beta + 2 \alpha) \in R$.
Applying Corollary \ref{grs-prop}, we conclude that 
$R = \{0, \pm \alpha, \pm \beta, \pm(\beta + \alpha), \pm(\beta + 2 \alpha)\}$.
Since neither $(\beta + 2 \alpha) + \alpha$, nor $(\beta + 2\alpha) = \beta$ is
a root, we conclude that $\langle \beta + 2 \alpha, \beta \rangle = 0$.
In particular, 
$b := \frac{2 \langle \beta, \alpha \rangle}{\langle \beta, \beta \rangle} = -1$.

\np
There are eight bases of $R$: $S_1 = S$, $S_2 = \{-\alpha,  \beta + 2 \alpha\}$,
$S_3 = \{ \alpha + \beta, -\beta \}$, $S_4 = \{\alpha + \beta, -(\beta + 2 \alpha)\}$, 
and their opposites. As in case (i) above, we will focus on $S_1, S_2$, $S_3$,
and $S_4$. The Cartan matrices of $R$ with respect to 
$S_1$, $S_2$, $S_3$, and $S_4$ are as follows:

\[
C_1 = \left( \begin{array}{cc} 2 & a \\ -1 & 2 \end{array} \right) \ , \quad
C_2 = \left( \begin{array}{cc} 2 & -4-a \\ -1 & 2 \end{array}
\right) \ ,  \quad
C_3 = \left( \begin{array}{cc} 2 & a \\ -1 & 2 \end{array} \right) \ ,  \quad
C_4 = \left( \begin{array}{cc} 2 & -4-a \\ -1 & 2 \end{array}
\right)  \ .  
\]

\np
Since $\langle \alpha, \beta \rangle < 0$ and $\langle \alpha, 2\alpha + \beta \rangle > 0$,
the parameter $a$ above satisfies
the inequality $-4 < a < 0$. For $a = -2$, $R$ is the root system
$B_2$.
Applying Proposition \ref{prop2.267}, we conclude that the GRSs corresponding to the
 parameters $a'$ and $a''$ are isomorphic if an only if $a' + a'' = -4$.

\np
The linear transformations $\varphi_{21}$, $\varphi_{31}$, and $\varphi_{41}$ defined by
\[ \begin{array}{lllll}
\varphi_{21}(\alpha) = - \alpha & \phantom{xxxx}
&\varphi_{31}(\alpha) = \alpha + \beta & \phantom{xxxx} & \varphi_{41}(\alpha) = \alpha+\beta \\
\varphi_{21}(\beta) = 2\alpha + \beta & \phantom{xxxx} &
\varphi_{31}(\beta) = - \beta  & \phantom{xxxx} & \varphi_{41}(\beta) = 
-2 \alpha - \beta \end{array}\]
define  equivalences $S_1 \approx S_2$, $S_1 \approx S_3$, and $S_1 \approx S_4$. 
Hence all bases of $R$ are  equivalent among themselves.

\np
Moreover, $\varphi_{31}$ and $\varphi_{41} \circ \varphi_{21}^{-1}$ provide isomorphisms
$S_1 \cong S_3$ and $S_2 \cong S_4$.
The matrices $C_1$ and $C_2$ coincide if and only if $a = -2$. Hence $S_1$ and
$S_2$ may be isomorphic only if $a = -2$. For $a = -2$, the linear transformation
$\varphi_{21}$ is, in fact, an isomorphism. In particular, 
$S_1 \cong S_2 \cong S_3 \cong S_4$ if and only if $a = -2$, i.e., $R$ is the root system
$B_2$.

\np
{\bf(iii)} The $\alpha$-string through $\beta$ contains three elements. 
Then the assumptions on $R$ imply that the vectors 
$\alpha, \beta, \beta + \alpha, \beta + 2\alpha, \beta + 3 \alpha$ belong to $R$,
while $\beta + 4\alpha, 2 \beta + 2 \alpha, 2 \beta + 4 \alpha, 3 \beta + 3 \alpha$ 
do not belong to $R$.
Since neither $\beta + (\beta + 2 \alpha)$ nor $\beta - (\beta + 2 \alpha)$ belongs
to $R$, we conclude that $\langle \beta, \beta + 2 \alpha \rangle = 0$. 
This in turn implies that $\langle  \beta, \beta + 3 \alpha \rangle<0$,
proving that $2 \beta + 3 \alpha \in R$. The fact that neither $2 \beta + 4 \alpha$, nor
$3 \beta + 3 \alpha$ belongs to $R$ shows that $2 \beta + 3 \alpha$ is the highest root.
Hence $R = \{0, \pm \alpha, \pm \beta, \pm(\beta + \alpha), 
\pm(\beta + 2 \alpha),\pm(\beta + 3 \alpha),
\pm(2 \beta + 3 \alpha)\}$.

\np 
Arguing as above, we also obtain that $\langle \beta + \alpha, \beta + 3 \alpha \rangle = 0$.
This equation, along with $\langle \beta, \beta + 2 \alpha \rangle = 0$ implies that $a = -3, b =-1$
and thus $R$ is isomorphic to the root system $G_2$. In particular, it admits 12 bases which
are isomorphic among themselves.


\np
\section{Virtual reflections} \label{sec3}

\np
\begin{definition} \label{def3.15}
Given a nonzero $\alpha \in R$, we define the virtual reflection $\sigma_\alpha$ 
to be the following permutation of $R$:
if $\alpha$ is primitive, then $\sigma_\alpha$ reverses all
$\alpha$-strings;
if $\alpha$ is not primitive, then $\sigma_\alpha := \sigma_{\alpha'}$, where 
$\alpha = k \alpha'$ and $\alpha'$ is primitive. More precisely, if $\alpha$ 
is primitive, then 
for any $\alpha$-string
$\beta, \ldots, \beta + q \alpha$, we define
\[\sigma_\alpha(\beta + i \alpha) := \beta + (q-i) \alpha.\]
Note that $\sigma_\alpha(\alpha) = - \alpha$ and $\sigma_\alpha^2 = \id_R$.
\end{definition}

\np 
If $R$ is a root system, then $\sigma_\alpha$ is the reflection along
$\alpha$. 
In general $\sigma_\alpha$ cannot be extended to a linear transformation of $V$; 
even when it can be extended to a linear transformation, it is not necessary 
a reflection, see Proposition \ref{prop3.30} below.

\np 
\begin{remark} \label{rem3.1}
\begin{enumerate}
\item[(i)] As far as we are aware, the idea of virtual reflections was first introduced
by Penkov and Serganova in the context of Lie superalgebras to deal with the
problem that the Weyl group is too small to act transitively on the set of Borel subalgebras
containing a fixed Cartan subalgebra.
\item[(ii)] In principle, we may define the virtual reflection along any root 
$\alpha$ to be the permutation of $R$ reversing the $\alpha$-strings. However, 
with such a definition
the virtual reflections along $\alpha$ and $2 \alpha$ need not coincide, which we 
find undesirable. 
\item[(iii)] It is tempting to define the Weyl group of $R$ as the subgroup of 
the symmetric group of $R$ generated by all $\sigma_\alpha$. 
If $R$ is a root system, then this group
is isomorphic to the Weyl group but we do not know whether it plays an important 
role for a general GRS $R$.
\item[(iv)] Virtual reflections are invariants of  equivalence, i.e., 
if $\varphi: R_1 \to R_2$ is an  equivalence, then 
$\varphi \circ \sigma_\alpha = \sigma_{\varphi(\alpha)} \circ \varphi$ for any 
$\alpha \in R_1 \backslash \{0\}$.
\item[(v)] If $\alpha \in S$, $\beta \in R^+$, and $\beta - \alpha \not \in R$,
then $\sigma_\alpha(\beta) = \beta + \ell_{\alpha \beta} \alpha$.
\end{enumerate}
\end{remark}

\np 
\begin{proposition} \label{prop3.5} Let $S$ be a base of $R$ and let 
$\alpha \in S$. Then $\sigma_\alpha(S)$ is a base of $R$ and
\[R^+(\sigma_\alpha(S)) = \sigma_\alpha(R^+) = 
(R^+ \backslash \ZZ_{>0} \alpha) \cup (\ZZ_{<0} \alpha \cap R).\]
\end{proposition}

\begin{proof} We first note that if $\beta \in R^+$ is not proportional to 
$\alpha$, then the $\alpha$-string through $\beta$ is contained in $R^+$. Hence
\[ \sigma_\alpha(R^+) = (R^+ \backslash \ZZ_{>0} \alpha) \cup (\ZZ_{<0} 
\alpha \cap R).\]
Proposition \ref{base-prop} implies that $R^+(S) = R^+(\zeta)$ for some $\zeta$. 
Let $\xi \in V$ be determined by $\langle \alpha, \xi \rangle = 1$ and
$\langle \alpha', \xi \rangle = 0$ for $\alpha' \in S \backslash \{\alpha\}$.
Then $\tilde{\zeta} := \zeta - \langle \zeta, \alpha\rangle \xi$ satisfies
\[\langle \tilde{\zeta}, \alpha \rangle = 0 \quad {\text{and}} 
\quad \langle \tilde{\zeta}, \alpha' \rangle > 0
\quad {\text{for every }} \alpha' \in S \backslash \{\alpha\} \ .\]
Then, for small enough $\vep > 0$, the vector $\zeta' := \tilde{\zeta} - \vep \xi$ yields
\[R^+(\zeta')  = (R^+ \backslash \ZZ_{>0} \alpha) 
\cup (\ZZ_{<0} \alpha \cap R) = \sigma_\alpha(R^+) \ .\]
In particular, $\sigma_\alpha(R^+)$ is a positive system.

\np
To complete the proof it remains to show that $S(\zeta') = \sigma_\alpha(S)$. 
Clearly $\sigma_\alpha(S) \subset \sigma_\alpha(R^+)$, so it suffices to prove 
that the elements of $\sigma_\alpha(S)$ are indecomposable as elements of
$\sigma_\alpha(R^+)$. First, $-\alpha = \sigma_\alpha(\alpha)$ clearly is
indecomposable.
Next, let $\beta \neq \alpha \in S$. Then $\sigma_\alpha(\beta) = \beta + k \alpha$
for some $k \in \ZZ_{\geq 0}$. Assume to the contrary that 

\begin{equation} \label{eq3.3}
\beta + k \alpha = \beta' + \beta''
\end{equation}
for some $\beta', \beta'' \in \sigma_\alpha(R^+) = (R^+ \backslash \ZZ_{>0} \alpha)
\cup (\ZZ_{<0} \alpha \cap R)$.
First, $\beta'$ and $\beta''$ cannot both belong to $\ZZ_{<0} \alpha \cap R$. 
Next, if $\beta' \in R^+ \backslash \ZZ_{>0} \alpha$ and 
$\beta'' \in \ZZ_{<0} \alpha \cap R$, equation \eqref{eq3.3} implies that $\beta'$
belongs to the $\alpha$-string through $\beta$, which contradicts the fact that
$\beta + k \alpha$ is the end point of this string. Finally, if both $\beta'$ and
$\beta''$ belong to $R^+ \backslash \ZZ_{>0} \alpha$, then
the decomposition of each of them as a combination of elements of $S$ will contain 
an element other than $\alpha$, making \eqref{eq3.3} impossible.
\end{proof}

\np
\begin{proposition} \label{prop2.75}
Let $R^+ \subset R$ be a positive system. Then $R^+ = R^+(S)$ for some base $S$ 
of $R$.
\end{proposition}

\begin{proof}
Let $S$ be a base for which the cardinality of $R^+ \backslash R^+(S)$ is 
minimal. If $R^+ \neq R^+(S)$, then there exists $\alpha \in S \backslash R^+$
(otherwise
$R^+(S) \subset R^+$ and hence $R^+(S) = R^+$). Then 
$\#(R^+ \backslash R^+(\sigma_\alpha(S))) < \#(R^+ \backslash R^+(S))$, 
contradicting the choice of $\alpha$.
\end{proof}

\np 
Combining Propositions \ref{base-prop} and \ref{prop2.75}, we obtain the following result.

\np
\begin{corollary} \label{cor4.35}
The correspondence $S \longleftrightarrow R^+(S)$ is a bijection between the bases of $R$ and the 
positive systems of $R$. 
\end{corollary}

\np 
\begin{proposition} \label{prop4.357}
Let $S$ be a base of $R$ with $R^+ = R^+(S)$. If $I \subset S$, then the set
\[P:= (R^+ \backslash R_I^+) \cup (- R_I^+) \]
is a positive system of $R$ whose base contains the set $-I$.
\end{proposition}

\begin{proof} 
Verifying that $P$ is a positive system is straightforward.
To show that the set $-I$ is contained in the base of $P$, it suffices to
show that, for every $\alpha \in I$, the element $-\alpha$ is indecomposable in $P$. 
We omit this simple argument.
\end{proof}

\np 
\begin{proposition} \label{prop3.7}
Let $S'$ and $S''$ be two bases of $R$. There exists a sequence of bases
\[S' = S_0, S_1, \ldots, S_k = S''\]
and roots $\alpha_i \in S_i$ for $0 \leq i \leq k-1$ such that $S_{i+1} =
\sigma_{\alpha_i} (S_i)$.
\end{proposition}

\begin{proof} Induction on $\#(R^+(S'') \backslash R^+(S'))$, the case 
$\#(R^+(S'') \backslash R^+(S')) = 0$ being tautologically true.
If $\#(R^+(S'') \backslash R^+(S')) >0$, consider 
$\alpha_0 \in S' \backslash R^+(S'')$ and set $S_1: =\sigma_{\alpha_0}(S')$. 
Then $\#(R^+(S'') \backslash R^+(S_1)) < \#(R^+(S'') \backslash R^+(S'))$. 
Applying the induction assumption to the pair $S_1, S''$ completes the proof.
\end{proof}

\np
\begin{proposition} \label{prop3.30} Let $S$ be a base of $R$ and let 
$\alpha \in S$. The following statements hold: 
\begin{enumerate}
\item[(i)] If $\sigma_\alpha$ is the restriction of a linear transformation 
$\varphi$ of $V$, then $\varphi$ is an 
equivalence between $S$ and $\sigma_\alpha(S)$.
\item[(ii)] If $\sigma_\alpha$ is the restriction of an isometry $\varphi$ 
of $V$, then $\varphi$ is the reflection about $\alpha$ and is an isomorphism 
 between $S$ and $\sigma_\alpha(S)$.
\end{enumerate}
\end{proposition}

\begin{proof}
The only non-trivial statement is that the any isometry extending $\sigma_\alpha$ is the 
reflection about $\alpha$. Indeed, such $\varphi$ satisfies $\varphi^2 = \id_V$ and
$\varphi(\alpha) = -\alpha$. Let $\beta \in S \backslash \{\alpha\}$ decompose as
$\beta = c \alpha + \bar{\beta}$ with respect to the decomposition 
$V = \R \alpha \oplus (\R \alpha)^\perp$. Then the fact that $\sigma_\alpha$ preserves the
$\alpha$-string through $\beta$ implies that $\varphi(\bar{\beta}) = \bar{\beta}$,
proving that $\varphi$ is the reflection about $\alpha$.
\end{proof}

\np
\section{Quotients} \label{sec40}

\np
\point {\bf Definition.} \label{sec4.05}
Let $(R,V)$ be a GRS and let 
$S$ be a base of $R$. Fix a subset $I \subset S$ and set 
$V_I=\Span I$. The orthogonal
decomposition $V=V_I \oplus V_I^{\perp}$ defines
the projection $\pi_I: V \lra V_I^\perp$. 
Define 

\np
\beq \label{proj-def}
R_I=R \cap V_I, \qquad \RIperp=\pi_I(R) \subset V_I^\perp \ .
\eeq  

\np
Clearly, $(R_I, V_I)$ is a GRS which is a subsystem of $(R,V)$. 
However, in general, $\RIperp \not\subset R$. The goal of this 
section is to show that $(R_I^\perp, V_I^\perp)$ is a GRS and to describe its bases. 
We call $(R_I^\perp, V_I^\perp)$ the {\it quotient} of $(V,R)$ with respect to $I$. 
Note that $(R_I^\perp, V_I^\perp)$ depends on $I$ but not on $S$ as long as $I$ is a subset 
of some base of $R$. Accordingly,
in what follows we will sometimes use the more functorial notation $R/I$, $V/I$, $S/I$, etc. in
place of $R_I^\perp, V_I^\perp, \pi_I(S) \backslash\{0\}$, etc.

\np
\point{\bf Main theorem about quotients.}

\np
\begin{theorem} \label{the5.25}
In the notation above $(R/I, V/I)$ is a GRS with a base $S/I$.
\end{theorem}

\begin{proof} First we note that $S/I$ satisfies the conditions of Definition 
\ref{rs-baseA} with respect
to $R/I$ and $V/I$, so we only need to prove that $(R/I, V/I)$ is a GRS. 
We proceed by induction on the cardinality of $I$.

\np
Start with the case when $I=\{\al\}$ consists of a single element. 
To simplify notation, for any $\nu \in V$, we denote $\pi_I(\nu) \in V/I$ by $\bar{\nu}$. 
In particular,
$\nubar=\pi_I(\nu)=\nu- \frac{\langle \nu, \al \rangle}{\langle \al, \al \rangle} \, \al$.
 
 \np
Let $\bebar$, $\gabar \in \bar{R}$. First we show that
$\langle\bebar, \gabar\rangle < 0$ implies that $\bar{\beta} + \bar{\gamma} \in \bar{R}$.
Indeed, assume that
$\langle\bebar, \gabar\rangle < 0$. Then
$$
\langle\bebar, \gabar\rangle=\langle\be,\gamma\rangle-
\langle\al,\be\rangle\langle\al,\gamma\rangle/\langle\al,\al\rangle < 0 \ .
$$
If $\langle\be,\gamma\rangle < 0$ then $\beta + \gamma \in R$ and hence 
$\bar{\beta} + \bar{\gamma} = \pi_I(\beta + \gamma) \in \bar{R}$. 
If, on the other hand, $\langle\be,\gamma\rangle \geq 0$, then
$$
0\leq \langle\be,\gamma\rangle < \langle\al,\be\rangle\langle\al,\gamma\rangle/
\langle\al,\al\rangle \ ,
$$
implying that either $\langle\al,\be\rangle>0$ and $\langle\al,\gamma\rangle>0$
or $\langle\al,\be\rangle<0$ and $\langle\al,\gamma\rangle < 0$. 
Assume first that $\langle\al,\be\rangle>0$ and $\langle\al,\gamma\rangle>0$.
Hence $\beta  - \al \in R$.
Let $\ell$ be the largest integer such that
$\beta - \ell \al \in R$. Since $\beta - (\ell + 1) \alpha \not \in R$, while
$\beta - (\ell -1) \alpha \in R$, we conclude that $\langle \beta - \ell \al, \al \rangle < 0$.
Consequently,  $\ell > \langle\alpha,\beta \rangle/\langle\al,\al\rangle$ and
$$
\langle\beta - \ell \al, \gamma \rangle=\langle\be,\gamma\rangle-\ell\langle\al,\gamma\rangle <
\langle\be,\gamma\rangle - \langle\al,\gamma\rangle\langle\al,
\be\rangle/\langle\al,\al\rangle = \langle\bebar,\gabar\rangle <0 \ .
$$
Hence $\beta  - \ell \al + \gamma \in R$ and thus $\bebar + \gabar  = 
\pi_I(\beta - \ell \al + \gamma) \in \bar{R}$.
In the case when $\langle\al,\be\rangle<0$ and $\langle\al,\gamma\rangle < 0$ 
we arrive at the conclusion that
$\bebar + \gabar \in \bar{R}$ in a similar way.

\np 
The proof that 
$\langle\bebar, \gabar\rangle > 0$ implies that $\bar{\beta} - \bar{\gamma} \in \bar{R}$
is similar and we omit it.

\np
Finally, we need to prove that $\langle\bebar,\gabar\rangle=0$
implies that either both $\bebar \pm \gabar$ belong to $\bar{R}$ or
both $\bebar \pm \gabar$  do not belong to $\bar{R}$.

\np 
Let $\langle\bebar,\gabar\rangle=0$. Assume first that $\langle\be,\gamma\rangle > 0$.
Then $\be-\gamma \in R$, and hence $\bebar-\gabar \in \bar{R}$, so it remains to show that
 $\bebar+\gabar \in \bar{R}$. From the equation
$$
\langle\bebar,\gabar\rangle=\langle\be,\gamma\rangle-
\langle\al,\be\rangle\langle\al,\gamma\rangle/\langle\al,\al\rangle=0
$$
we conclude that
$$
\langle\al,\be\rangle\langle\al,\gamma\rangle/\langle\al,\al\rangle>0 \ .
$$
This in turns implies that either $\langle\al,\be\rangle>0$ and $\langle\al,\gamma\rangle>0$
or $\langle\al,\be\rangle<0$ and $\langle\al,\gamma\rangle < 0$. Repeating the argument above, we 
arrive at the conclusion that $\bar{\beta} + \bar{\gamma} \in \bar{R}$.

\np 
The case when $\langle\be,\gamma\rangle < 0$ is dealt with analogously, so 
it remains to consider the
case when $\langle\be,\gamma\rangle = 0$. We will show then that $\bebar + \gabar \in \bar{R}$
implies $\bebar - \gabar \in \bar{R}$ and vice versa, i.e., $\bebar - \gabar \in \bar{R}$
implies $\bebar + \gabar \in \bar{R}$. Assume $\bebar + \gabar \in \bar{R}$. 
Then there is an integer $s$ such that $\beta + \gamma + s \alpha \in R$. 
Note that 
$$
\langle\bebar,\gabar\rangle=\langle\be,\gamma\rangle-
\langle\al,\be\rangle\langle\al,\gamma\rangle/\langle\al,\al\rangle=0
$$
along with $\langle\bebar,\gabar\rangle=\langle\be,\gamma\rangle = 0$
imply that $\langle\al,\be\rangle = 0$ or $\langle\al,\gamma\rangle = 0$. Say,
$\langle\al,\be\rangle=0$. Then
$$
\langle\be+\gamma+s\al,\be\rangle=\langle \be,\be\rangle>0
$$
and thus  $(\be+\gamma+s\al)-\be=\gamma+s\al \in R$. Since
$\langle\gamma+s\al,\be\rangle=0$ and $\be+(\gamma+s\al)\in R$, we
have $\be-(\gamma+s\al) \in R$, proving 
that $\bebar-\gabar = \pi_I(\be - (\gamma + s \al)) \in \bar{R}$.
The fact that $\bebar - \gabar \in \bar{R}$ implies $\bebar + \gabar \in \bar{R}$
is proved analogously. This completes the proof that $(R/I, V/I)$ is a GRS
when $I$ is a singleton.

\np 
Assume now that $I$ consists of more than one element and let $J \subset I$ be a singleton.
Let $\pi_{\bar{I}}: V/J \to V/I$ be the orthogonal projection 
determined by the set $\bar{I} := I/J \subset S/J$.
Since $\pi_I = \pi_{\bar{I}} \circ \pi_J$, the pair $(R/I, V/I)$ is the 
quotient of $(R/J, V/J)$ with
respect to $\bar{I} \subset S/J$. Moreover, $(R/J, V/J)$ is a GRS because $J$ 
is a singleton and hence
$(R/I, V/I)$ is the quotient of a GRS with respect to $\bar{I}$ whose 
cardinality is one less than the cardinality
of $I$. Thus, by induction, $(R/I, V/I)$ is a GRS.
\end{proof}

\np
\point{\bf Functoriality.} \label{sec5.35}
We record an important property of quotients of GRSs. Its 
proof is immediate and we omit it.

\np
\begin{proposition}
Let $R$ be a GRS and with a base $S$. 
If $I \subset J \subset S$, then $(R/I)/(J/I) \cong R/J$. 
\end{proposition}

\np 
\point{\bf The fibres of $\pi_I$.} \label{sec5.45}
For $\nu \in R/I$, denote the fibre of $\nu$ under $\pi_I$
by $R^\nu$, i.e., $R^\nu = \pi_I^{-1}(\nu)$.

\np 
\begin{proposition} \label{prop5.331}
Let $I \subset S$ and let $R^+ = R^+(S)$. If $\nu \in R/I$ is nonzero, then
$R^\nu \subset R^+$ or $R^\nu \subset R^-$.
\end{proposition}

\begin{proof}
If  $\gamma' = \sum_{\alpha \in S} c_\alpha' \alpha$ and 
$\gamma'' = \sum_{\alpha \in S} c_\alpha'' \alpha$ are two elements of $R^\nu$, then
$c_\alpha' = c_\alpha''$ for every $\alpha \in S \backslash I$. Moreover, there is
$\alpha_{0} \in S \backslash I$   such that $c_{\alpha_0}' = c_{\alpha_0}'' \neq 0$. Since the sign 
of $c_{\alpha_0}'$ (respectively, of $c_{\alpha_0}''$) 
determines whether $\gamma'$ (respectively, $\gamma''$) belongs to $R^+$ or $R^-$, we conclude that
$\gamma'$ and $\gamma''$ both belong to $R^+$ or both belong to $R^-$, proving that
$R^\nu \subset R^+$ or $R^\nu \subset R^-$.
\end{proof}

\np 
Note that $R^\nu$ is a poset with the order inherited from $R$.
Moreover the edges of the Hasse diagram of $R^\nu$ are labeled by elements of $I$, cf. 
Remark \ref{rem2.85}.

\np 
\begin{proposition} \label{prop5.255}
If $\nu \neq 0$, then $R^\nu$ is a lattice. Moreover, 
if $\beta' = \sum_{\alpha \in S} k_\alpha' \alpha$ and
$\beta'' = \sum_{\alpha \in S} k_\alpha'' \alpha$, then
$\beta' \wedge \beta'' = \sum_{\alpha \in S} \min(k_\alpha', k_\alpha'') \alpha$ and
$\beta' \vee \beta'' = \sum_{\alpha \in S} \max(k_\alpha', k_\alpha'') \alpha$.
\end{proposition}

\begin{proof}
The observation that if $R^\nu \subset R^-$, 
then $-R^\nu = R^{-\nu} \subset R^+$ and $\gamma' \prec \gamma ''$ is
equivalent to $-\gamma'' \prec -\gamma'$ implies that it suffices to consider 
the case $R^\nu \subset R^+$. 

\np 
Assume $R^\nu \subset R^+$. Let 
\[\beta' = \sum_{\alpha \in S} k_\alpha' \alpha  \quad {\text{and}} \quad 
\beta'' = \sum_{\alpha \in S} k_\alpha'' \alpha \in R^\nu \ .\]
The explicit formula in Proposition \ref{prop2.67} for the meet of $\beta'$ and $\beta''$ 
in the poset $R^+$ implies that it belongs to $R^\nu$, proving that it is also 
the meet of $\beta'$ and $\beta''$ in $R^\nu$.

\np 
Next we prove that $\beta' \vee \beta''$ exists and equals
$\sum_{\alpha \in S} \max(k_\alpha', k_\alpha'') \alpha$.
Applying Proposition \ref{prop4.357}, we define a positive system $P$
whose basis contains $-I$.
Using that $\pi_I = \pi_{-I}$ and the formula for $P$, we conclude that $R^\nu \subset P$.
The orders on $R^\nu$ inherited from the orders on $R$ 
corresponding to $R^+$ and $P$ are opposite to each other. In particular,
the meet of $\beta'$ and  $\beta''$ with respect to $P$ becomes their join
with respect to $R^+$. This completes the proof that  $\beta' \vee \beta''$ exists and equals
$\sum_{\alpha \in S} \max(k_\alpha', k_\alpha'') \alpha$.
\end{proof}

\np 
We end  the discussion of the set $R^\nu$ by an immediate corollary of 
Proposition \ref{prop5.255}
which is useful for some applications.

\np 
\begin{corollary} \label{cor5.37}
If $\nu \in R/I$ then the Hasse diagram of $R^\nu$ is connected. If, in addition, 
$\nu \neq 0$, then $R^\nu$ contains a smallest and a largest element.
\end{corollary}

\np
\point{\bf Bases of $R/I$.} 

\np
\begin{proposition}\label{rs-base}
A subset $T$ of $R/I$ is a base of $R/I$ if and only if there exists a base $S$ of $R$
such that $I \subset S$ and $T= \pi_I(S) \backslash \{0\} = S/I$.
\end{proposition}

\begin{proof}
The fact that $S/I$ is a base of $R/I$ follows from Theorem \ref{the5.25}. 

\np
Assume that $T$ is a base of $R/I$ and for every $\nu \in T$, let $\beta(\nu)$ be the smallest 
element of $R^\nu$. Set $S:= I \cup \{\beta(\nu) \, | \, \nu \in T\}$. The base $T$ 
determines the positive system $(R/I)^+$ of $R/I$. Let 
$P = R_I^+ \cup \pi_I^{-1}((R/I)^+ \backslash \{0\})$. It is easy to verify that $P$ is a 
positive system in $R$ containing $S$. 

\np 
We claim that the elements of $S$ are indecomposable in $P$. Assume to the contrary that
there is $\alpha \in S$ such that $\alpha = \gamma' + \gamma''$ for some 
nonzero $\gamma', \gamma'' \in P$.
If $\alpha \in I$, then $0 = \pi_I(\alpha) = \pi_I(\gamma') + \pi_I(\gamma'')$ 
implies that $\gamma', \gamma'' \in R_I$, which is impossible because $\alpha$ is a simple root
of $R_I$ and hence it is indecomposable in $R_I^+$.
If, on the other hand, $\alpha = \beta(\nu)$ for some $\nu \in T$, then
the equation
$\nu = \pi_I(\beta(\nu)) = \pi_I(\gamma') + \pi_I(\gamma'')$ and the fact that
$\nu$ is indecomposable in $(R/I)^+$ imply that $\pi_I(\gamma') = \nu$ and $\gamma'' \in R_I$
(or vice-versa). This leads to the equation $\beta(\nu) = \gamma' + \gamma''$ in which
$\gamma' \in R^\nu$ and $\gamma'' \in R_I^+$, contradicting the assumption that $\beta(\nu)$
is the smallest element of $R^\nu$.

\np 
Clearly the cardinality of $S$ equals the rank of $R$ and $S$ consists of elements indecomposable
in $P$, proving that $S$ is the base associated to $P$. By construction, 
$I \subset S$ and $T = S/I$.
\end{proof}

\np
\begin{corollary} \label{cor5.35}
There is a bijection between the bases of $R/I$ and the
bases of $R$ containing $I$.
\end{corollary}

\np
\section{Quotients of root systems I: structure and properties} \label{Sec55}

\np 
\point{\bf {Notation.}} \label{sec55.5}
In this section we follow \cite{B} for the 
notation and labeling of roots in a root system.
We refer to root systems as $X_l$, where $X = A, B, C, D, E, F, G$ and $l$ 
takes values according to 
$X$. In each root system $X_l$ we fix an ordered base 
$\Sigma = \{\alpha_1, \ldots, \alpha_l\}$ labeled as in \cite{B}.
If $J \subset \{1, \ldots, l\}$, we denote the projection $\pi_{\{\alpha_j \, | \, j \in J\}}$ by 
$\pi_J$ and let $X_l^J$ be the quotient GRS $X_l/\{\alpha_j \, | \, j \in J^c\}$ where $J^c$ 
is the complement of $J$ in $\{1, \ldots, l\}$. We denote by $\Sigma^J$ the 
base $\pi_{J^c}(\Sigma)$ of $X_l^J$, i.e.,
$\Sigma^J = \{\pi_{J^c}(\alpha_j) \, | \, j \in J\}$. The reason for 
using $J$ in the notation of the quotient
by $\{\alpha_j \, | \, j \in J^c\}$ is that the rank of $X_l^J$ equals the cardinality of $J$.
For brevity of notation we sometimes write the set $J$ as a list, e.g. 
$F_4^{2,4}$ denotes the quotient 
$F_4/\{\alpha_1,\alpha_3\}$.

\np
Calculating the equivalence class of every quotient of a root system is not a difficult problem.
The quotients of the infinite series $A_l, B_l, C_l, D_l$ are calculated explicitly in \cite{DR}.
Generating the list of all equivalence classes of 
quotients of the exceptional root systems is a tedious but straightforward
process. Indeed, the GRS $X_l^J$ is equivalent to the set obtained from a 
list of the roots of $X_l$ by
ignoring the coordinates corresponding to $J^c$. Here is an example.

\np 
\begin{example} \label{ex6.45}
Consider the root system $F_4$. The list of positive roots of $F_4$ 
with respect to $\Sigma$  is as follows:
\[\begin{array}{llllllll}
(1,0,0,0) & (0,1,0,0) & (0,0,1,0) & (0,0,0,1) & (1,1,0,0)&(0,1,1,0)&(0,0,1,1) & (1,1,1,0)\\
(0,1,2,0) & (0,1,1,1) & (1,1,2,0)&(1,1,1,1) & (0,1,2,1) & (1,2,2,0) & (1,1,2,1) & (0,1,2,2)\\
(1,2,2,1)&(1,1,2,2)&(1,2,3,1) & (1,2,2,2) & (1,2,3,2) & (1,2,4,2) & (1,3,4,2)&(2,3,4,2) \ .
\end{array}
\]
Ignoring the first and the third coordinates of each of these vectors and 
listing each resulting vector only once
we obtain the following GRS which is equivalent to $(F_4^{2,4}, \Sigma^{2,4})$:
\[R = \{(0,0), \pm (1,0), \pm (0,1), \pm (1,1), \pm (2,0), \pm(1,2), 
\pm(2,1), \pm(2,2), \pm (3,2)\} \]
with base $S = \{(1,0), (0,1)\}$.

\np 
Computing the Euclidean structure of $F_4^{2,4}$ 
requires an additional calculation which we omit here and record the result:
\[
\|\bar{\alpha}_2\|^2 = 1, \quad \|\bar{\alpha}_4\|^2 = \frac{3}{2}, \quad
\langle \bar{\alpha}_2, \bar{\alpha}_4 \rangle = -1 \ .
\]
\end{example}

\np Two natural questions arise:
\begin{enumerate}
    \item[(i)] For what $J', J''$ are the GRSs $X_l^{J'}$ 
    and $X_l^{J''}$ isomorphic (or equivalent)?
    \item[(ii)] Given $\alpha \in \Sigma^J$, can we describe the base $\sigma_\alpha(\Sigma^J)$ of
    $X_l^J$?
\end{enumerate}

\np 
To clarify the second question, we note that, by Proposition \ref{rs-base}, any base of $X_l^J$ is
a quotient of a base of $X_l$ containing the set $\{\alpha_j \, | \, j \in J^c\}$. In particular
$\Sigma^J = \{\pi_{J^c}(\alpha_j) \, | \, j \in J\}$ is such a quotient. On the other hand,
all bases of $X_l$ are conjugate under the action of the 
Weyl group of $X_l$. Hence, after conjugating,
we find that $\sigma_\alpha(\Sigma^J)$ is isomorphic to $\Sigma^{J'}$ for some index set $J'$.
So, the second question is whether we can describe explicitly 
such a set $J'$. In Theorem \ref{the6.375} below we provide an explicit 
description of such a set, answering the second question above. 
Theorem \ref{the6.375} also provides a description of
the virtual reflection $\sigma_{\alpha_j}$ from Section \ref{sec3} for GRSs which 
are quotients of root systems.
Finally, this result also means that $X_l^{J}$ is isomorphic to $X_l^{J'}$,
providing a partial answer (a necessary condition) of the first question.

\np 
\point{\bf Kostant's theorem.} \label{sec6.2}
Let $\gg$ be a complex simple Lie algebra  
with a Cartan subalgebra $\gh \subset \gg$, roots $\Delta$, and a base $S$ of $\Delta$. 
Given a subset $I \subset S$,
let $\gs' \subset \gg$ denote the semisimple subalgebra of $\gg$ 
generated by the root spaces of $\gg$
corresponding to $I$. Set $\gh_\gs: = \gh \cap \gs'$,
$\gs:= \gs' + \gh$, and 
let $\gt_\gs:= \gc_\gs(\gs')$ be the centralizer
of $\gs'$ in $\gs$. Then 

\begin{equation} \label{eq6.17}
\gh = \gh_\gs \oplus \gt_\gs,
\end{equation}
where the direct sum is orthogonal, and  $\gs = \gs' \oplus \gt_\gs$.
As an $\gs$-module, $\gg$ decomposes as

\begin{equation} \label{eq625}
\gg = \oplus_{\nu \in \Delta/I} \ \gg_\nu, 
\end{equation}
where
$\gg_\nu = \{ X \in \gg \, | \, [t, X] = \nu(t) X {\text{ for every }} t \in \gt_s\}$
with $\nu \in \Delta/I = \Delta_I^\perp \subset (\gh^*_\gs)^\perp = \gt_\gs^*$.
Kostant, \cite{Ko}, proved that, for $\nu \neq 0$, $\gg_\nu$ is an irreducible $\gk$-module.

\np
\point {\bf Main Theorem.} \label{sec6.35}
We modify slightly the notation above to fit our setting. 
Let $\gg$ be a complex simple Lie algebra of type $X_l$ 
with a Cartan subalgebra $\gh \subset \gg$. Given an index set $J \subset \{1, 2, \ldots , l\}$
and an element $j \in J$, set $I := J^c$ and $K:= I \cup \{j\}$. Let
$\gs'$ and $\gk'$ denote the semisimple subalgebras of $\gg$ generated by the root spaces of $\gg$
corresponding to the roots indexed by 
$I$ and $K$ respectively. Note that, unlike Section \ref{sec6.2} above, $I$ denotes
a set of indices instead of a set of roots. In order to simplify the notation, we 
will sometimes use $I$ in place of $\{\alpha_i \, | \, i \in I\} \subset S$.

\np 
Let $\gh_\gs$, $\gh_\gk$, $\gt_\gs$, and $\gt_\gk$ be
defined as in Section \ref{sec6.2} above. 
Combining the decompositions \eqref{eq6.17} with respect to $\gs$ and $\gk$  
and switching to the respective dual spaces, we 
obtain the orthogonal decomposition

\begin{equation} \label{eq6.35}
\gh^* = \gh_\gs^* \oplus \CC \bar{\alpha} \oplus \gt_\gk^*,
\end{equation}
where $\bar{\alpha} := \pi_I(\alpha_j)$.

\np
Let $\vartheta$ be the unique diagram automorphism of $\gk'$
with the property that, for every $\gk'$-module $W$, we have $W^\vartheta \simeq W^*$,
where $W^\vartheta$ is the $\gk'$-module on which the action of $\gk'$ is twisted by $\vartheta$.
More precisely, the underlying vector space of $W^\vartheta$ is $W$ and, for $X \in \gk'$,
we define the twisted action of $X$ on $v \in W$ by
\[ X \cdot_\vartheta v := \vartheta(X) \cdot v.\]

\np 
Clearly, $\vartheta$ is an involution. More explicitly,
if $\gk'$ is simple, then $\vartheta \neq \id$ if and only if $\gk'$ is 
of type $A_l, D_{2l+1},$ or $E_6$;
if $\gk'$ is of type $A_l, D_{2l+1},$ or $E_6$, then $\vartheta$ the the unique diagram 
automorphism of order 2. If $\gk'$ is semisimple,
that $\vartheta$ is described as above on each simple ideal of $\gk'$ (equivalently, on each 
component of the Dynkin diagram of $\gk'$). 
Extend $\vartheta$ to an involution of $\gk$ by setting $\vartheta(t) := - t$ for every $t \in \gt_\gk$.
This ensures that $W^\vartheta \simeq W^*$ for every $\gk$-module $W$.
The involution $\vartheta$ induces an involution $\vartheta^*$ on $\gh^*$. 

\np
\begin{example} \label{ex7.65}
Figure \ref{fig25} illustrates the three involutions $\vartheta$ corresponding to $X_l = E_8$ and
$J = \{1,6,8\}$, i.e., the first one corresponds to $K = \{1,2,3,4,5,7\}$, the second -- to  
$K = \{2,3,4,5,6,7\}$, and the third -- to $K = \{2,3,4,5,7, 8\}$. The respective subalgebras
$\gk'$ are isomorphic to $D_5 \times A_1$, $D_6$, and $D_4 \times A_2$. 
\end{example}

\begin{figure}[!b]
\begin{center}
\hspace{-1cm}
    \begin{tabular}{ccc}
        \includegraphics[width=.30\textwidth]{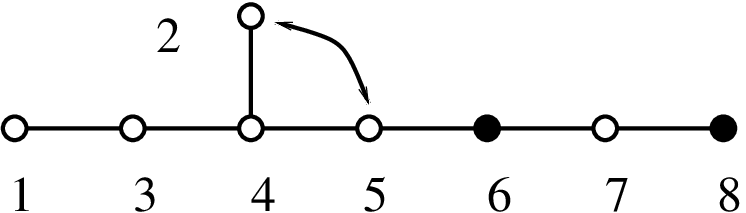}&
        \includegraphics[width=.30\textwidth]{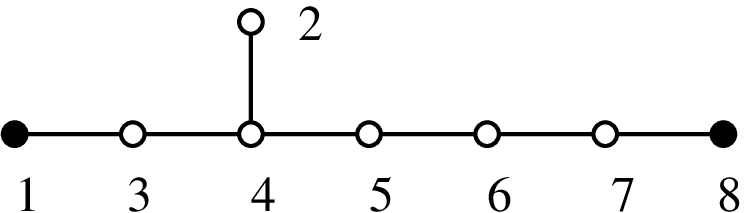}&
        \includegraphics[width=.30\textwidth]{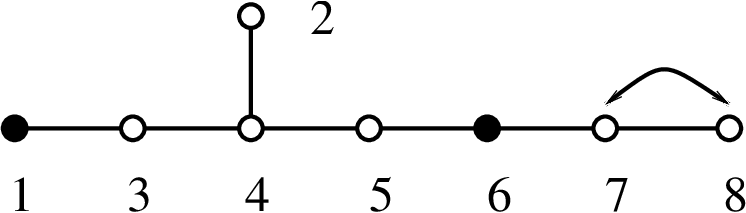}
    \end{tabular}
\caption{The three involutions $\vartheta$ corresponding to $E_8^{1,6,8}$}
\label{fig25}
\end{center}
\end{figure}

\np 
If $\Delta$ are the roots of $\gg$, we define $V:= \Span_\RR \Delta$, and let
$V_I$, $V_{\vartheta(I)}$, and $V_K$ be the subspaces of $V$ spanned by roots corresponding to
$I, \vartheta(I)$, and $K$ respectively. Then $\vartheta^*$ restricts to an
involution, still denoted by $\vartheta^*$ of $V$. Clearly $\vartheta^*:V \to V$ 
respects the orthogonal decomposition $V = V_K \oplus V_K^\perp$. Furthermore, the 
restrictions of $\vartheta^*$ to both $V_K$ and $V_K^\perp$ are isometries: 
the former because $\vartheta$ was initially defined as a diagram automorphism of $\gk'$
and the latter because it equals $-\id_{V_K^\perp}$. Combining these facts, we
conclude that $\vartheta^*:V \to V$ is an isometry.

\np 
Since $V_I^\perp =  V_K^\perp \oplus \RR \, \bar{\alpha}$ and 
$V_{\vartheta(I)}^\perp =  V_K^\perp \oplus \RR \, \vartheta^*(\bar{\alpha})$, we have the isometry
$\varphi:V_I^\perp \to V_{\vartheta(I)}^\perp$ defined as

\begin{equation} \label{eq6.37}
\varphi := - \vartheta^*: \left\{ \begin{array}{ccl}
       V_I^\perp &\to& V_{\vartheta(I)}^\perp \\
      \nu +  t \, \bar{\alpha} & \mapsto & \nu - t \, \vartheta^*(\bar{\alpha}) \ .
 \end{array} \right. 
\end{equation}
Remarkably, $\varphi$ is an isomorphism between the GRSs $X_l^J$ and $X_l^{\vartheta(J)}$
and an isomorphism between the bases $\sigma_{\alpha_j}(\Sigma^J)$ and $\Sigma^{\vartheta(J)}$
of $X_l^J$ and $X_l^{\vartheta(J)}$ respectively, as the following theorem states.

\np 
\begin{theorem} \label{the6.375}
In the notation above, $\varphi$ is an isomorphism of GRSs 
\[ \varphi: (X_l^J, V_I^\perp, \sigma_{\alpha_j}(\Sigma^J)) \to (X_l^{\vartheta(J)}, 
V_{\vartheta(I)}^\perp, \Sigma^{\vartheta(J)}) \ .\]
\end{theorem}

\np 
\point{\bf{Proof of Theorem \ref{the6.375}.}} 
We begin with some background on representations of Lie algebras.

\np 
Let $W$ be a semisimple finite-dimensional $\gk$-module. It decomposes as
\[W = \oplus_{\lambda \in \supp W} \, W_\lambda \ ,\]
where $W_\lambda := \{ w \in W \, | \, t \cdot w = \lambda(t) 
w {\text{ for every }} t \in \gt_\gk\}$ and
$\supp W = \{\lambda \in \gt_\gk^* \, | \, W_\lambda \neq 0\}$. If $w \in W_\lambda$, 
then the twisted action of $t \in \gt_\gk$ on $w$ is given by
\[t \cdot_\vartheta w = \vartheta(t) \cdot w = \lambda(\vartheta(t)) w = \vartheta^*(\lambda)(t) w \ ,\]
i.e., $w \in W^\vartheta_{\vartheta^*(\lambda)}$. Hence
\[W^\vartheta = \oplus_{\lambda \in \supp W} \, W^\vartheta_{\vartheta^*(\lambda)} \quad \text{and} \quad
\supp W^\vartheta = \vartheta^*(\supp W) \ .\]
On the other hand,
\[W^* = \oplus_{\lambda \in \supp W} (W_\lambda)^* = 
\oplus_{\lambda \in \supp W} \, W^*_{-\lambda} \]
shows that $\supp W^* = - \supp W$. 
Using that $W^\vartheta \cong W^*$ we conclude that $\vartheta^*(\supp W) = - \supp W$.

\np 
Next we consider the module $W$ above as a module over $\gs$ and $\tilde{\gs}:= \vartheta (\gs)$.
Assume that, as an $\tilde{\gs}$-module $W$ decomposes as 
\[W = \oplus_{\mu \in \supp_{\tilde{\gs}} W} \, W_\mu \quad {\text{with}} \quad 
\supp_{\tilde{\gs}} W \subset \gt_{\tilde{\gs}}^* \ . \]
Then $W^\vartheta$ considered as a module over $\gs$ decomposes as
\[W^\vartheta = \oplus_{\mu \in \supp_{\tilde{\gs}} W} \, W^\vartheta_{\vartheta^*(\mu)} \ ,\]
i.e., $\supp_\gs W^\vartheta = \vartheta^*(\supp_{\tilde{\gs}} W)$.
Recalling that $W^\vartheta \cong W^*$ and that, as an $\gs$-module, $W^*$ decomposes as
\[W^* = \oplus_{\mu \in \supp_{\gs} W} \, W^*_{-\mu} \ ,\] 
we conclude that

\begin{equation} \label{eq6.55}
- \supp_{\gs} W = \supp_{\gs} W^* = \supp_{\gs} W^\vartheta = \vartheta^*(\supp_{\tilde{\gs}} W).
\end{equation}

\np 
We are now ready to prove Theorem \ref{the6.375}.

\np 
We have already noted that $\varphi: V_I^\perp \to V_{\vartheta(I)}^\perp$ is an isometry.
To prove that $\varphi$ is a bijection between $X_l^J$ and $X_l^{\vartheta(J)}$, 
we apply the discussion above to the $\gk$-module $\gg$.
Indeed, by definition, $X_l^J = \supp_{\gs} \gg$ and $X_l^{\vartheta(J)} = \supp_{\tilde{\gs}} \gg$,
and \eqref{eq6.55} implies that $X_l^J = -\vartheta^*(X_l^{\vartheta(J)})$. 
Since $\vartheta^*$ is an involution,
we conclude that $\varphi$ is a bijection between $X_l^J$ and  $X_l^{\vartheta(J)}$.

\np 
To complete the proof, we need to show that $\varphi$ is a bijection between 
$\sigma_{\alpha_j}(\Sigma^J)$ and $\Sigma^{\vartheta(J)}$. 
Let $\Sigma = \{\alpha_1, \alpha_2, \ldots, \alpha_l\}$.
Then $\Sigma^J = \{\pi_I(\alpha_k) \, | \, k \in J\}$. Recall that $\bar{\alpha} = \pi_I(\alpha_j)$.
The definition of $\sigma_{\alpha_j}$ implies that
\[\sigma_{\alpha_j}(\Sigma^J) = \{-\bar{\alpha} \} \cup
\{\pi_I(\alpha_k) + i_k \bar{\alpha}\, | \, k \neq j \in J \}   \ ,\]
where, for $k \neq j$, $\alpha_k, \ldots, \alpha_k + i_k \alpha_j$ 
is the $\alpha_j$-string through $\alpha_k$ in $\Delta$.
On the other hand, $\vartheta(J)$ is obtained from $J$ by replacing $j$ by $\vartheta(j)$ and 
 $\vartheta(I)$ is obtained from $I$ by replacing $\vartheta(j)$ by $j$, i.e., 
$\vartheta(J) = (J \backslash \{j\}) \cup \{\vartheta(j)\}$ 
and $\vartheta(I) = (I \backslash \{\vartheta(j)\}) \cup \{j\}$. 
(If $\vartheta(j) = j$, then $\vartheta(J) = J$ and $\vartheta(I) = I$.)
Hence
\[\Sigma^{\vartheta(J)} = \{\tilde{\alpha} := \pi_{\vartheta(I)}(\alpha_{\vartheta(j)})\} 
\cup \{\pi_{\vartheta(I)} (\alpha_k) \, | \, k \neq j \in J\} \ .\]
Note first that
\[\varphi(- \bar{\alpha}) = \vartheta^*(\bar{\alpha}) = \vartheta^*(\pi_I(\alpha_j)) =
\pi_{\vartheta(I)}(\alpha_{\vartheta(j)}) = \tilde{\alpha} \ .\]
Next we show that, for $k \neq j \in J$, 

\begin{equation} \label{eq6.65}
\varphi(\pi_I(\alpha_k) + i_k\bar{\alpha}) = \pi_{\vartheta(I)}(\alpha_k) \ .
\end{equation}
Let $p_1: \Delta/I \to \Delta/K$ and $p_2: \Delta/\vartheta(I) \to \Delta/K$ be the 
natural projections. Then the diagram
\[
\begin{tikzcd}
&  \Delta \arrow[ld, "\pi_I"'] \arrow[dd, "\pi_K"] \arrow[rd, "\pi_{\vartheta(I)}"] &  \\
\Delta/I \arrow[rd, "p_1"'] && \Delta/\vartheta(I) \arrow[ld, "p_2"]\\
& \Delta_K&
\end{tikzcd}
\]
commutes.
Note that the $\bar{\alpha}$-string through $\pi_I(\alpha_k)$ is the 
fibre $p_1^{-1}(\pi_K(\alpha_k))$
and the $\tilde{\alpha}$-string through $\pi_{\vartheta(I)}(\alpha_k)$ is the 
fibre $p_2^{-1}(\pi_K(\alpha_k))$.
The definition of $\varphi$ shows that 
$\varphi(p_1^{-1}(\pi_K(\alpha_k))) = p_1^{-1}(\pi_K(\alpha_k))$,
and by linearity $\varphi$ maps the end points of the first 
string into the end points of the second string.
Finally, as $\varphi(- \bar{\alpha}) = \tilde{\alpha}$, 
it maps the right end of the first string onto the left 
end of the second string, proving \eqref{eq6.65}.
The proof is complete. \hfill $\square$

\np 
\point{\bf The graph $\cX_{l,k}$.} 
For the rest of the section we fix notation expanding on the notation from Section
\ref{sec6.35}.
Given a root system $X_l$, a set $J \subset \{1, 2, \ldots, l\}$, and 
an element $j \in J$, we let $I := J^c = \{1, 2, \ldots, l\} \backslash J$ and
$K = I \cup \{j\}$. Then
we denote by $\vartheta_{J,j}$ the permutation of $\{1,2, \ldots, l\}$ defined by
\[\vartheta_{J,j}(i) = \left\{ \begin{array}{ccl}
\vartheta(i) & {\text{if}} & i \in K\\
i & {\text{if}} & i \not \in K \ ,
\end{array} \right.
\]
where $\vartheta$ is the diagram automorphism defined in Section \ref{sec6.35}.

\np 
Define the graph $\cX_{l,k}$
which records the isomorphisms between rank $k$ quotients of the root system $X_l$ provided by 
Theorem \ref{the6.375}. Namely, the vertices of $\cX_{l,k}$ are the rank $k$ quotients of 
$X_l$, i.e., the GRSs $\{X_l^J \, | \, \# J = k\}$. The distinct vertices $X_l^{J'}$ and $X_l^{J''}$
are linked by an edge if $J'' = \vartheta_{J', j'}(J')$ for some $j' \in J'$. Note that we do
allow loops at vertices $\cX_{l,k}$, i.e., we do not assign 
an edge of the graph to $ \vartheta_{J, j}$ if $J = \vartheta_{J, j}(J)$.

\np 
\begin{example} \label{ex7.88}
Revisiting Example \ref{ex7.65} in which $X_l = E_8$, and $J = \{1,6,8\}$, 
on Figure \ref{fig25} we have displayed $\vartheta_{J,1}$, 
$\vartheta_{J,6}$, and $\vartheta_{J,8}$. Note that
$\vartheta_{J,1} \neq \id_{\{1,2,3,4,5,7\}}$ but $\vartheta_{J,1}(J) = J$; 
$\vartheta_{J,6} = \id_{\{2,3,4,5,6,7\}}$ and thus $\vartheta_{J,6}(6) = J$;
finally, $\vartheta_{J,8}(8) = J' := \{1,6,7\}$. 
\end{example}

\np
The graph $\cX_{l,k}$ carries the information about the ``action'' of $\vartheta_{J,j}$ on $J$. 
Namely, if $J'$ and $J''$ are vertices of $\cX_{l,k}$ linked by and edge, then
$J' \backslash J'' = \{j'\}$, $J'' \backslash J' = \{j''\}$ and 
$J'' = \vartheta_{J', j'}(J')$, $J' = \vartheta_{J'', j''}(J'')$. 
Moreover, if $j \in J$ is such that there is no edge between $J$ and another vertex $J'$
with $J \backslash J' = \{j\}$, then $J = \vartheta_{J, j}(J)$.

\np 
\begin{example} \label{ex6.455} {\bf(i)}
The graph $\cF_{4,2}$ is displayed on Figure \ref{fig:f4}.

\begin{figure}[h!]
\begin{center}
\includegraphics[width=4in]{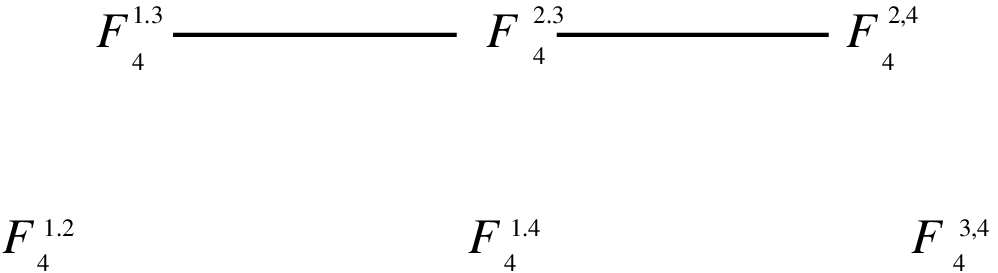}
\caption{The graph $\cF_{4,2}$}
\label{fig:f4}
\end{center}
\end{figure}

\np 
From the graph we recover
\[\begin{array}{cccc}
\vartheta_{\{1,2\},1}(F_4^{1,2}) = F_4^{1,2}, & \vartheta_{\{1,2\},2}(F_4^{1,2}) = F_4^{1,2},
&\vartheta_{\{1,3\},1}(F_4^{1,3}) = F_4^{2,3}, &\vartheta_{\{1,3\},3}(F_4^{1,3}) = F_4^{1,3}, \\
&&&\\
\vartheta_{\{1,4\},1}(F_4^{1,4}) = F_4^{1,4}, &\vartheta_{\{1,4\},4}(F_4^{1,4}) = F_4^{1,4}, & 
\vartheta_{\{2,3\},2}(F_4^{2,3}) = F_4^{1,3}, & \vartheta_{\{2,3\},3}(F_4^{2,3}) = F_4^{2,4}, \\ 
&&&\\
\vartheta_{\{2,4\},2}(F_4^{2,4}) = F_4^{2,4}, &\vartheta_{\{2,4\},4}(F_4^{2,4}) = F_4^{2,3},
&\vartheta_{\{3,4\},3}(F_4^{3,4}) = F_4^{3,4}, &\vartheta_{\{3,4\},4}(F_4^{3,4}) = F_4^{3,4}.
\end{array}
\]

\np 
Theorem \ref{the6.375} implies that the GRSs in each connected component are 
isomorphic among themselves, 
i.e., $F_4^{1,3} \cong F_4^{2,3} \cong F_4^{2,4}$. It is not clear, however whether GRSs 
in different 
connected components are isomorphic or not. As we will establish in Theorem \ref{the7.455} 
below, among the GRSs 
in the three components above there are no equivalences and, in particular, the 6 quotients 
of $F_4$ result in 4
non-isomorphic GRSs. 

\np The next natural question is whether the bases 
$\Sigma^{1,3}, \Sigma^{2,3}, \Sigma^{2,4}$ of the GRS 
$R = F_4^{1,3} \cong F_4^{2,3} \cong F_4^{2,4}$
are isomorphic. Noticing that the highest root  of $R$ in each of these  bases has coordinates
$(2,4), (3,4), (3,2)$ respectively, we conclude that the three bases are pairwise inequivalent and, 
as a consequence, they are pairwise non-isomorphic.

\np 
{\bf (ii)} The connected component of $\cE_{7,3}$ that contains $E_7^{1,4,5}$
is displayed on Figure \ref{fig:e7}.

\begin{figure}[h!]
\begin{center}
\includegraphics[width=4in]{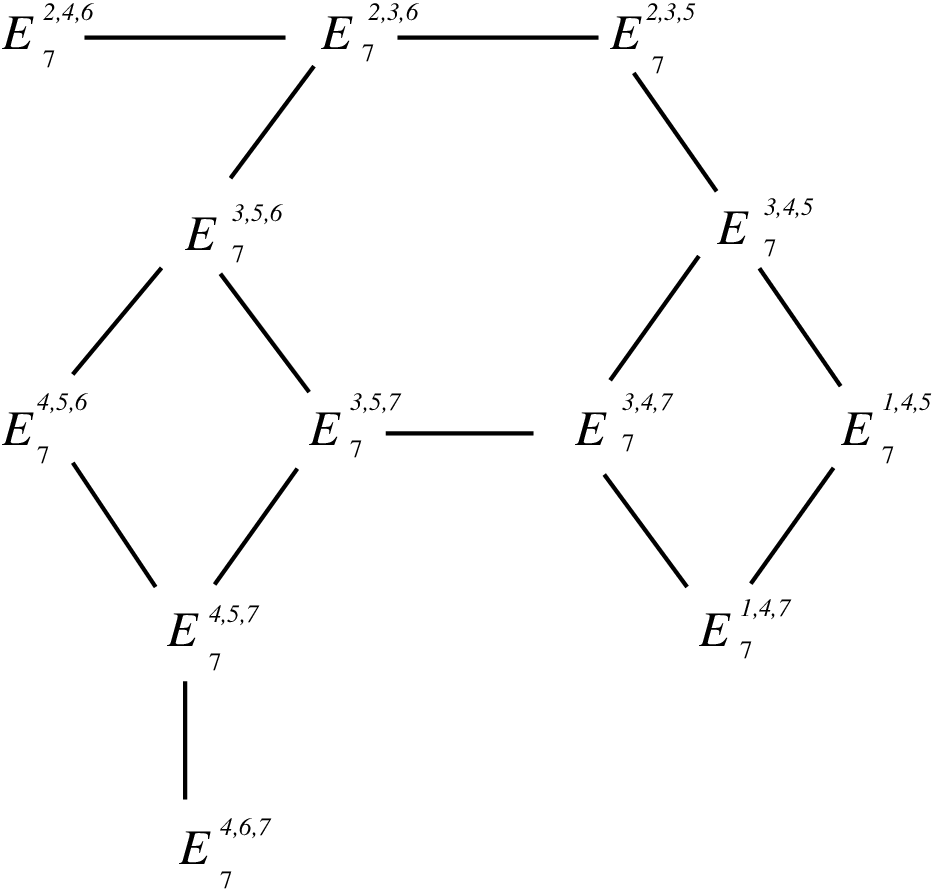}
\caption{Connected component of $\cE_{7,3}$ containing $E_7^{1,4,5}$}
\label{fig:e7}
\end{center}
\end{figure}

\np This component represents a rank 3 GRS $R$ with 28 non-zero roots, see Table \ref{bigtable}.
A simple but cumbersome calculation, which we omit here,
shows that he bases of $R$ corresponding to the vertices of $\cE_{7,3}$ are pairwise inequivalent. 
\end{example}

\np 
\point{\bf Quotients of simply-laced root systems.} 

\np 
\begin{theorem} \label{prop6.61}
Every root system is isomorphic to a quotient of a simply-laced root system.
\end{theorem}

\begin{proof} We have:
\begin{itemize}
    \item[(i)] $B_l \cong D_{l+2}^{1,2,\ldots,l}$\ ;
    \item[(ii)] $C_l \cong D_{2l}^{2,4, \ldots, 2l}$\ ;
    \item[(iii)] $F_4 \cong E_7^{1,3,4,6} \cong E_8^{1,6,7,8}$\ ;
    \item[(iv)] $G_2 \cong E_6^{2,4} \cong E_7^{1,3} \cong E_8^{7,8}$ \ .
\end{itemize}
The isomorphisms (i) and (ii) are established in Section \ref{sec8.15} below and (iii) and (iv) are straightforward calculations.
Note that, for $B_l$ and $C_l$ we have listed just one of many ways these are quotients of $D_N$ for
an appropriate $N$, while for $F_4$ and $G_2$ we have listed all ways in which they 
are quotients of $E_l$.
\end{proof}

\np
\point {\bf Lie superalgebras.} We complete the discussion of quotients of root systems by
showing that the root systems of the contragredient superalgebras $\gs\gl(m|n)$, $m \neq n$,
$\gg\gl(m|m)$, $\go\gs\gp(m, 2n)$, ${\bf{D}}(2,1;\alpha)$, ${\bf{F}}(4)$, and ${\bf{G}}(3)$ are
equivalent to quotients of root systems. Let $\gg$ denote one of the Lie superalgebras above
and let $\gt \subset \gg$ be a toral subalgebra. Under the action of $\gt$,  $\gg$ decomposes as

\begin{equation} \label{eq6.77}
\gg = \oplus_{\nu \in R} \, \gg_\nu \ ,
\end{equation}
where $\gg_\nu = \{ x \in \gg \, | \, [t,x] = \nu(t) x {\text{ for every }} t \in \gt\}$. 
Assume further that $\gc_\gg(\gt) = \gm := \gg_0$. Extending the results of Kostant, \cite{Ko},
in \cite{DF}, we studied the properties of the decomposition \eqref{eq6.77} and the set of
Kostant roots $R$. In particular, we proved that each $\gg_\nu$ is an irreducible $\gm$-module.
A crucial step in the proof was the fact that, for 
$\gg \neq {\bf{D}}(2,1;\alpha), {\bf{G}}(3), {\bf{F}}(4)$,
the root system of $\gg$ is equivalent to a quotient of a root system of a simple Lie algebra.
The cases $\gg = {\bf{D}}(2,1;\alpha), {\bf{G}}(3), {\bf{F}}(4)$ were dealt 
with separately in the Appendix
of \cite{DF}. It turns out, however, that the root systems of 
${\bf{D}}(2,1;\alpha), {\bf{G}}(3), {\bf{F}}(4)$
also are equivalent to quotients of root systems, as shown in Theorem \ref{prop6.77} below.
This observation allows for a uniform treatment in \cite{DF} of all contragredient superalgebras,
making the rather technical Appendix thereof unnecessary. 

\np 
\begin{theorem} \label{prop6.77}
Let $\gg$ be one of the Lie superalgebras $\gs\gl(m|n)$, $m \neq n$,
$\gg\gl(m|m)$, $\go\gs\gp(m, 2n)$, ${\bf{D}}(2,1;\alpha)$, ${\bf{F}}(4)$, or ${\bf{G}}(3)$ 
and let $\Delta$ be the root system of $\gg$. Then $\Delta$ is equivalent to a quotient 
of a root system. 
\end{theorem}

\begin{proof} We have:
\begin{enumerate}
    \item[(i)] $A_{m+n-1}$ or $A_{2m-1}$ if $\gg = \gs\gl(m|n)$, $m \neq n$, 
    or $\gg = \gg\gl(m|m)$, respectively;
    \item[(ii)] $B_{m+2n}^{1,2, \ldots, m, m+2, m+4, \ldots, m+2n}$ if $\gg = \go\gs\gp(2m+1, 2n)$;
    \item[(iii)] $D_{m+2n}^{1,2, \ldots, m, m+2, m+4, \ldots, m+2n}$ if $\gg = \go\gs\gp(2m, 2n)$;
    \item[(iv)] $D_4^{1,2,3}$ if $\gg= {\bf{D}}(2,1;\alpha)$;
    \item[(v)] $E_7^{1,5,6,7}$ if $\gg = {\bf{F}}(4)$;
    \item[(vi)] $E_8^{1,3,8}$ if $\gg = {\bf{G}}(3)$. 
\end{enumerate}
The equivalences (i)--(iii) follow from the lists of quotients of root systems
of types $A_l, B_l, C_l,$ and $D_l$ provided in \cite{DR}, while the equivalences (iv)--(vi)
can be established by a straightforward but somewhat tedious direct calculation which we omit here.
\end{proof}

\np
\section{GRSs of rank 2} \label{sec7.2}

\np
In this section we classify all irreducible GRSs of rank 2 and 
describe their bases
and Cartan matrices. We have already carried out this for the reduced GRSs of rank 2 
in Section \ref{subsec2.455}.
We also show that every GRS of rank 2 is equivalent to a quotient of a root system.

\np
\point {\bf Multipliers of primitive roots.} The first step in the classification
is to find a bound on the multipliers of primitive roots in an irreducible GRS of rank at least 2.

\np
\begin{theorem} \label{the7.15}
Let $R$ be an irreducible GRS of rank at least 2. Then 
$\al \in R$ implies  $5\alpha\notin R$.
\end{theorem}

\begin{proof}
Assume to the contrary, that there is a string $\alpha,2\alpha,3\alpha,\dots,k\alpha\in R$ with 
$k\geq5$. Without loss of generality we may assume that $\alpha$ is 
primitive and $(k+1) \alpha \not \in R$.
Let $\beta \in R$ be a root which is neither orthogonal nor parallel to $\alpha$.
Replacing $R$ by its subsystem $R \cap \Span\{\alpha, \beta\}$ we may 
further assume that the rank of
$R$ equals 2; furthermore, replacing $\beta$ by another root in the rank 2 GRS $R$, we assume that
$S=\{\alpha,\beta\}$ is a base of $R$.

\np 
Let $\beta, \beta + \alpha, \ldots, \beta + \ell \alpha$ be the 
$\alpha$-string through $\beta$ and let $i$ be the largest integer such that
$\langle \beta + (i-1) \alpha, \alpha \rangle <0$. Then 
$\beta + (i+1) \alpha = (\beta + (i-1) \alpha) + 2 \alpha \in R$
and $\langle \beta + (i+1) \alpha, \alpha \rangle >0$. Since $-k \alpha \in R$,
we obtain that $(\beta + (i+1) \alpha) - k \alpha = \beta + (i+1 - k) \alpha \in R$,
concluding that $i+1 - k \geq 0$, or $i \geq k-1$.
Similarly, $(\beta + (i-1) \alpha) + k \alpha = \beta + (i-1 + k) \alpha \in R$,
implying that $\ell \geq i-1 + k \geq 2k -2$. In other words, the
$\alpha$-string through $\beta$ consists of at least $2k-1$ roots.

\np
Consider the roots $\beta + \alpha$ and 
$\beta + (\ell -1) \alpha$. Their difference, $(\ell -2) \alpha$ does not belong to $R$ because
$\ell -2 \geq 2k - 4 > k$. Thus $\langle \beta + \alpha, \beta + (\ell -1) \alpha \rangle \leq 0$
and, consequently, $\langle \beta , \beta + (\ell -1) \alpha \rangle <0$, 
$\langle \beta + \alpha, \beta + \ell \alpha \rangle < 0$, and 
$\langle \beta, \beta + \ell \alpha \rangle < 0$. Hence 
$2 \beta + (\ell-1) \alpha, 2 \beta + \ell \alpha, 2 \beta + (\ell +1) \alpha \in R$.
Starting with $2 \beta + (\ell-1) \alpha, 2 \beta + \ell \alpha, 2 \beta + (\ell +1) \alpha \in R$
and repeating the argument above we conclude that the $\alpha$-string 
through $2 \beta + \ell \alpha$
contains at least $2k-1$ roots and there are at least three roots of the form $4 \beta + j \alpha$.
Consequently, the $\alpha$-string through such a root $4 \beta + j \alpha$ contains at least $2k -1$ roots
and we can continue to get roots of the form $8 \beta + j \alpha$, etc. 
However, this process must terminate
because $R$ is a finite set by the definition of a GRS. This contradiction completes the proof.
\end{proof}

\np 
\point{\bf List of irreducible GRSs of rank 2.} \label{sec7.25}
Up to  equivalence, there are 16 irreducible GRSs of rank 2. The list below is organized 
by the maximal multiplier of a root. We describe the GRSs in terms of a 
basis $\{x, y\}$ of the Euclidean space 
$V$ containing the GRS $R$. The vectors $x$ and $y$ are not necessarily roots. 
The GRSs denoted by 1(i), 1(ii), 1(iii)
are the reduced GRSs discussed in Section \ref{subsec2.455}.

\np 
\begin{itemize}
    \item [1(i)] $R = \{0, \pm x, \pm y, \pm (x + y)\}$, where $\langle x, y \rangle <0$, 
    $\langle x+y, x \rangle > 0$, and $\langle x+y, y \rangle > 0$. This is a two-parameter family
    of isomorphism classes of GRSs, see Section \ref{subsec2.455}.
    \item[1(ii)] $R = \{0, \pm x, \pm y, \pm (y + x), \pm (y + 2x)\}$, where 
    $\langle y, y + 2x \rangle = 0$. This is a one-parameter family
    of isomorphism classes of GRSs, see Section \ref{subsec2.455}.
    \item[1(iii)] $R = \{0, \pm x, \pm y, \pm (y + x), \pm (y + 2x), \pm (y + 3x), \pm (2y + 3x)\}$,
    $\langle x, 2y + 3x \rangle = 0$ and $\langle y, y+ 2 x \rangle =0$. 
    This is the root system $G_2$,
    see Section \ref{subsec2.455}.
    \item[] 
    \item[2(i)] $R = \{0, \pm x, \pm 2x, \pm y, \pm y \pm x\}$, where $x \perp y$, $\|y\| > \| x\|$.
    This is a one-parameter family of isomorphism classes of GRSs.
    \item[2(ii)] $R = \{0, \pm x, \pm 2x, \pm y, \pm 2y, \pm y \pm x\}$, where $x \perp y$.
    This is a one-parameter family of isomorphism classes of GRSs.
    \item[2(iii)] $R = \{0, \pm x, \pm 2x, \pm y \pm \frac{1}{2} x, \pm y \pm \frac{3}{2} x\}$, 
    where $\|y\| = \frac{3}{2} \|x\|$ and
    $|\langle x, y \rangle| < \frac{1}{2} \|x\|^2$.
    This is a one-parameter family of isomorphism classes of GRSs.
    \item[2(iv)] $R = \{0, \pm x, \pm 2x, \pm y 
    \pm \frac{1}{2} x, \pm y \pm \frac{3}{2} x, \pm 2y\}$, 
    where $x \perp y$ and $\frac{\sqrt{3}}{2}\|x\| < \|y\| < \frac{3}{2} \|x\|$.
    This is a one-parameter family of isomorphism classes of GRSs.
    \item[2(v)] $R = \{0, \pm x, \pm 2x, \pm y, \pm y \pm x, \pm y \pm 2 x, \pm 2y\}$,
    where $x \perp y$ and $\|y\| = \sqrt{2} \|x\|$.
    \item[2(vi)] $R = \{0, \pm x, \pm 2x, \pm y  \pm \frac{1}{2} x, 
    \pm y  \pm \frac{3}{2} x, \pm 2y, \pm 2y \pm  x\}$,
    where $x \perp y$ and $\|y\| = \frac{\sqrt{3}}{2} \|x\|$.
    \item[2(vii)] $R = \{0, \pm x, \pm 2x, \pm y, \pm y \pm x, \pm y 
    \pm 2 x, \pm 2y, \pm 2y \pm x\}$,
    where $x \perp y$ and $\|y\| =  \|x\|$.
    \item[2(viii)] $R = \{0, \pm x, \pm 2x, \pm y  \pm \frac{1}{2} x, \pm y  \pm \frac{3}{2} x, 
    \pm y  \pm \frac{5}{2} x, \pm 2y, \pm 2y \pm  x\}$,
    where $x \perp y$ and $\|y\| = \frac{\sqrt{5}}{2} \|x\|$.
    \item[] 
    \item[3(i)] $R = \{0, \pm x, \pm 2x, \pm 3x, \pm y, \pm y \pm x, \pm y \pm 2 x\}$,
    where $x \perp y$ and $\|y\| = 2 \|x\|$.
    \item[3(ii)] $R = \{0, \pm x, \pm 2x, \pm 3x, \pm y, \pm y \pm x, \pm y \pm 2 x, \pm 2y\}$,
    where $x \perp y$ and $\sqrt{2} \|x\| < \|y\| < 2 \|x\|$.
    This is a one-parameter family of isomorphism classes of GRSs.
    \item[3(iii)] $R = \{0, \pm x, \pm 2x, \pm3x, \pm y  \pm \frac{1}{2} x, 
    \pm y  \pm \frac{3}{2} x, 
    \pm y  \pm \frac{5}{2} x,\pm 2y\}$,
    where $x \perp y$ and $\|y\| = \frac{\sqrt{15}}{2} \|x\|$.
    \item[3(iv)] $R = \{0, \pm x, \pm 2x, \pm 3x, \pm y, \pm y \pm x, \pm y \pm 2 x, \pm 2y,
    \pm 2y \pm x, \pm 2y \pm 2 x, \pm 3y\}$,
    where $x \perp y$ and $ \|y\| = \|x\|$.
    \item[] 
    \item[4] $R = \{0, \pm x, \pm 2x, \pm 3x, \pm 4x, \pm y, \pm y \pm x, 
    \pm y \pm 2 x, \pm y \pm 3x, \pm 2y\}$,
    where $x \perp y$ and $\|y\| = \sqrt{6} \|x\|$.
    \end{itemize}
 
\np    
\point{\bf {Classification theorem.}}

\np 
\begin{theorem} \label{the7.25}
Every irreducible GRS of rank 2 is isomorphic to a GRS from one of the 16 families 
in the list above. 
\end{theorem}

\begin{proof}[Idea of proof.] Let $R$ be an irreducible GRS of rank 2 and let $\alpha \in R$ be a
primitive root with a largest multiplier among the multipliers of primitive 
roots in $R$. Fixing a base
$S = \{\alpha, \beta\}$ of $R$, we are then able to use the properties of GRSs 
to arrive at the list above.
The arguments in all cases are similar and carrying out the full classification is somewhat tedious.
We provide the full argument in Section \ref{secApp}. Here we limit ourselves to
the case when $R$ admits a root with multiplier 
equal to 4 which we present below.
\end{proof}

\np 
As an illustration of the type of arguments we use in proving Theorem \ref{the7.25}, we
prove the following proposition.

\np 
\begin{proposition} \label{prop7.35}
Let $R$ be an irreducible GRS of rank 2 which admits a primitive root with multiplier 4.
Then 
\[R = \{0, \pm x, \pm 2x, \pm 3x, \pm 4x, \pm y, \pm y \pm x, 
\pm y \pm 2 x, \pm y \pm 3x, \pm 2y\} \ ,\]
where $x \perp y$ and $\|y\| = \sqrt{6} \|x\|$. 
\end{proposition}

\begin{proof}
Assume $x \in R$ is a primitive root whose multiplier is 4, i.e., $x, 2x, 3x, 4x \in R$ but 
$5x \not \in R$. Let $S=\{x, \beta\}$ be a base of $R$ containing $x$.
We will carry out the proof in two steps: we will prove first that if 
the $x$-string through $\beta$ contains
a root orthogonal to $x$, then $R$ is the GRS described above and then we will
prove that the assumption that the $x$-string through $\beta$ does not contains
a root orthogonal to $x$ leads to a contradiction.

\begin{enumerate}
\item[(i)] {\bf There is a root $y \perp x$ in the $x$-string through $\beta$.}
Since $\langle y \pm x, x\rangle \gtrless 0$, we conclude that 
$y, y \pm x, y \pm 2x, y \pm 3x \in R$.
First we prove that $y \pm 4x \not \in R$. Assume, to the contrary, 
that $y + 4x \in R$ (the case when 
$y - 4x \in R$ is dealt with in a similar way).
Since $(y + 4x) - (y - x) \not \in R$, we have $\langle y +4 x, y - x\rangle \leq 0$,
implying that $\langle y + 4 x, y - 2x\rangle < 0$ and hence $2y + 2x \in R$. 
As above, $2y + 2x \in R$ leads to $2y, 2y \pm x, 2y \pm 2x, 2y \pm 3x \in R$.
Since $(2y +2x) - (2y - 3x) \not \in R$, we have $\langle 2y + 2 x, 2y - 3x\rangle \leq 0$.
Hence $\langle 2y + 3 x, 2y - 3x\rangle < 0$ and $4y = (2y + 3x) + (2y - 3x) \in R$.
Since $y, 2y, 4y \in R$, we conclude that $3y \in R$ and 
Theorem \ref{the7.15} implies that $k y \not \in R$
for $k \geq 5$. If the $x$-string through $3y$ contains elements 
other than $3y$, it would contain at least
7 elements and, repeating the argument above we would have $6 y \in R$, 
contradicting Theorem \ref{the7.15}.
In particular $3y + x \not \in R$, implying that $\langle 2y - 3 x, y +4x\rangle \geq 0$.
Thus $\langle y - 2 x, y +3x\rangle = \frac{1}{2} \langle 2y - 3 x, y +4x\rangle \geq 0$ 
and since $2y + x = (y+3x) + (y-2x) \in R$, we conclude that $5x = (y+3x) - (y-2x) \in R$.
This contradiction completes the proof that $y \pm 4x \not \in R$.

\np 
So far we have established that the $x$-string through $\beta$ is 
$y - 3x, y -2x, y-x, y, y+x, y+2x, y + 3x$;
in particular, $\beta = y - 3x$. As above, we conclude that $2y \in R$. 
Let $R'$ be the subsystem of $R$
obtained by intersecting $R$ with the lattice generated by $x$ and $2y$. 
If $R'$ is reducible, then the 
$x$-string through $2y$ consists of $2y$ alone; if $R'$ is irreducible, 
the argument above shows that
the $x$-string through $2y$ is $2y - 3x, 2y -2x, 2y-x, 2y, 2y+x, 2y+2x, 2y + 3x$. 
Next we show that the later
is impossible. Indeed, assuming that $2y, 2y \pm x, 2y \pm 2x \in R$, 
we repeat the argument above to conclude that
$3y, 4y \in R$ and $3y + i x \not \in R$ for $i \neq 0$. 
Since neither the sum, nor the difference of
$2y-3x$ and $y + 2x$ is a root, we conclude that $\langle 2y - 3 x, y +2x\rangle =  0$ and, for
the same reason $\langle 2y - 3 x, y + x\rangle = 0$. 
These two equations contradict each other, proving 
that the $x$-string through $2y$ consists of $2y$ alone. Finally, 
Corollary \ref{grs-prop} implies that
\[R = \{0, \pm x, \pm 2x, \pm 3x, \pm 4x, \pm y, \pm y \pm x, 
\pm y \pm 2 x, \pm y \pm 3x, \pm 2y\} \ .\] 
The equation $\|y\| = \sqrt{6} \|x\|$ follows from the fact that 
neither the sum nor the difference of
the roots $y + 3x$ and $y - 2x$ is a root and thus $\langle y + 3 x, y - 2x\rangle =  0$.

\item[]

\item[(ii)] 
{\bf Assume that there is no root $y \perp x$ in the $x$-string through $\beta$.}
For brevity we will refer to the roots of the form $k \beta + i x$ for some $i \in \ZZ$ as 
{\it roots at level $k$}. In the arguments above we have essentially established the following
properties of roots at level $k$: 

\begin{enumerate}
    \item[--] If there is a single root at level $k$, then it is orthogonal to $x$;
    \item[--] If there are more than one roots at level $k$, then there are at least 7 roots
    at level $k$;
    \item[--] If there are exactly 7 roots at level $k$, then one of them is orthogonal to $x$;
    \item[--] If there are at least 8 roots at level $k$, then there are at 
    least 7 roots at level $2k$.
\end{enumerate}

\np 
If $R$ does not contain a root orthogonal to $x$ at levels $1, 2, 4, 8, \ldots$, etc., then
$R$ contains at least 8 roots at each of them, making $R$ infinite. Assume $l$ is the smallest
integer such that there is a root orthogonal to $x$ at level $2^l$. Denote this root by $2y$
and consider the subsystem $R'$ of $R$ obtained by intersecting $R$ with the 
lattice generated by $x$
and $2y$. By the first part of the proof, 
\[R' = \{0, \pm x, \pm 2x, \pm 3x, \pm 4x, \pm 2y, \pm 2y \pm x, \pm 2y \pm 2 x, 
\pm 2y \pm 3x, \pm 4y\} \] 
and $2 \|y \| = \sqrt{6} \|x\|$.
The roots of $R$ at level $2^{l-1}$ form a single $x$-string and $2y$ is the sum 
of two roots (in fact several pairs of
roots) at level $2^{l-1}$. Hence 
$y \pm \frac{1}{2} x, y \pm \frac{3}{2} x, y \pm \frac{5}{2} x, y \pm \frac{7}{2} x \in R$.
The inequality $\langle 2y - 3 x, y + \frac{7}{2} x \rangle <  0$ implies 
that $3y + \frac{1}{2} x \in R$.
This means that there are at least 8 roots at level $3 \times 2^{l-1}$ and hence there
are at least 7 roots at level $3 \times 2^l$. However, 
roots at level $3 \times 2^l$ belong to the lattice
generated by $2y$ and $x$ and hence they belong to $R'$. 
The explicit form of $R'$ shows that it contains
no roots at level $3 \times 2^l$. This contradiction completes the proof.
\end{enumerate}
\end{proof}

\np 
\point{\bf{GRSs of rank two are equivalent to quotients of root systems.}}
In the rest of this section we use the notation introduced in Section \ref{Sec55} according to
which $F_4^{2,4}$ denotes the quotient $F_4/\{1,3\}$. We can now revisit Example \ref{ex6.45}.

\np 
\begin{example} \label{ex7.45}
In Example \ref{ex6.45} we established that $F_4^{2,4}$ is equivalent to the GRS
\[R = \{(0,0), \pm (1,0), \pm (0,1), \pm (1,1), \pm (2,0), \pm(1,2), 
\pm(2,1), \pm(2,2), \pm (3,2)\} \ .\]
The largest multiplier of a primitive root in $R$ equals 2, 
hence $R$ is equivalent to one of the GRSs
2(i) -- 2(viii). Since $R$ contains 17 roots, we conclude that $R$ is equivalent to the GRS 2(v).
An explicit equivalence is provided by the assignment $(1,0) \mapsto y$ and $(0,1) \mapsto x-y$.

\np 
There is an additional observation we can make in this example. 
Since the equivalence class 2(v) contains
a unique isomorphism class, we can conclude that $F_4^{2,4}$ is 
actually isomorphic to the GRS 2(v).
In other words, we can recover the Euclidean structure in the vector space spanned by $R$: 
$(1,0) \perp (1,1)$ and $\|(1,0)\| = \sqrt{2} \|(1,1)\|$ which leads to
$\|(0,1)\| = \sqrt{\frac{3}{2}} \|(1,0)\|$ and $\langle (1,0), (0,1) \rangle = - \|(1,0)\|^2$.
This, of course, agrees with the calculation of the Euclidean structure of $F_4^{2,4}$ but we
did not need to perform any calculations to arrive at it.
\end{example}

\np 
Calculating explicitly all rank 2 quotients of exceptional root systems (there are 71 of them) and 
using the list from \cite{DR}, we arrive at the following remarkable result.

\np 
\begin{theorem} \label{the7.455}
Every irreducible GRS of rank 2 is equivalent to a quotient of a root system.
\end{theorem}

\begin{proof}[List]
Here is the exhaustive list of equivalence classes of rank 2 quotients of root systems:
\begin{itemize}
    \item[1(i):] $A_l^{i,j}$ for $1 \leq i < j \leq l$, 
    $D_l^{1,l-1}$, $D_l^{1,l}$, $D_l^{l-1,l}$, $E_6^{1,6}$;
    \item[1(ii):] $B_l^{1,2}$, $C_l^{i,l}$ for $1\leq i < l$, 
    $D_l^{1,2}$, $D_l^{i, j}$ for $2 \leq i \leq l-2$ and
    $l-1 \leq j \leq l$, $E_6^{1,2}$, $E_6^{1,3}$, $E_6^{2,6}$, 
    $E_6^{5,6}$, $E_7^{1,7}$, $E_7^{6,7}$;
    \item[1(iii):] $E_6^{2,4}$, $E_7^{1,3}$, $E_8^{7,8}$, $F_4^{1,2}$, $G_2^{1,2}$; 
    \item[]
    \item[2(i):] $B_l^{i,i+1}$ for $i \geq 2$, $B_l^{1,j}$ 
    for $j \geq 3$, $D_l^{1,j}$ for $3 \leq j \leq l-1$,
    $D_l^{i,i+1}$ for $2 \leq i \leq l-2$, $E_6^{1,5}$, 
    $E_6^{2,3}$, $E_6^{2,5}$, $E_6^{3,6}$, $E_7^{1,2}$, $E_7^{2,7}$;
    \item[2(ii):] $B_l^{i,j}$ for $i \geq 2$ and 
    $j \geq i+2$, $C_l^{i,j}$ for $j < l$, $D_l^{i,j}$ for
    $i \geq 2$ and $i+2 \leq j \leq l-2$, $E_7^{1,6}$, $E_8^{1,8}$, $F_4^{1,4}$;
    \item[2(iii):] $E_6^{1,4}$, $E_6^{3,4}$, $E_6^{3,5}$, $E_6^{4,5}$, $E_6^{4,6}$,
    $E_7^{2,3}$, $E_7^{2,6}$, $E_7^{3,7}$, $E_7^{5,6}$, $E_7^{5,7}$;
    \item[2(iv):] $E_7^{1,5}$, $E_7^{2,4}$, $E_7^{2,5}$, $E_8^{1,2}$, $E_8^{1,3}$, $E_8^{2,8}$;
    \item[2(v):] $E_7^{1,4}$, $E_7^{3,4}$, $E_7^{3,6}$, $E_8^{1,7}$, $E_8^{6,7}$, $E_8^{6,8}$, 
    $F_4^{1,3}$, $F_4^{2,3}$, $F_4^{2,4}$;
    \item[2(vi):] $E_7^{4,6}$, $E_8^{1,6}$, $F_4^{3,4}$;
    \item[2(vii):] $E_8^{1,5}$, $E_8^{2,5}$;
    \item[2(viii):] $E_8^{1,4}$, $E_8^{3,4}$, $E_8^{3,7}$, $E_8^{5,7}$;
    \item[] 
    \item[3(i):] $E_7^{3,5}$, $E_7^{4,5}$, $E_7^{4,7}$;
    \item[3(ii):] $E_8^{2,3}$, $E_8^{2,7}$, $E_8^{3,8}$;
    \item[3(iii):] $E_8^{2,4}$, $E_8^{2,6}$, $E_8^{5,6}$, $E_8^{5,8}$;
    \item[3(iv):] $E_8^{4,6}$, $E_8^{4,7}$;
    \item[]
    \item[4:] $E_8^{3,5}$, $E_8^{3,6}$, $E_8^{4,5}$, $E_8^{4,8}$. \qedhere
\end{itemize}
\end{proof}

\np 
It is natural to ask whether every irreducible GRS of rank two is isomorphic to 
a quotient of a root system. The answer is obviously negative - there are equivalence classes of GRSs of rank two
whose isomorphism classes depend on one or two continuous parameters, while there are only 
countably many quotients of root systems. In the following example we analyze which
isomorphism classes of the GRSs in the equivalence class 1(i) are quotients of root systems.

\np 
\begin{example} \label{ex7.485} 
As noted in the list following Theorem \ref{the7.455}, the GRSs 
in the equivalence class 1(i) are equivalent 
to the quotients $A_l^{i,j}$. A straightforward calculation shows 
that the Cartan matrices of $A_l^{i,j}$, 
$1 \leq i < j \leq l$ and of the isomorphisms classes of 1(i) are respectively
\[\left[ \begin{array}{cc} 2 & -2\, \frac{i}{j} \\ -2\, \frac{l+1-j}{l+1} & 2 \end{array} \right] 
\quad \quad {\text{and}} \quad \quad 
\left[ \begin{array}{cc} 2 & a \\ b & 2 \end{array}\right] \ , -2 < a,b < 0 \ .\]
Proposition \ref{prop2.267} implies that a 1(i) GRS is isomorphic to a quotient $A_l^{i,j}$ if 
and only if the entries of its Cartan matrix are rational. In particular, the isomorphism classes of
quotients of roots systems are dense (in the obvious sense) in the equivalence class 1(i).
We note that the Cartan matrices of the remaining quotients of roots systems equivalent to 1(i) also
have rational entries.
\end{example}

\np 
In contrast to the equivalence class 1(i), the equivalence class 3(ii) contains a one-parameter 
family of isomorphism classes of GRSs but there is a unique isomorphism class of quotients (with 3
inequivalent bases producing 3 different Cartan matrices). 

\np
\point{\bf{A word on classifying irreducible GRSs.}}
The question of classifying all irreducible GRSs, either up to equivalence or up to
isomorphism, remains open. One may hope that the classical approach through
Dynkin diagrams can be applicable. More precisely, one can define an analog of
a Dynkin diagram for a pair $(R,S)$ as follows: 
The nodes are the elements of $S$ and encode the corresponding rank one subsystems. 
The edge between two nodes encodes 
the irreducible rank 2 subsystem corresponding to these nodes.
(If two simple roots are orthogonal, there is no edge between them.) 
Unfortunately, this construction does not determine the equivalence class of $R$,
as the following example shows. \footnote{We thank C. Paquette for pointing out this example.}

\np 
\begin{example} \label{ex7.55}
Let $(R,S) = (D_4^{1,2,3}, \Sigma^{1,2,3} = \{\bar{\alpha}_1, \bar{\alpha}_2, \bar{\alpha}_3\})$. 
Then
\[R^+ =\{0, \bar{\alpha}_1, \bar{\alpha}_2, \bar{\alpha}_3, 
\bar{\alpha}_1+\bar{\alpha}_2, \bar{\alpha}_2+ \bar{\alpha}_3, 
\bar{\alpha}_1+\bar{\alpha}_2+\bar{\alpha}_3, \bar{\alpha}_1+ 2\bar{\alpha}_2+\bar{\alpha}_3\}\ .\]
All rank 1 subsystems are reduced, $\bar{\alpha}_1 \perp \bar{\alpha}_3$, and the rank 2 subsystems
corresponding to the pairs $(\bar{\alpha}_1, \bar{\alpha}_2)$ and $(\bar{\alpha}_2, \bar{\alpha}_3)$
are both equivalent to the GRS 1(i). Hence the ``Dynkin diagram'' proposed above is equivalent
to the Dynkin diagram of the root system $A_3$, however $R$ and $A_3$ are not equivalent.

\np 
On the other hand, the ``Dynkin diagram'' of $R$ is not isomorphic to the Dynkin
diagram of $A_3$ since $\|\bar{\alpha}_1\|^2 = \|\bar{\alpha}_3\|^2 = 2$ but
$\|\bar{\alpha}_2\|^2 = 3/2$. 
\end{example}

\np We do not know whether the analogs of Dynkin diagrams proposed above determine
isomorphism classes of GRSs. However, even if they do, carrying out the complete classification
will be difficult due to the 
large number of possible edges -- up to equivalence, 
there are 16 irreducible GRSs of rank 2 and some of them represent families of isomorphism
classes. Moreover, most of the rank 2 irreducible GRSs 
admit several inequivalent bases, leading to a very large number 
of possible edges. (In the classical case,
there are only 3 types of edges corresponding to the root 
systems $A_2$, $B_2$, and $G_2$.) Moreover,
inequivalent bases of $R$ will produce different graphs and identifying which graphs
correspond to the same GRS will require more work.
Based on the classification of rank 2 GRSs and the similarities between
Table \ref{bigtable} and the table in \cite{CH}, we make the following conjecture.

\np
\begin{conjecture} \label{conj}
Every irreducible GRS of rank at least 2 is  equivalent to a quotient of a root system.
\end{conjecture}

\np 
A positive answer to this conjecture along with the descriptions of all quotient of root
systems in Section \ref{sec8} below
leads to a classification of all GRSs up to  equivalence.

\np
\begin{remark} \label{rem5.95}
Following the posting of a preliminary version of this text, M. Cuntz and B. M\"uhlherr have proved Conjecture \ref{conj},
\cite{CM}.
\end{remark}

\np
\section{Quotients of root systems II: lists and isomorphisms} \label{sec8}

\np 
In this section we provide a complete list of all quotients of root systems and discuss
the corresponding graphs $\cX_{l,k}$ as well as 
the isomorphisms and equivalences between quotients of root systems. 
We continue to use the notation introduced in Section \ref{sec55.5}.

\np 
\point{\bf Quotients of classical roots systems.} \label{sec8.15}
The quotients of the classical root systems were calculated in \cite{DR}. Below we
summarize these results and determine the isomorphism and equivalence classes of $X_l$
for each classical $X_l$.

\np 
{\bf Quotients of $A_l$.} Let $J = \{1 \leq j_1 < j_2 < \ldots < j_k \leq l\}$.
For $1 \leq s \leq k+1$, set $l_s:= j_s - j_{s-1}$, where $j_0:= 0$ and $j_{k+1} := l+1$
and let $L(J) = (l_1, l_2, \ldots, l_{k+1})$. 
Then
\[A_l^J = \{0, \delta_s - \delta_t \, | \, 1 \leq s \neq t \leq k+1 \} \quad
{\text{and}} \quad 
\Sigma^J = \{\delta_s - \delta_{s+1} \, | \, 1 \leq s \leq k \}\ \ ,\]
where 

\begin{equation} \label{eq6.515}
\langle \delta_s , \delta_t \rangle = \left\{ \begin{array}{ccl}
 1/l_s & {\text{if}} & s = t \\
0 & {\text{if}} & s \neq t \ . \end{array} \right. 
\end{equation}
It is clear that $A_l^{J'} \approx A_l^{J''}$ for any $\# J' = \# J'' = k$.
A quick calculation shows that $A_l^{J'} \cong A_l^{J''}$ if and only if
the sequence $L(J'')$ is a permutation of the sequence $L(J')$.

\np 
For $1 \leq s \leq k$, $\vartheta_{J,j_s}(j_s) = j_{s-1} - j_s + j_{s+1}$ and
hence 
\[L(\vartheta_{J, j_s}(J)) = (l_1, \ldots, l_{s-1}, l_{s+1}, l_s, l_{s+2}, \ldots, l_{k+1}) \ ,\]
proving that $A_l^{J'}$ and $A_l^{J''}$ belong to the same connected component of the graph $\cA_{l,k}$
if and only if the sequence $L(J'')$ is a permutation of the sequence $L(J')$. Comparing
with the above, we conclude that $A_l^{J'} \cong A_l^{J''}$ if and only if 
$A_l^{J'}$ and $A_l^{J''}$ belong to the same connected component of $\cA_{l,k}$.

\np 
Turning to bases, we see that $\Sigma_l^{J'} \approx \Sigma_l^{J''}$ for any $\# J' = \# J'' = k$,
while $\Sigma_l^{J'} \cong \Sigma_l^{J''}$ if and only if $L(J') = L(J'')$
or $L(J') = w_o (L(J''))$, where $w_o$ is the longest element of the symmetric group on 
$k+1$ elements,
i.e., the permutation that reverses sequences of length $k+1$.

\np 
In summary: All vertices of $\cA_{l,k}$ belong to the same equivalence class of GRSs; 
the isomorphism classes of GRSs are exactly the connected components of $\cA_{l,k}$;
all bases $\Sigma^J$ are equivalent among themselves; and isomorphism classes of bases
correspond to the orbits of the group generated by $w_o$ (isomorphic to $\ZZ_2$)
acting on sequences of length $k+1$ whose elements add up to $l+1$.

\np 
{\bf Quotients of $B_l$.} Let $J = \{1 \leq j_1 < j_2 < \ldots < j_k \leq l\}$.
For $1 \leq s \leq k$, set $l_s:= j_s - j_{s-1}$, where $j_0:= 0$
and let $L(J) = (l_1, l_2, \ldots, l_{k})$.
Then
\[B_l^J = \{0, \pm \delta_s \pm \delta_t, \pm \delta_s \, | \, 1 \leq s < t \leq k \}  
\cup \{\pm \delta_{2s} \, | \, l_s \geq 2 \} \quad {\text{and}} \quad 
\Sigma^J = \{\delta_s - \delta_{s+1}, \delta_k \, | \, 1 \leq s \leq k-1 \} \ ,\]
subject to \eqref{eq6.515}.

\np
It is clear that $B_l^{J'} \approx B_l^{J''}$ if and only if $1$ appears the same number of times 
in $L(J')$ and $L(J'')$.
A quick calculation shows that $B_l^{J'} \cong B_l^{J''}$ if and only if
the sequence $L(J'')$ is a permutation of the sequence $L(J')$ or 
$L(J'')$ is proportional to a permutation of $L(J')$ and neither of them contains $1$.

\np 
For $1 \leq s \leq k-1$, $\vartheta_{J,j_s}(j_s) = j_{s-1} - j_s + j_{s+1}$ and 
$\vartheta_{J, j_k}(j_k) = j_k$.
Hence 

\begin{equation} \label{eq6.85}
L(\vartheta_{J, j_s}(J)) = \left\{ 
\begin{array}{ccc}
(l_1, \ldots, l_{s-1}, l_{s+1}, l_s, l_{s+2}, \ldots, l_{k}) & {\text{for}} & 1 \leq s \leq k-1\\
(l_1, \ldots, l_{s-1}, l_{s}, l_{s+1}, l_{s+2}, \ldots, l_{k}) & {\text{for}} &  s = k \ ,
\end{array} \right.
\end{equation}
proving that $B_l^{J'}$ and $B_l^{J''}$ belong to the same connected component of $\cB_{l,k}$
if and only if the sequence $L(J'')$ is a permutation of the sequence $L(J')$. Comparing
with the above, we conclude that $B_l^{J'}$ and $B_l^{J''}$ may be isomorphic even if they
do not belong to the same connected component of $\cB_{l,k}$. For example,
$B_{10}^{2,6} \cong B_{10}^{6,9}$ because the corresponding sequence $L(J'')$, namely $(6,3)$, 
is proportional to a permutation of the sequence $L(J')$, namely $(2,4)$, and neither contains $1$. 

\np 
Turning to bases, we see that $\Sigma_l^{J'} \approx \Sigma_l^{J''}$ for any $\# J' = \# J'' = k$,
while $\Sigma_l^{J'} \cong \Sigma_l^{J''}$ if and only if $L(J') = L(J'')$
or $L(J')$ and $L(J'')$ are proportional and neither contains $1$.
Note that, if $L(J')$ and $L(J'')$ are proportional and $L(J')$ contains $1$ 
(implying that $L(J'')$ does not),
then there is a conformal linear map between the respective Euclidean spaces which is a bijection 
between $\Sigma_l^{J'}$ and $\Sigma_l^{J''}$. However it is not a bijection between $B_l^{J'}$ and 
$B_l^{J''}$, so it is not an isomorphism between $\Sigma^{J'}$ and $\Sigma^{J''}$.

\np 
In summary: There are $k+1$ equivalence classes of vertices of $\cB_{l,k}$ depending
on the number of times $1$ appears in the respective sequences $L(J)$; 
there are isomorphism classes of GRSs that consist of more than one connected 
component of $\cB_{l,k}$;
all bases $\Sigma^J$ are equivalent among themselves; and isomorphism classes of 
bases are described above.

\np 
{\bf Quotients of $C_l$.} As in the case of $B_l$, let 
$J = \{1 \leq j_1 < j_2 < \ldots < j_k \leq l\}$.
For $1 \leq s \leq k$, set $l_s:= j_s - j_{s-1}$, where $j_0:= 0$
and let $L(J) = (l_1, l_2, \ldots, l_{k})$.
Then
\[C_l^J =\left\{ \begin{array}{ccl}
\{0, \pm \delta_s \pm \delta_t,  \pm 2 \delta_s \, | \, 1 \leq s < t \leq k \}  
& {\text{if}} & j_k = l \\
\{0, \pm \delta_s \pm \delta_t, \pm \delta_s , \pm 2 \delta_s  \, | \, 1 \leq s < t \leq k \}  
& {\text{if}} & j_k < l 
\end{array} \right. \]
and
\[\Sigma^J = \left\{ \begin{array}{ccl}
\{\delta_s - \delta_{s+1}, 2 \delta_k \, | \, 1 \leq s \leq k-1 \} & {\text{if}} & j_k = l \\
\{\delta_s - \delta_{s+1},  \delta_k \, | \, 1 \leq s \leq k-1 \} & {\text{if}} & j_k < l \ ,
\end{array} \right.
\]
subject to \eqref{eq6.515}.
Following, \cite{DR}, we refer to the cases $j_k = l$ and $j_k < l$ as Type I and 
Type II respectively.

\np
It is clear that $C_l^{J'} \approx C_l^{J''}$ if and only if $J'$ and $J''$ are of the same type.
If $J'$ and $J''$ are of Type I, then $C_l^{J'} \cong C_l^{J''}$ if and only if
the sequence $L(J'')$ is a permutation of the sequence $L(J')$;
if $J'$ and $J''$ are of Type II, then $C_l^{J'} \cong C_l^{J''}$ if and only if
$L(J'')$ is proportional to a permutation of $L(J')$.

\np 
As in the case of $B_l$,
for $1 \leq s \leq k-1$, $\vartheta_{J,j_s}(j_s) = j_{s-1} - j_s + j_{s+1}$,
$\vartheta_{J, j_k}(j_k) = j_k$, and
$B_l$, $L(\vartheta_{J, j_s}(J))$ is given by \eqref{eq6.85},
proving that $C_l^{J'}$ and $C_l^{J''}$ belong to the same connected component of $\cC_{l,k}$
if and only if the sequence $L(J'')$ is a permutation of the sequence $L(J')$. As a consequence,
in Type I, the connected components of $\cC_{l,k}$ are isomorphism classes. 
In Type II, as in the case of 
$B_l$, $C_l^{J'}$ and $C_l^{J''}$ may be isomorphic even if they
do not belong to the same connected component of $\cC_{l,k}$.  

\np 
Turning to bases, we see that the two types are exactly the equivalence classes of bases. 
Moreover, $\Sigma_l^{J'} \cong \Sigma_l^{J''}$ if and only if $L(J') = L(J'')$ in Type I and
$L(J')$ and $L(J'')$ are proportional in Type II.

\np 
In summary: There are two equivalence classes of vertices of $\cC_{l,k}$ and two equivalence 
classes of bases $\Sigma^J$ depending on their types; in Type II
there are isomorphism classes of GRSs that consist of more than one connected 
component of $\cC_{l,k}$;
and isomorphism classes of bases are described above.

\np 
{\bf Quotients of $D_l$.} As in the cases of $B_l$ and $C_l$, 
let $J = \{1 \leq j_1 < j_2 < \ldots < j_k \leq l\}$.
For $1 \leq s \leq k-1$, set $l_s:= j_s - j_{s-1}$, where $j_0:= 0$. Unlike the cases
$B_l$ and $C_l$, set
\[ 
l_k := \left\{ \begin{array}{cc}
l - j_{k-1} & {\text{if}} \quad j_{k-1} \leq l-2 < j_k \\
j_k - j_{k-1} & {\text{otherwise}} 
\end{array} \right.
\]
and let $L(J) = (l_1, l_2, \ldots, l_{k})$.
Then
\[D_l^J =\left\{ \begin{array}{ccl}
\{0, \pm \delta_s \pm \delta_t \, | \, 1 \leq s < t \leq k \} \cup 
\{\pm 2 \delta_s \, | \, l_s \geq 2\} 
& {\text{if}} & j_k \geq  l -1 \\
\{0, \pm \delta_s \pm \delta_t, \pm \delta_s \, | \, 1 \leq s < t \leq k \} \cup 
\{\pm 2 \delta_s \, | \, l_s \geq 2\}  
& {\text{if}} & j_k \leq l-2 
\end{array} \right. \]
and
\[\Sigma^J = \left\{ \begin{array}{ccc}
\{\delta_s - \delta_{s+1}, \delta_{k-1} + \delta_k \, 
| \, 1 \leq s \leq k-1 \} & {\text{if}} &  j_{k-1} = l-1, j_k = l \\
\{\delta_s - \delta_{s+1}, 2 \delta_k \, | \, 1 \leq s \leq k-1 \} & {\text{if}} & 
j_{k-1} \leq l-2< j_k \\
\{\delta_s - \delta_{s+1},  \delta_k \, | \, 1 \leq s \leq k-1 \} & {\text{if}} & j_k \leq l-2 \ ,
\end{array} \right.
\]
subject to \eqref{eq6.515}.
We refer to the cases $j_k \geq l-1$ and $j_k \leq l-2$ as Type I and Type II respectively.

\np
It is clear that $D_l^{J'} \approx D_l^{J''}$ if and only if $J'$ and $J''$ are of the same type
and 1 appears the same number of times in $L(J')$ and $L(J'')$.
Further, $D_l^{J'} \cong D_l^{J''}$ if and only if $D_l^{J'} \approx D_l^{J''}$
and $L(J'')$ is proportional to a permutation of $L(J')$. Note that, in Type I,
$l_1 + \ldots + l_k = l$ implies that 
$L(J'')$ is proportional to a permutation of $L(J')$ 
if and only if $L(J'')$ is a permutation of $L(J')$.

\np 
Computing $\vartheta_{J,j_s}(j_s)$ is more complicated. Here is the result.
\begin{enumerate}
    \item[(i)] If $s \leq k-2$, then $\vartheta_{J,j_s}(j_s) = j_{s-1} - j_s + j_{s+1}$;
    \item[(ii)] If $s = k-1$, then
    \begin{enumerate}
        \item[(a)] $j_{k} \leq l-2 \implies \vartheta_{J,j_{k-1}}(j_{k-1}) 
        = j_{k-2} - j_{k-1} + j_{k}$,
        \item[(b)] $j_{k-2} + 2 \leq j_{k-1} \leq l-2 < j_k \implies \vartheta_{J,j_{k-1}}(j_{k-1}) 
        = j_{k-2} - j_{k-1} + l$,
        \item[(c)] $j_{k-2} + 1 = j_{k-1} \leq l-2 < j_k \implies 
        \vartheta_{J,j_{k-1}}(j_{k-1}) =  \left\{ 
        \begin{array}{ccl}
        l & {\text{if}} & j_k = l-1\\
        l-1 & {\text{if}} & j_k = l \ ,
        \end{array}\right.$
        \item[(d)] $j_{k-1} = l-1, j_k = l \implies 
        \vartheta_{J,j_{k-1}}(j_{k-1}) = j_{k-2} +1$;
    \end{enumerate}
    \item[(iii)] If $s = k$, then
    \begin{enumerate}
        \item[(a)] $j_{k} \leq l-2 \implies \vartheta_{J,j_{k}}(j_{k}) = j_k$,
        \item[(b)] $j_k = l-1 \implies 
        \vartheta_{J,j_{k}}(j_{k}) = \left\{ 
        \begin{array}{ccl}
        l & {\text{if}} & l - j_{k-1} {\text{ is odd}}\\
        l-1 & {\text{if}} & l - j_{k-1} {\text { is even}} \ ,
        \end{array}\right.$
        \item[(c)] $j_{k-1} \leq l-2, j_k = l \implies 
        \vartheta_{J,j_{k}}(j_{k}) = \left\{ 
        \begin{array}{ccl}
        l-1 & {\text{if}} & l - j_{k-1} {\text{ is odd}}\\
        l & {\text{if}} & l - j_{k-1} {\text { is even}} \ ,
        \end{array}\right.$
        \item[(d)] $j_{k-1} = l-1, j_k = l \implies \vartheta_{J,j_{k}}(j_{k}) = j_{k-2}+1$.
    \end{enumerate}
\end{enumerate}
Despite the complicated formulas for $\vartheta_{J,j_s}(j_s)$, the 
formula for $L(\vartheta_{J, j_s}(J))$ is very similar to \eqref{eq6.85}, namely
\[L(\vartheta_{J, j_s}(J)) = \left\{ 
\begin{array}{cc}
(l_1, \ldots, l_{s-1}, l_{s+1}, l_s, l_{s+2}, \ldots, l_{k}) & {\text{if }} 
1 \leq s \leq k-1 {\text{ or }} s = k, j_{k-1} = l-1, j_k = l\\
(l_1, \ldots, l_{s-1}, l_{s}, l_{s+1}, l_{s+2}, \ldots, l_{k}) & {\text{otherwise}}  \ ,
\end{array} \right.\]
proving that $D_l^{J'}$ and $D_l^{J''}$ belong to the same connected component of $\cC_{l,k}$
if and only if the sequence $L(J'')$ is a permutation of the sequence $L(J')$. As a consequence,
in Type I, the connected components of $\cD_{l,k}$ are isomorphism classes. 
In Type II, as in the cases of 
$B_l$ and $C_l$, $D_l^{J'}$ and $D_l^{J''}$ may be isomorphic even if they
do not belong to the same connected component of $\cD_{l,k}$.  

\np 
Turning to bases, we see that there are three equivalence classes of bases. 
Moreover, $\Sigma_l^{J'} \cong \Sigma_l^{J''}$ if and only if 
$\Sigma_l^{J'} \approx \Sigma_l^{J''}$ and 
\begin{enumerate}
    \item[(i)] $L(J') = L(J'')$ or $L(J') = \tau (L(J''))$ if $j_{k-1} = l-1, j_k = l$ 
    and $\tau$ transposes the last two coordinates of a sequence;
    \item[(ii)] $L(J') = L(J'')$ if $j_{k-1} \leq l-2 < j_k$;
    \item[(iii)] $L(J')$ and $L(J'')$ are proportional if $j_k \leq l-2$.
\end{enumerate}

\np 
In summary: The equivalence classes of vertices of $\cD_{l,k}$ 
depend on the type of the corresponding index set $J$ and the number of times
$1$ appears in $L(J)$; there are three equivalence 
classes of bases $\Sigma^J$;
there are isomorphism classes of GRSs that consist of more than one connected 
component of $\cD_{l,k}$;
and isomorphism classes of bases are described above.

\np 
\point{\bf Quotients of exceptional roots systems.} \label{sec8.25}
Table \ref{bigtable} below summarizes the information about all quotients of exceptional root systems
of rank at least 2. The reason for not including quotients of rank 1 is that, up to isomorphism,  
every GRSs of rank 1 is determined by its cardinality, cf. Section \ref{sec2.21}. Moreover, the number of nonzero roots of
$X_l^j$ for $1 \leq j \leq l$ is twice the coefficient of $\alpha_j$ in the highest root of $X_l$.
Thus all the information about the rank 1 quotients of root systems is contained in the 
expressions of the highest roots of the respective root systems in terms of simple roots.

\np 
Table \ref{bigtable} contains the following columns:
\begin{enumerate}
\item the rank $k$ of the quotient;
\item the root system $X_l$ whose quotients are discussed;
\item the number of connected components of the graph $\cX_{l,k}$;
\item the column ``name'' which contains a label for each connected component of $\cX_{l,k}$;
\item the column denoted ``$J$ or $J^c$'' which contains all quotients of $X_l$ that belong
to the respective component of $\cX_{l,k}$. We list each set $J$ (or $J^c$) as a $k$-digit (or $l-k$-digit) 
number and separate
all sets $J$ belonging to the same component by commas. For example, the connected component of $\cF_{4,2}$
which contains $F_4^{1,3}$ is denoted by $\cF_{4,2}^{II}$ and contains the sets $\{1,3\}$, 
$\{2,3\}$, and $\{2,4\}$ which is reflected by the third row of the table reading $13,23,24$.
In order to save space and improve readability, for $k \geq 5$ we encode the complements $J^c$
of the sets $J$. Since the unique quotient of $X_l$ of rank $l$ is $X_l$ itself, so we record 
``$X_l$'' instead of writing $12\ldots l$ for $J$. Finally, since $\cE_{l,l-1}$ is connected,
we do not list all $J$ in this case;
\item the column denoted $|R| - 1$ which contains the number of nonzero roots in the GRS $R$
corresponding to the respective component of $\cX_{l,k}$.
\end{enumerate}

\np 
\point{\bf Isomorphisms and equivalences between quotients of root systems.} \label{sec8.35}
In this section we list all isomorphisms and equivalences between quotients of 
two root systems, one of which is exceptional. Since the information presented in 
Section \ref{sec8.35} is sufficient to deduce when two quotients of
a classical root system are isomorphic or equivalent, we will not comment on this case. 
Furthermore, when a quotient of an exceptional root system is isomorphic or equivalent to a quotient of
a classical root system, we will list just one quotient of a classical root system as an example.

\np
{\bf Rank 2.} Recall that Theorem \ref{the6.375} implies that each connected component of $X_{l,k}$ consists 
of quotients of $X_l$ which are isomorphic among themselves. 
Theorem \ref{the7.455} lists all equivalences between rank 2 quotients of root systems. 
Comparing the lists provided by Theorem \ref{the7.455} with Table \ref{bigtable}, we
notice that, for exceptional $X_l$, each connected component of $\cX_{l,2}$ represents a different equivalence
class of a rank 2 GRS. Theorem \ref{the7.455} implies the following statement.

\np 
\begin{proposition} \label{prop8.35}
The equivalences between
rank 2 quotients of root systems (when at least one of them is exceptional) are:
\begin{itemize}
    \item[(i)] $\cE_{6,2}^{IV} \approx A_2 \approx$ 1(i);
    \item[(ii)]  $\cE_{6,2}^{I} \approx \cE_{7,2}^{VI} \approx B_2 \approx$ 1(ii);
    \item[(iii)] $\cE_{6,2}^{V} \approx \cE_{7,2}^{II} 
    \approx \cE_{8,2}^{XI} \approx \cF_{4,2}^{I} \approx G_2 \approx$ 1(iii);
    \item[(iv)] $\cE_{6,2}^{III} \approx \cE_{7,2}^{I} \approx B_3^{1,3} \approx$ 2(i);
    \item[(v)] $\cE_{7,2}^{V} \approx \cE_{8,2}^{VI} \approx \cF_{4,2}^{III} \approx B_4^{2,4} \approx$ 2(ii);
    \item[(vi)]  $\cE_{6,2}^{II} \approx \cE_{7,2}^{VII} \approx$ 2(iii);
    \item[(vii)]  $\cE_{7,2}^{IV} \approx \cE_{8,2}^{I} \approx$ 2(iv);
    \item[(viii)] $\cE_{7,2}^{III} \approx \cE_{8,2}^{V} \approx \cF_{4,2}^{II}  \approx$ 2(v);
    \item[(ix)] $\cE_{7,2}^{IX} \approx \cE_{8,2}^{IV} \approx \cF_{4,2}^{IV}  \approx$ 2(vi). \hfill $\square$
\end{itemize}
\end{proposition}

\np 
Turning to the question of isomorphisms, first we note that since each of the equivalence classes of GRSs 
1(iii), 2(v), and 2(vi) contains a unique isomorphism class of GRSs, cf. Section \ref{sec7.25}, we have the isomorphisms

\begin{equation} \label{eq8.33}
    \cE_{6,2}^{V} \cong \cE_{7,2}^{II} \cong \cE_{8,2}^{XI} \cong \cF_{4,2}^{I} \cong G_2, \quad 
\cE_{7,2}^{III} \cong \cE_{8,2}^{V} \cong \cF_{4,2}^{II}, \quad 
\cE_{7,2}^{IX} \cong \cE_{8,2}^{IV} \cong \cF_{4,2}^{IV} \ .
\end{equation}
Note that the isomorphisms \eqref{eq8.33} also follow from the isomorphisms provided by Theorem \ref{prop6.61}.
In fact, Theorem \ref{prop6.61} implies also

\begin{equation} \label{eq8.34}
    \cE_{7,2}^{V} \cong \cE_{8,2}^{VI} \cong \cF_{4,2}^{III} \ .
\end{equation}
Moreover, these are almost all the isomorphisms between equivalence classes listed in Proposition \ref{prop8.35}.

\np 
\begin{proposition} \label{prop8.37}
Among the equivalences listed in Proposition \ref{prop8.35}, the isomorphisms are provided by \eqref{eq8.33}, 
\eqref{eq8.34}, and $\cE_{6,2}^{IV} \cong A_2$.
\end{proposition}

\begin{proof}
The proof is carried out by an explicit calculation which we omit here.
\end{proof}

\np 
{\bf Rank 3.} The isomorphisms provided by Theorem \ref{prop6.61} imply the following isomorphisms of
rank 3 quotients:

\begin{equation} \label{eq8.37}
    \cE_{7,3}^{IV} \cong \cE_{8,3}^{VII} \cong \cF_{4,3}^{I} \quad {\text{and}} \quad 
    \cE_{7,3}^{VI} \cong \cE_{8,3}^{VI} \cong \cF_{4,3}^{II} \ .
\end{equation}
Comparing the cardinality of different rank 3 quotients in Table \ref{bigtable}, we see that 
the only possible equivalencies other that \eqref{eq8.37} are $\cE_{7,3}^{III} \approx \cE_{7,3}^{IV}$
and $\cE_{6,3}^{II} \approx \cE_{7,3}^{I}$. A direct calculation shows that
$\cE_{7,3}^{III} \not \approx \cE_{7,3}^{IV}$ and  $\cE_{6,3}^{II} \approx \cE_{7,3}^{I}$ 
but $\cE_{6,3}^{II} \not \cong \cE_{7,3}^{I}$. With some additional work one can prove the following result.

\np 
\begin{proposition} \label{prop8.39}
The isomorphisms  between
rank 3 quotients of root systems (when at least one of them is exceptional) are:
\begin{itemize}
    \item[(i)] $\cE_{6,3}^{I} \cong D_8^{3,5,8}$;
    \item[(ii)]  $\cE_{7,3}^{VII} \cong C_3$;
    \item[(iii)] $\cE_{7,3}^{IV} \cong \cE_{8,3}^{VII} \cong \cF_{4,3}^{I}$; 
    \item[(iv)] $\cE_{7,3}^{VI} \cong \cE_{8,3}^{VI} \cong \cF_{4,3}^{II}$.
\end{itemize}
In addition, we have one equivalence which is not an isomorphism:
\begin{itemize}
    \item[(v)] $\cE_{6,3}^{II} \approx \cE_{7,3}^{I}$. \hfill $\square$
\end{itemize}
\end{proposition}

\np 
{\bf Rank greater than 3.} The information provided by Table \ref{bigtable} allows us to conclude 
that the only equivalence between rank greater than 3  quotients of root systems  (when 
at least one of them is exceptional) is the isomorphism
\[ \cE_{7,4}^{IV} \cong \cE_{8,4}^{VI} \cong F_4\]
provided by Theorem \ref{prop6.61}. 

\section{Appendix: Proof of Theorem \ref{the7.25}} \label{secApp}

\np
Below we provide the details of the proof 
of Theorem \ref{the7.25} that we omitted in Section \ref{sec7.2}. 

\np 
Let $R$ be an irreducible GRS of rank 2.  Throughout this section we fix  a primitive root  
$x \in R$ with a largest multiplier $m$ among the multipliers of primitive roots in $R$ and
a base $S = \{x, \beta\}$ of $R$. If $\gamma \in R$, we define {\it the level of $\gamma$} as the coefficient of $\beta$ in the decomposition of $\gamma$ as a linear combination of 
$\beta$ and $x$. The roots at any level form a single $x$-string, e.g. the roots at level 1 are exactly the roots in the $x$-string through $\beta$.

\np
The proof of Theorem \ref{the7.25} is carried out by classifying separately the irreducible GRSs of rank 2 with a fixed largest multiplier $m$.  Theorem \ref{the7.15} implies that $1 \leq m \leq 4$.
The GRSs with $m=1$ and $m=4$ are classified respectively in Section \ref{subsec2.455} and Proposition \ref{prop7.35}. Below we consider the cases $m = 3$ and $m = 2$.

\np 
\point{\bf{GRSs with largest multiplier $m = 3$.}}

\np
We have $x, 2x, 3x \in R$ but $4x \not \in R$.
We consider two cases, depending on whether there is a root $y$ orthogonal to $x$ at level 1.

\np 
\ppoint{\bf{There is a root $y \perp x$ at level 1.}} \label{sec7.1.1}
Since $ R$ is irreducible, $\beta \neq y$ and hence $y \pm x \in R$. Next, 
$\langle y + x, 3 x \rangle>0$ and $\langle y - x, 3 x \rangle<0$ imply that 
$y - 2x \in R$ and $y + 2x \in R$. Since $4x = (y + 2x) - (y - 2x) \not \in R$,
we conclude that $\langle y+2x , y - 2x \rangle\leq0$.

\np
Here are two subcases:

\np 
{\bf A.} \underline{$\langle y+2x , y - 2x \rangle = 0$}. In this case $\|y\| = 2 \|x\|$. Since $4x = (y + 2x) - (y - 2x) \not \in R$,
we have $2y = (y + 2x) + (y - 2x) \not \in R$. Next we notice that
\[\langle y-x, y + 3x \rangle = \| y\|^2 - 3 \|x\|^2 = \|x\|^2 > 0 \ .\]
This inequality implies that $y + 3x \not \in R$ because otherwise $4x = (y + 3x) - (y - x) \in R$
which is false. Similarly, $y - 3x \not \in R$. Since none of the vectors $2y, y \pm 3x$ 
belongs to $R$, Corollary \ref{grs-prop} implies that 
\[R = \{0, \pm x, \pm 2x, \pm 3x, \pm y, \pm y \pm x,  \pm y \pm 2x\} \]
and $\|y\| = 2 \|x\|$, i.e., $R$ is the GRS 3(i).

\np 
{\bf B.} {\underline{$\langle y+2x , y - 2x \rangle < 0$}.} 
Then $2y = (y + 2x) +  (y - 2x) \in R$. 
Denote by $R'$ the subsystem of $R$ obtained by intersecting $R$ with the lattice generated by $x$ and $2y$.
We have two possibilities for $R'$:

\begin{enumerate}
\item[(i)] {\bf $R'$ is reducible.} Since $R'$ is reducible, $2y$ is the only root of $R$ at level 2. Thus $R$ contains the roots of the GRS 3(ii).
To prove that $R$ coincides with the GRS 3(ii), we first show that $3y \not \in R$ and  $y \pm 3x \not \in R$.
First, if $3y \in R$, Corollary \ref{grs-prop} implies that $3y$ is not the only root at level 3 and thus $3y \pm x \in R$. But then $\langle 3y+x, y  \rangle  > 0$
would imply that $2y + x \in R$ which is false. Second, if we assume that $y + 3x \in R$, the fact that none of the vectors
$(y + 3x) \pm (y-x)$ and $(y+3x) \pm (y-2x)$ belongs to $R$ would imply that $y+3x \perp y-x$ and $y+3x \perp y - 2x$ which is impossible, 
proving that $y + 3x \not \in R$. Similarly, $y - 3x \not \in R$. This shows that 
\[R = \{0, \pm x, \pm 2x, \pm 3x, \pm y , \pm y \pm x, \pm  y \pm 2x, \pm 2y\} \ .\]
To prove that $R$ is isomorphic to a GRS 3(ii), it remains to establish that  $\sqrt{2} \|x\| < \|y\| < 2\|x\|$. The first inequality follows from the fact that
$(y+2x) - (y-x) = 3x \in R$ while $(y+2x) + (y-x) = 2y + x \not \in R$ and hence $\langle y+2x, y-x  \rangle  > 0$. The second inequality is true by assumption.

\item[]

\item[(ii)] {\bf $R'$ is irreducible.} Since $R'$ is irreducible, $2y$ is not the only root of $R$ at level 2 and thus $2y \pm x, 2y \pm 2x \in R$.
Note that $4y \not \in R$ because otherwise $y$ would have multiplier 4, which is impossible. 
Then $(2y + 2x) \pm (2y - 2x) \not \in R$ implies that $2y + 2x \perp 2y - 2x$, i.e., $\|y\| = \|x \|$. Moreover,
\[\langle 2y+2x, y - 2x \rangle = 2 \|y\|^2 - 4 \|x\|^2 <0 \]
shows that $3y =  (2y+2x) +  (y - 2x) \in R$. So far we have shown that $R$ contains the roots of the GRS 3(iv). To prove that
$R$ coincides with the GRS 3(iv), it suffices to show that $4y, 3y \pm x, 2y \pm 3x, y \pm 3x \not \in R$ and then apply Corollary \ref{grs-prop}.
First, as noted above, $4y \not \in R$ since otherwise $y$ would have multiplier 4. Second, if $R$ contains roots other than $3y$ at level 3,
it would contain both $3y + 2x$ and $3y-2x$. However,
\[\langle 3y+2x, 3y - 2x \rangle = 9 \|y\|^2 - 4 \|x\|^2 > 0 \]
would imply $4x \in R$ which is false. Third, if we assume that $2y+3x \in R$,
\[\langle 2y+3x, y - x \rangle = 2 \|y\|^2 - 3 \|x\|^2 <0 \]
would imply $3y + 2x \in R$ which we showed is impossible. Finally, the assumption $y + 3x \in R$ and
\[\langle y+3x, 2y - x \rangle = 2 \|y\|^2 - 3 \|x\|^2 <0 \]
lead again to $3y + 2x \in R$. This completes the proof that in this case $R$ is the GRS 3(iv).
\end{enumerate}

\np
\ppoint{\bf{There is no root in $R$ at level 1 perpendicular to $x$.}} \label{sec7.1.2}
In analogy to the similar case in the proof of 
Proposition \ref{prop7.35} we can make the following observation about roots at level $k \geq 1$ in $R$:

\begin{enumerate}
\item[--] If there is a single root at level $k$, then it is orthogonal to $x$;
\item[--] If there are more than one roots at level $k$, then there are at least 5 roots at level $k$;
\item[--] If there are exactly 5 roots at level $k$, then one of them is orthogonal to $x$;
\item[--] If there are at least 6 roots at level $k$, then there is at least one root at level $2k$.
\end{enumerate}

\np 
Consider the roots in $R$ at levels $1, 2, 4, 8, \ldots$, etc. If none of them is orthogonal to $x$, then each of these 
levels contains at least 6 roots, which is impossible since $R$ is a finite set. Assume than $l \geq 1$ is the smallest integer 
such that there is a root orthogonal to $x$ at level $2^l$. Denote this root by $2y$ and consider the subsystem $R'$ of $R$
obtained by intersecting $R$ with the lattice generated by $x$ and $2y$. We consider two cases for $R'$:

\np 
{\bf A.} {\underline{$R'$ is reducible}.} Then $R'$ (and hence $R$) contains the roots $\pm x, \pm 2x, \pm 3x, 2y$ and does not contain the 
roots $2y \pm x$. In particular, the only root of $R$ at level $2^l$ is $2y$.  
The level $2^{l-1}$ roots in $R$ are exactly the roots in the $x$-string through $\beta''$ for some $\beta'' \in R$ with $\beta'' - x \not \in R$.
Note that $\{x, \beta''\}$ is a basis of the subsystem $R''$ obtained by intersecting $R$ with the lattice  generated by $x$ and $\beta''$.
Moreover $R'$ and $R''$ consists of the roots of $R$ at levels divisible by $2^l$ and $2^{l-1}$ respectively. In particular, $R' \subset R''$ and $2y \in R''$, which implies that
$2y = 2 \beta'' + s x$ for some integer $s$. Thus $\beta'' = y - \frac{s}{2} x$. Since $y \not \in R$, we conclude that $s$ is odd, which in turn implies that 
$y \pm \frac{1}{2} x, y \pm \frac{3}{2} x, y \pm \frac{5}{2} x$  are all the roots in $R$ at level $2^{l-1}$ and 
$\langle y - \frac{3}{2} x, y + \frac{5}{2} x \rangle = 0$, i.e., that $\|y\| = \frac{\sqrt{15}}{2} \|x \|$. So far we have showed that 
$R$ contains the roots of the GRS 3(iii). Next we show that $R$ contains no other roots. 

\np 
First we show that $l = 1$. Indeed, if $l>1$, then $R$ contains at least 6 roots at level $2^{l-2}$ because $R$ contains no roots orthogonal to $x$.
Corollary \ref{grs-prop} implies that there are roots at level $3 \times 2^{l-2}$. However, $3y \not \in R$ because otherwise $y = 3y - 2y \in R$ is a root at level $2^{l-2}$ orthogonal to $x$.
In short, at level $3 \times 2^{l-2}$ there are roots in $R$ but none of them is orthogonal to $x$, which implies that there are at least 6 roots at this level.
A simple argument, similar to the one used to show that  $y \pm \frac{1}{2} x, y \pm \frac{3}{2} x, y \pm \frac{5}{2} x$ belong to $R$, 
shows that $R$ contains either 
\[\begin{array}{lcl}
{\text{at level }} 2^{l-2} & : &\frac{1}{2} y - \frac{11}{4} x, \frac{1}{2} y - \frac{7}{4} x, \frac{1}{2} y - \frac{3}{4} x, \frac{1}{2} y + \frac{1}{4} x, \frac{1}{2} y + \frac{5}{4} x, \frac{1}{2} y + \frac{9}{4} x, \\
&&\\
{\text{at level }} 3 \times 2^{l-2} &:  & \frac{3}{2} y - \frac{9}{4} x, \frac{3}{2} y - \frac{5}{4} x, \frac{3}{2} y - \frac{1}{4} x, \frac{3}{2} y + \frac{3}{4} x, \frac{3}{2} y + \frac{7}{4} x, \frac{3}{2} y + \frac{11}{4} x \end{array} \]
or
\[\begin{array}{lcl}
{\text{at level }} 2^{l-2}&:  & \frac{1}{2} y - \frac{9}{4} x, \frac{1}{2} y - \frac{5}{4} x, \frac{1}{2} y - \frac{1}{4} x, \frac{1}{2} y + \frac{3}{4} x, \frac{1}{2} y + \frac{7}{4} x, \frac{1}{2} y + \frac{11}{4} x, \\
&&\\
{\text{at level }} 3 \times 2^{l-2} & : &  \frac{3}{2} y - \frac{11}{4} x, \frac{3}{2} y - \frac{7}{4} x, \frac{3}{2} y - \frac{3}{4} x, \frac{3}{2} y + \frac{1}{4} x, \frac{3}{2} y + \frac{5}{4} x, \frac{3}{2} y + \frac{9}{4} x  \ . \end{array}\]

\np
In the former case 
\[\langle \frac{1}{2} y - \frac{11}{4} x, \frac{3}{2} y + \frac{7}{4} x \rangle = \frac{3}{4} \|y\|^2 - \frac{77}{16} \|x\|^2 =  \frac{3}{4} \frac{15}{4} \|x\|^2 - \frac{77}{16} \|x\|^2 = - 8 \|x\|^2 <0 \]
implies that $2y - x = ( \frac{1}{2} y - \frac{11}{4} x) + (\frac{3}{2} y + \frac{7}{4} x) \in R$, contradicting the assumption that $2y$ is the only root at level $2^l$. Analogously, in the later case we
would obtain $2y + x = ( \frac{1}{2} y + \frac{11}{4} x) + (\frac{3}{2} y - \frac{7}{4} x) \in R$. This completes the proof that $l = 1$.

\np
To complete the proof that $R$ is the GRS 3(iii), it remains to show that $R$ contains no roots at level 3. Indeed, if $R$ contains roots at level 3, it would contain at least 6 roots at level 3,
namely $3y\pm \frac{1}{2} x, 3y \pm \frac{3}{2} x, 3y \pm \frac{5}{2} x$. But then $\langle 3y + \frac{5}{2} x, y + \frac{3}{2} x \rangle >0$ would imply that
$2y + x = (3y + \frac{5}{2} x) -  (y + \frac{3}{2} x) \in R$ which is false. This completes the proof that if $R'$ is reducible, then $R$ is the GRS 3(iii).

\np 
{\bf B.} {\underline{$R'$ is irreducible}.}  Section \ref{sec7.1.1} implies that $R'$ is one of the root systems 3(i), 3(ii), or 3(iv) (with $2y$ in place of $y$). We will show that 
none of these options is possible.

\begin{enumerate}
\item[(i)] {\bf $R'$ is 3(i) or 3(ii).} Then $R$ contains the roots $\pm x, \pm 2x, \pm 3x, 2y, 2y \pm x, 2y \pm 2x$ and $\|y\| \leq \|x\|$. Since $R$ contains no roots orthogonal to $x$ at level $2^{l-1}$,
$R$ must contain the vectors $y \pm \frac{1}{2} x, y \pm \frac{3}{2} x, y \pm \frac{5}{2} x$.  The inequalities $\langle 2y - 2x, y + \frac{5}{2} x\rangle = 2 \|y\|^2 - 5 \|x\|^2 <0$
and $\langle 2y - 2x, y + \frac{3}{2} x\rangle = 2 \|y\|^2 - 3 \|x\|^2 <0$ imply that $3y -\frac{1}{2} x$ and $3y + \frac{1}{2} x$ belong to $R$ (and thus $3y \not \in R$). As a consequence, $R$ contains 
at least 6 roots at level $3 \times 2^{l-1}$. In particular, $R$ contains the roots  $3y \pm \frac{1}{2} x, 3y \pm \frac{3}{2} x, 3y \pm \frac{5}{2} x$. Since 
$\langle 3y - \frac{5}{2}x, 3y + \frac{5}{2} x\rangle = 9 \|y\|^2 - \frac{25}{4} \|x\|^2  > 0$, we conclude that $5x = (3y + \frac{5}{2} x) - (3y - \frac{5}{2} x) \in R$ which is absurd.
This proves that $R'$ cannot be the GRS 3(i) or 3(ii).
	
\item[]
		
\item[(ii)] {\bf $R'$ is 3(iv).} In this case $R'$ (and thus $R$) contains the vectors $2y+x$ and $4y + 2x$. One proves as above that $y + \frac{1}{2} x$ and $3y + \frac{3}{2} x$ belong to $R$,
which is impossible since the root $y + \frac{1}{2} x$ would have multiplier 4. 		
\end{enumerate}

\np
\point{\bf{GRSs with largest multiplier $m = 2$.}}

\np
We have $x, 2x \in R$ but $3x \not \in R$.
We consider two cases, depending on whether there is a root $y$ orthogonal to $x$ at level 1.

\np
\ppoint{\bf{There is a root $y \perp x$ at level 1.}} In this case $\beta = y - s x$ for some $s>0$ and hence $R$ is contained in the lattice generated by $y$ and $x$. 
Thus thus every level of roots contains a root orthogonal to $x$. Since $m = 2$, $3y \not \in R$
which implies that $R$ has no roots at levels 3 and above. Moreover, $\langle y, y \pm 3x \rangle > 0$ implies that $y \pm 3x \not \in R$ since otherwise we would have $\pm 3x = (y \pm 3x) - y \in R$.
Finally, we note that $y \perp2 x$ implies that $y + 2 x \in R$ if and only if $y - 2 x \in R$. 
After these preliminary remarks, we consider two cases for $R$ depending on whether $y \pm 2x \in R$.

\np
{\bf A.} {\underline{$y \pm 2x \not \in R$}.}  Consider the following alternatives for $R$.

\begin{enumerate}
\item[(i)] {\bf $2y \not \in R$.} In this case Corollary \ref{grs-prop} implies that $R$ consists of the roots of the GRS 2(i). The inequality $\|y\| > \|x\|$ follows from the fact that $2x = (y + x) - (y - x) \in R$
but $2y = (y + x) + (y - x) \not \in R$ and hence $\langle y+x, y-x \rangle >0$.

\item[]

\item[(ii)] {\bf $2y \in R$ but $2y \pm x \not \in R$.} Since $R$ contains no roots at level 3, we conclude that $R$ is isomorphic to a GRS 2(ii).

\item[]

\item[(iii)] {\bf $2y, 2y \pm x \in R$.} Since neither $(2y-x) + (y+x) = 3y \in R$ nor $(2y-x) - (y+x) = y - 2x \in R$, we conclude that $\langle 2y-x, y+x \rangle = 0$ and thus $\|x\| = \sqrt{2} \|y\|$.
Then $\langle 2y-2x, y+x \rangle = 2 \|y\|^2 - 2 \|x\|^2 < 0$ implies that $2y - 2x \not \in R$ since otherwise $3y - x = (2y - 2x) + (y + x)$ would be a root at level 3. Similarly, $2y+2x \not \in R$. This proves that
$R = \{0, \pm x, \pm 2x, \pm y, \pm y \pm x, \pm 2y, \pm 2y \pm x\}$ with $\|x\| = \sqrt{2} \|y\|$, which is the GRS 2(v) with $x$ and $y$ switched.
\end{enumerate}

\np
{\bf{B.}} {\underline{$y \pm 2x  \in R$}.}  Since $3x = (y + 2x) - (y - x) \not \in R$, we have $\langle y+2x, y-x \rangle \leq 0$. Hence $\langle y+2x, y-2x \rangle < 0$
and thus $2y = (y + 2x) + (y-2x) \in R$. Consider the following alternatives for $R$.

\begin{enumerate}
\item[(i)] {\bf $2y \pm x \not \in R$.} Since neither $(y+2x) -(y-x) = 3x$ nor $(y+2x) + (y-x) = 2y +x$ is a root, we have $\langle y+2x, y-x \rangle = 0$, i.e., $\|y\| = \sqrt{2} \|x\|$. 
This proves that $R$ is the GRS 2(v).

\item[]

\item[(ii)] {\bf $2y \pm x \in R$.} Since neither $(2y-x) - (y+2x) = y - 3x$ nor $(2y-x) + (y+2x) = 3y + x$ is a root, we have $\langle 2y-x, y+2x \rangle = 0$, i.e., $\|y\| = \|x\|$. 
Then $\langle 2y-2x, y+2x \rangle < 0$, which implies that $2y-2x \not \in R$ because otherwise $3y = (2y-2x) +(y+2x) \in R$. This proves that $R$ is the GRS 2(vii).
\end{enumerate}

\np
\ppoint{\bf{There is no root in $R$  at level 1 perpendicular to $x$.}} We consider the following cases depending on the roots of $R$ at level 2.

\np
{\bf A.} {\underline{$R$ contains no roots at level 2}.} Observe first that there are exactly 4 roots at level 1, say $y\pm \frac{1}{2} x, y \pm \frac{3}{2}x$. Thus 
$R = \{0, \pm x , \pm 2x, \pm y \pm \frac{1}{2} x, \pm y \pm \frac{3}{2}x\}$. 

\np
Since neither $(y+\frac{3}{2} x ) + (y - \frac{3}{2} x) = 2y$ nor $(y+\frac{3}{2} x ) - (y - \frac{3}{2} x) = 3x$ belongs to $R$, we conclude that 
$\langle y+\frac{3}{2} x , y - \frac{3}{2} x \rangle = 0$, i.e., $\|y\| = \frac{3}{2} \|x\|$.
Moreover,
$\langle y \pm \frac{1}{2} x, x \rangle \lessgtr 0$ because otherwise there would be more than 4 roots at level 1. The inequalities $\langle y \pm \frac{1}{2} x, x \rangle \lessgtr 0$
are equivalent to $-\frac{1}{2} \|x\|^2 < \langle x, y \rangle < \frac{1}{2} \|x\|^2$. This proves that $R$ is isomorphic to a GRS 2(iii).

\np
{\bf B.} {\underline{There is exactly one root at level 2 of $R$}.} Denote this root by $2y$. Then $2y \perp x$. As above, $R$ contains the vectors $y \pm \frac{1}{2} x, y \pm \frac{3}{2} x$.
Assuming that $y + \frac{5}{2} x \in R$ we would arrive at $2y + x = (y-\frac{3}{2}x) + (y + \frac{5}{2} x) \in R$ which is false. Hence $y + \frac{5}{2} x \not \in R$ and, similarly, $y  - \frac{5}{2} x \not \in R$.
Note also that there are no roots of $R$ at level 3. Indeed, otherwise $R$ would contain $3y \pm \frac{1}{2} x, 3y \pm \frac{3}{2} x$. This is impossible since, for  $z:= y + \frac{1}{2} x$, the $z$-string through 0
contains the roots $z$ and $3z$ but does not contain $2z$, contradicting Corollary \ref{prop2.12}.
Finally, noting that $3x = (y + \frac{3}{2} x) - (y - \frac{3}{2} x) \not \in R$ while $2y= (y + \frac{3}{2} x) + (y - \frac{3}{2} x) \in R$ and 
$2x = (y + \frac{3}{2} x) - (y - \frac{1}{2} x)  \in R$ while $2y + x = (y + \frac{3}{2} x) + (y - \frac{1}{2} x) \not \in R$ we conclude that
\[\langle y + \frac{3}{2} x, y - \frac{3}{2} x \rangle < 0 \quad \quad {\text{and}} \quad \quad \langle y + \frac{3}{2} x, y - \frac{1}{2} x \rangle > 0 \ .\]
These inequalities are equivalent to $\frac{\sqrt{3}}{2} \|x\| < \|y \| < \frac{3}{2} \|x\|$. This completes the proof that $R$ is the GRS 2(iv).

\np
{\bf C.} {\underline{Level 2 of $R$ contains both a root orthogonal to $x$ and roots not orthogonal to $x$}.} Denote the root at level 2 orthogonal to $x$ by $2y$. 
Then $R$ contains the roots $\pm x, \pm 2x, y \pm \frac{1}{2} x, y \pm \frac{3}{2} x, 2y, 2y \pm x$. 
Let $R'$ be the subsystem of $R$ 
obtained by intersecting $R$ with the lattice generated by $x$ and $2y$. By the results of Section \ref{sec7.1.1}, $R'$ is isomorphic to a GRS 2(i), 2(ii), 2(v), or 2(vii).

\np
If $R'$ is isomorphic to one of the GRSs 2(ii), 2(v), or 2(vii), then $4y \in R$. Since $\langle 4y, y - \frac{3}{2} x \rangle >0$, we conclude that $3y + \frac{3}{2} x \in R$ which is a contradiction with the 
assumption that $m=2$ because $y + \frac{1}{2} x, 2y + x, 3y + \frac{3}{2} x$ all belong to $R$.

\np 
It remains to study the case when $R'$ is a GRS 2(i). We consider two cases depending on whether $y + \frac{5}{2} x \in R$.

\begin{enumerate}
\item[(i)] {\bf {$y + \frac{5}{2} x \not \in R$}.}  Then $y - \frac{5}{2} x \not \in R$ since otherwise $\langle 2y, y - \frac{5}{2} x \rangle >0$ would imply that $y + \frac{5}{2} x= 2y - (y - \frac{5}{2} x)  \in R$.
Moreover, $R$ contains no roots at level 3. Indeed, otherwise $R$ would contain the roots $3y \pm \frac{1}{2}x, 3y \pm \frac{3}{2} x$ which, as above, is a contradiction with the assumption $m=2$ 
because $y + \frac{1}{2} x, 2y + x, 3y + \frac{3}{2} x$ all belong to $R$. 
Noting that neither $(2y + x) + (y - \frac{3}{2} x) = 3y - \frac{1}{2} x$ nor $(2y + x) - (y - \frac{3}{2} x) = y + \frac{5}{2} x$ belongs to $R$, we conclude that $2y + x \perp y - \frac{3}{2} x$. This implies that
$\|y\| = \frac{\sqrt{3}}{2} \|x\|$, proving that $R$ is the GRS 2(vi). 

\item[]

\item[(ii)] {\bf {$y + \frac{5}{2} x \in R$}.} Note first that $y \pm \frac{7}{2} x \not \in \R$ since otherwise $\langle y \pm \frac{7}{2} x, y \pm \frac{1}{2} x \rangle >0$ would imply that $\pm 3x = (y \pm \frac{7}{2} x) -
(y \pm \frac{1}{2} x) \in R$. Moreover, as above, $R$ contains no roots at level 3 since otherwise $y + \frac{1}{2} x$ would have multiplier 3. Finally, the fact that 
neither $(y + \frac{5}{2}x) + (y - \frac{1}{2} x) = 2y +2 x$ nor $(y + \frac{5}{2}x) - (y - \frac{1}{2} x) = 3 x$ belongs to $R$, we conclude that $y + \frac{5}{2} x \perp y - \frac{1}{2} x$. This implies that
$\|y\| = \frac{\sqrt{5}}{2} \|x\|$, proving that $R$ is the GRS 2(viii). 

\end{enumerate}

\np
{\bf D.} {\underline{$R$ contains no roots orthogonal to $x$ at levels 1 and 2}.} There are at least 4 roots at level 1 of $R$. Denote by $z$ the root at level 1 such that $\langle z, x \rangle <0$ and
$\langle z + x, x \rangle > 0$. Then $R$ contains the roots $z-x, z, z+x, z+2x, 2z, 2z+x, 2z + 2x$ as well as $2z+3x$ or $2z - x$ depending on whether $\langle 2z + x, x \rangle <0$ or $\langle 2z + x, x \rangle >0$.
In what follows we will assume that $\langle 2z + x, x \rangle <0$ and hence $2z+3x \in R$. The other case is dealt with in a similar way.

\np
The assumption that $m = 2$ implies that neither $3z$ nor $3z + 3x$ is a root. Hence, $R$ has at most 2 roots at level 3, which implies that either there are no root at level 3 in $R$ or there is
exactly one root $3y$ at level 3 of $R$ and $3y \perp x$. 

\begin{enumerate}
\item[(i)] {\bf {There are no roots at level 3 of $R$}.} We make the following observations:

\begin{enumerate}
\item[]
\item[1.] $z+4x \not \in R$ because $\langle z + 4x, z+x \rangle = \langle z+x, z+x \rangle + 3 \langle x, z+x \rangle > 0$ but $(z + 4x) - (z+x) = 3x \not \in R$;
\item[]
\item[2.] $\langle 2z+ 3x, z - x \rangle = 0$ because neither $3z - 2x = (2z+3x) + (z-x)$  nor $z+ 4x = (2z+3x) - (z-x)$ belongs to $R$;
\item[]
\item[3.] $z + 3x \in R$ since $\langle 2z + 2x, z-x \rangle = \langle 2z+3x, z-x \rangle - \langle x, z-x \rangle > 0$;
\item[]
\item[4.] $\langle z+ 3x, z \rangle = 0$ because $(z+3x) + 2z = 3z + 3x \not \in R$ implies $\langle z+ 3x, 2z \rangle \geq 0$ and 
$(z+3x) - z =  3x \not \in R$ implies $\langle z+ 3x, z \rangle \leq 0$;
\item[]
\item[5.] $z - 2x \in R$ because $\langle z + 3x, 2z+x \rangle = \langle z+3x, z \rangle + \langle z+3x, z+x \rangle > 0$;
\item[]
\item[6.] $\langle z - 2x, z +x \rangle = 0$ because $(z-2x) + (2z +2x) = 3z \not \in R$ implies $\langle z - 2x, 2z + 2x  \rangle \geq 0$ and 
$(z-2x) - (z+x) =  -3x \not \in R$ implies $\langle  z-2x, z+ x \rangle \leq 0$.
\end{enumerate}

\np
The equations 
\[ \langle 2z+ 3x, z - x \rangle = 0, \quad \quad \langle z+ 3x, z \rangle = 0, \quad  \quad \langle z - 2x, z +x \rangle = 0 \]
from 2., 4.,  and 6. above imply that $\langle x, x \rangle = \langle x, z \rangle = \langle z,z \rangle = 0$.
This contradiction proves that there are no GRSs in this case.

\item[]

\item[(ii)] {\bf {There is a unique root $3y$ at level 3 of $R$}.}  A simple argument, similar to the one used in Section \ref{sec7.1.2}.A, 
shows that $R$ contains either 
\[ y - \frac{5}{3} x,  y - \frac{2}{3} x,  y +\frac{1}{3} x,  y + \frac{4}{3} x, 2 y - \frac{4}{3} x, 2 y - \frac{1}{3} x, 2 y + \frac{2}{3} x, 2 y + \frac{5}{3} x\]
or
\[ y - \frac{4}{3} x,  y - \frac{1}{3} x,  y +\frac{2}{3} x,  y + \frac{5}{3} x, 2 y - \frac{5}{3} x, 2 y - \frac{2}{3} x, 2 y + \frac{1}{3} x, 2 y + \frac{4}{3} x \ .\]
Since the latter case is obtained from the former by replacing $x$ by $-x$, it suffices to consider only the former one.

\np
First we show that $y + \frac{7}{3} x \not \in R$. Indeed, if we assume that $y + \frac{7}{3} x  \in R$, the fact that $(y + \frac{7}{3} x) + (y - \frac{2}{3} x ) = 2y + \frac{5}{3} x \in R$
and $(y + \frac{7}{3} x) - (y - \frac{2}{3} x ) = 3x \not  \in R$ implies that $\langle y + \frac{7}{3} x, y - \frac{2}{3} x \rangle <0$. This in turn implies that $3y + x = (y + \frac{7}{3} x) + (2y - \frac{4}{3} x ) \in R$
which is false.

\np
Similarly, $y-\frac{8}{3} x \not \in R$. Indeed, since neither $(2y + \frac{2}{3} x) + (y - \frac{5}{3} x) = 3y - x$ nor  $(2y + \frac{2}{3} x) - (y - \frac{5}{3} x) = y + \frac{7}{3} x$ belongs to $R$, we have
 $\langle 2 y + \frac{2}{3} x, y - \frac{5}{3} x \rangle =0$. Thus  $\langle 2 y + \frac{2}{3} x, y - \frac{8}{3} x \rangle < 0$ and the assumption $y-\frac{8}{3} x \in R$ would lead to the false conclusion that
 $3y - 2x = (y - \frac{8}{3} x) + (2 y + \frac{2}{3} x) \in R$.
 
 \np 
Next we note that there are no roots at level 4. Indeed, if there are roots at level 4, then $4y - \frac{5}{3} x,  4y - \frac{2}{3} x,  4y +\frac{1}{3} x,  4y + \frac{4}{3} x$ belong to $R$. 
This, however contradicts Corollary \ref{prop2.12} because $y + \frac{1}{3}x, 2y + \frac{2}{3} x, 4y + \frac{4}{3} x \in R$ but $3y + x \not \in R$.
As a consequence, neither $2y + \frac{8}{3} x$ nor $2y - \frac{7}{3} x$ belongs to $R$ since otherwise there will be at least 5 roots at level 2 and hence at least one root at level 4.

\np
Finally, since neither $4y + \frac{1}{3} x = (2y + \frac{5}{3} x) + (2y - \frac{4}{3} x)$ nor $3 x = (2y + \frac{5}{3} x) - (2y - \frac{4}{3} x)$
belongs to $R$, we conclude that $2y + \frac{5}{3} x \perp 2y - \frac{4}{3} x$, implying that $\|y\| = \frac{\sqrt{5}}{3} \|x\|$.

\np 
Consider the vectors $x':= y + \frac{1}{3} x$ and $y':= \frac{1}{2}y - \frac{5}{6} x$. It is a straightforward calculation that $R$ is the GRS 2(viii) when written in terms of the vectors $x'$ and $y'$.
\end{enumerate}

\np
This completes the proof of Theorem \ref{the7.25}.



\renewcommand{\arraystretch}{1.1}
\begin{table}[bp] 
\caption{\label{bigtable}Quotients of exceptional root systems}
\centering
\begin{NiceTabular}{|c|c|c|c|p{12cm}|c|}
\hline \hline
$k$ &$X_l$  & $N$  & name & \center{$J$ or $J^c$} & $|R|-1$\\
\hline \hline 
\Block{30-1}{2}&\Block{1-1}{$G_2$} & \Block{1-1}{1} & $\cG_{2,2}$ & 
\center{$G_2$}
& 12 \\ 
\cline{2-6}&\Block{4-1}{$F_4$} & \Block{4-1}{4} & $\cF_{4,2}^I$ & 
\center{12}
& 12 \\ 
\cline{4-6} &&& $\cF_{4,2}^{II}$ & 
\center{13, 23, 24}
& 16 \\ 
\cline{4-6} &&& $\cF_{4,2}^{III}$ & 
\center{14}
& 16 \\ 
\cline{4-6} &&& $\cF_{4,2}^{IV}$ & 
\center{34}
& 18 \\ 
\cline{2-6}
&\Block{5-1}{$E_6$} & \Block{5-1}{5} & $\cE_{6,2}^I$ & 
\center{12, 13, 26, 56}
& 8 \\ 
\cline{4-6} &&& $\cE_{6,2}^{II}$ & 
\center{14, 34, 35, 45, 46}
& 12 \\ 
\cline{4-6} &&& $\cE_{6,2}^{III}$ & 
\center{15, 23, 25, 36}
& 10 \\ 
\cline{4-6} &&& $\cE_{6,2}^{IV}$ & 
\center{16}
& 6 \\ 
\cline{4-6} &&& $\cE_{6,2}^{V}$ & 
\center{24}
& 12 \\ 
\cline{2-6}
&\Block{9-1}{$E_7$} & \Block{9-1}{9} & $\cE_{7,2}^I$ & 
\center{12, 27}
& 10 \\ 
\cline{4-6} &&& $\cE_{7,2}^{II}$ & 
\center{13}
& 12 \\ 
\cline{4-6} &&& $\cE_{7,2}^{III}$ & 
\center{14, 34, 36}
& 16 \\ 
\cline{4-6} &&& $\cE_{7,2}^{IV}$ & 
\center{15, 24, 25}
& 14 \\ 
\cline{4-6} &&& $\cE_{7,2}^{V}$ & 
\center{16}
& 12 \\ 
\cline{4-6} &&& $\cE_{7,2}^{VI}$ & 
\center{17, 67}
& 8 \\ 
\cline{4-6} &&& $\cE_{7,2}^{VII}$ & 
\center{23, 26, 37, 56, 57}
& 12 \\ 
\cline{4-6} &&& $\cE_{7,2}^{VIII}$ & 
\center{35, 45, 47}
& 16 \\ 
\cline{4-6} &&& $\cE_{7,2}^{IX}$ & 
\center{46}
& 18 \\ 
\cline{2-6}
&\Block{11-1}{$E_8$} & \Block{11-1}{11} & $\cE_{8,2}^I$ & 
\center{12, 13, 28}
& 14 \\ 
\cline{4-6} &&& $\cE_{8,2}^{II}$ & 
\center{14, 34, 37, 57}
& 22 \\ 
\cline{4-6} &&& $\cE_{8,2}^{III}$ & 
\center{15, 25}
& 20 \\ 
\cline{4-6} &&& $\cE_{8,2}^{IV}$ & 
\center{16}
& 18 \\ 
\cline{4-6} &&& $\cE_{8,2}^{V}$ & 
\center{17, 67, 68}
& 16 \\ 
\cline{4-6} &&& $\cE_{8,2}^{VI}$ & 
\center{18}
& 12 \\ 
\cline{4-6} &&& $\cE_{8,2}^{VII}$ & 
\center{23, 27, 38}
& 18 \\ 
\cline{4-6} &&& $\cE_{8,2}^{VIII}$ & 
\center{24, 26, 56, 58}
& 20 \\ 
\cline{4-6} &&& $\cE_{8,2}^{IX}$ & 
\center{35, 36, 45, 48}
& 24 \\ 
\cline{4-6} &&& $\cE_{8,2}^{X}$ & 
\center{46, 47}
& 28 \\ 
\cline{4-6} &&& $\cE_{8,2}^{XI}$ & 
\center{78}
& 12 \\ 
\hline \hline
\end{NiceTabular}
\end{table}

\begin{table}[htbp]
\setcounter{table}{0}
\caption{Quotients of exceptional root systems. (continued)}
\centering
\begin{NiceTabular}{|c|c|c|c|p{12cm}|c|}
\hline \hline
$k$ &$X_l$  & $N$  & name &\center{$J$ or $J^c$}  & $|R|-1$\\
\hline \hline 
\Block{20-1}{3} & \Block{2-1}{$F_4$} & \Block{2-1}{2} & $\cF_{4,3}^I$ &
\center{123, 124}& 26\\ 
\cline{4-6} && &$\cF_{4,3}^{II}$ & \center{134, 234}& 32 \\
\cline{2-6}
&\Block{3-1}{$E_6$} & \Block{3-1}{3} & 
$\cE_{6,3}^I$ & 
\center{123, 126, 136, 156, 256} & 16 \\ 
\cline{4-6}&&& $\cE_{6,3}^{II}$ &
\center{124, 125, 134, 135, 234, 236, 245, 246, 356, 456} &20\\
\cline{4-6}&&& $\cE_{6,3}^{III}$ & 
\center{145, 146, 235, 345, 346} & 22\\
\cline{2-6}
&\Block{7-1}{$E_7$} & \Block{7-1}{7} & $\cE_{7,3}^I$ &
\center{123, 127, 137, 267, 567} & 20\\
\cline{4-6} &&& $\cE_{7,3}^{II}$ &
\center{124, 126, 156, 157, 234, 237, 256, 257, 367}& 24 \\
\cline{4-6} &&& $\cE_{7,3}^{III}$ &
\center{125, 135, 245, 247}& 26 \\
\cline{4-6} &&& $\cE_{7,3}^{IV}$ &
\center{134, 136}& 26 \\
\cline{4-6} &&& $\cE_{7,3}^{V}$ &
\center{145, 147, 235, 236, 246, 345, 347, 356, 357, 456, 457, 467} & 28\\
\cline{4-6} &&& $\cE_{7,3}^{VI}$ &
\center{146, 346}& 32\\
\cline{4-6} &&& $\cE_{7,3}^{VII}$ & \center{167}& 18 \\
\cline{2-6}
&\Block{8-1}{$E_8$} & \Block{8-1}{8} & $\cE_{8,3}^I$ &
\center{123, 128, 138, 278} & 28\\
\cline{4-6} &&& $\cE_{8,3}^{II}$ &
\center{124, 127, 134, 137, 234, 238, 267, 268, 378, 567, 568, 578}& 34\\
\cline{4-6} &&& $\cE_{8,3}^{III}$ &
\center{125, 126, 135, 136, 156, 158, 245, 248, 256, 258}& 36\\
\cline{4-6} &&& $\cE_{8,3}^{IV}$ &
\center{145, 148, 157, 235, 237, 257, 345, 348, 367, 368}& 40\\
\cline{4-6} &&& $\cE_{8,3}^{V}$ &
\center{146, 147, 346, 347, 357, 457, 467, 468} &46\\
\cline{4-6} &&& $\cE_{8,3}^{VI}$ &
\center{167, 168}&32\\
\cline{4-6} &&& $\cE_{8,3}^{VII}$ &
\center{178, 678}&26\\
\cline{4-6} &&& $\cE_{8,3}^{VIII}$ &
\center{236, 246, 247, 356, 358, 456, 458, 478} &42\\
\hline \hline
\Block{13-1}{4} & $F_4$ & 1 & $\cF_{4,4}$ 
& \center{$F_4$} & 48\\ 
\cline{2-6}
&\Block{2-1}{$E_6$} & \Block{2-1}{2} & $\cE_{6,4}^I$ & 
\center{1234, 1236, 1256, 1356, 2456}
& 30 \\ 
\cline{4-6}&&& $\cE_{6,4}^{II}$ &
\center{1235, 1245, 1246, 1345, 1346, 1456, 2345, 2346, 2356, 3456}
&34\\
\cline{2-6}
&\Block{4-1}{$E_7$} & \Block{4-1}{4} & $\cE_{7,4}^I$ & 
\center{1234, 1237, 1267, 1367, 1567, 2567}
& 36 \\ 
\cline{4-6} &&& $\cE_{7,4}^{II}$ & 
\center{1235, 1236, 1245, 1247, 1256, 1257, 1345, 1347, 1356, 1357, 2345, 2347,
2367, 2456, 2457, 2467, 3567, 4567} &42\\
\cline{4-6} &&& $\cE_{7,4}^{III}$ & 
\center{1246, 1456, 1457, 1467, 2346, 2356, 2357, 3456, 3457, 3467} & 46 \\
\cline{4-6} &&& $\cE_{7,4}^{IV}$ & 
\center{1346} & 48\\
\cline{2-6}
&\Block{6-1}{$E_8$} & \Block{6-1}{6} & $\cE_{8,4}^I$ & 
\center{1234, 1238, 1278, 1378, 2678, 5678}
& 50 \\ 
\cline{4-6} &&& $\cE_{8,4}^{II}$ & 
\center{1235, 1237, 1245, 1248, 1267, 1268, 1345, 1348, 1367, 1368, 1567, 1568,
1578, 2345, 2348, 2378, 2567, 2568, 2578, 3678}
& 58 \\ 
\cline{4-6} &&& $\cE_{8,4}^{III}$ & 
\center{1236, 1256, 1258, 1356, 1358, 2456, 2458, 2478}
& 60 \\ 
\cline{4-6} &&& $\cE_{8,4}^{IV}$ & 
\center{1246, 1247, 1257, 1346, 1347, 1357, 1456, 1458, 1478, 2346, 2347, 2356,
2358, 2367, 2368, 2457, 2467, 2468, 3456, 3458, 3478, 3567, 3568, 3578, 
4567, 4568, 4578, 4678}
& 66 \\ 
\cline{4-6} &&& $\cE_{8,4}^{V}$ & 
\center{1457, 1467, 1468, 2357, 3457, 3467, 3468}
& 72 \\ 
\cline{4-6} &&& $\cE_{8,4}^{VI}$ &  \center{1678} & 48 \\
\hline \hline
\end{NiceTabular}
\end{table}

\begin{table}[htbp]
\setcounter{table}{0}
\caption{Quotients of exceptional root systems. (continued)}
\centering
\begin{NiceTabular}{|c|c|c|c|p{12cm}|c|}
\hline \hline
$k$ &$X_l$  & $N$  & name  & \center{$J$ or $J^c$} & $|R|-1$\\
\hline \hline 
\Block{6-1}{5}&\Block{1-1}{$E_6$} & \Block{1-1}{1} & $\cE_{6,5}$ & 
\center{all rank 5 quotients of $E_6$}
& 50 \\ 
\cline{2-6}
&\Block{2-1}{$E_7$} & \Block{2-1}{2} & $\cE_{7,5}^I$ & 
\center{67,56,45,34,24,12}
& 60 \\ 
\cline{4-6} &&& $\cE_{7,5}^{II}$ & 
\center{57,47,46,37,36,35,27,26,25,23,17,16,15,14,12}
& 66 \\ 
\cline{2-6}
&\Block{3-1}{$E_8$} & \Block{3-1}{3} & $\cE_{8,5}^I$ & 
\center{678,567,456,345,245,234,134} & 82 \\ 
\cline{4-6} &&& $\cE_{8,5}^{II}$ & 
\center{578, 568, 478, 467, 458, 457, 378, 367, 356, 348, 347, 346, 278,
        267, 256, 248, 247, 246, 178, 167, 156, 145, 138, 137, 136, 135, 
        124, 123}
& 92 \\ 
\cline{4-6} &&& $\cE_{8,5}^{III}$ & 
\center{468, 368, 358, 357, 268, 258, 257, 238, 237, 236, 235,
        168, 158, 157, 148, 147, 146, 128, 127, 126, 125}
& 100 \\ 
\hline \hline
\Block{4-1}{6}&\Block{1-1}{$E_6$} & \Block{1-1}{1} & $\cE_{6,6}$ & \center{$E_6$}
& 72 \\ 
\cline{2-6}
&\Block{1-1}{$E_7$} & \Block{1-1}{1} & $\cE_{7,6}$ & 
\center{all rank 6 quotients of $E_7$}
& 92 \\ 
\cline{2-6}
&\Block{2-1}{$E_8$} & \Block{2-1}{2} & $\cE_{8,6}^I$ & 
\center{78, 67, 56, 45, 34, 24, 13}
& 126 \\ 
\cline{4-6} &&& $\cE_{8,6}^{II}$ & 
\center{68, 58, 57, 48, 47, 46, 38, 37, 36, 35, 28, 27, 26, 25, 23,
        18, 17, 16, 15, 14, 12}
& 136 \\ 
\hline \hline
\Block{2-1}{7}&\Block{1-1}{$E_7$} & \Block{1-1}{1} & $\cE_{7,7}$ & \center{$E_7$}
& 126 \\ 
\cline{2-6}
&\Block{1-1}{$E_8$} & \Block{1-1}{1} & $\cE_{8,7}$ & 
\center{all rank 7 quotients of $E_8$}
& 182 \\ 
\hline \hline
\Block{1-1}{8}&\Block{1-1}{$E_8$} & \Block{1-1}{1} & $\cE_{8,8}$ & \center{$E_8$}
& 240 \\ 
\hline \hline
\end{NiceTabular}
\end{table}

\renewcommand{\arraystretch}{1}

\end{document}